\pdfoutput=1
\documentclass[3p]{elsarticle}
  \biboptions{sort&compress}
\makeatletter
\def\ps@pprintTitle{%
    \let\@oddhead\@empty
    \let\@evenhead\@empty
    \let\@oddfoot\@empty%
    \let\@evenfoot\@oddfoot}
\makeatother
\usepackage{cmap} 
\usepackage[british]{babel}
\usepackage[utf8]{inputenc}
\usepackage[T1]{fontenc}
\usepackage{amssymb,amsthm}
\usepackage{mathtools}
  \mathtoolsset{mathic}
  \numberwithin{equation}{section}
\usepackage{textcomp}
\usepackage{mathptmx}
\usepackage{MnSymbol,accents}
\usepackage[spacing]{microtype}
\usepackage[inline]{enumitem} 
  \setlist{nosep,listparindent=\parindent}
  \setlist[enumerate,1]{label=\upshape{(\roman*)}}
\usepackage{hyperref,bookmark,url}
  \bookmarksetup{  
    numbered
  }
  \hypersetup{  
    breaklinks=false,
    pdfborderstyle={/S/U/W .5},  
    citebordercolor=.235 .702 .443,
    urlbordercolor=.255 .412 .882,
    linkbordercolor=.804 .149 .149,
  }
  \hypersetup{ 
    pdfauthor={Erik Quaeghebeur, Gert de Cooman, and Filip Hermans},
    pdftitle={Accept \& Reject Statement-Based Uncertainty Models}, 
    pdfkeywords={statements, accept, reject, acceptability, indifference, desirability, favourability, preference, prevision} 
  }
\usepackage{fixltx2e} 
\usepackage{tikz}
  \usetikzlibrary{positioning,calc,fit,shapes.geometric,patterns}
  \pgfdeclarelayer{background}
  \pgfdeclarelayer{foreground}
  \pgfsetlayers{background,main,foreground}
  \tikzstyle{vec}=[circle,inner sep=1pt,outer sep=-1pt,fill]
  \tikzstyle{border}=[thick]
  \tikzstyle{favborder}=[border,dotted]
  \tikzstyle{exclborder}=[border,dashed]
\usepackage{pdfsync}
\usepackage{wrapfig}
\usepackage{enumitem}

\newcommand*{\defeq}{\coloneqq}
\newcommand*{\cset}[3][]{\set[#1]{#2:#3}}
\newcommand*{\naturals}{\mathbb{N}}
\newcommand*{\reals}{\mathbb{R}}
\newcommand*{\compl}[1]{{#1}^c}
\newcommand*{\powset}[1]{2^{#1}}
\newcommand*{\union}{\cup}
\newcommand*{\Union}{\bigcup}
\newcommand*{\intersection}{\cap}
\newcommand*{\Intersection}{\bigcap}
\newcommand*{\dedunion}{\uplus}
\newcommand*{\dedUnion}{\biguplus}
\newcommand*{\reckunion}{\sqcupplus}
\newcommand*{\reckUnion}{\bigsqcupplus}
\DeclarePairedDelimiter{\group}{(}{)}
\DeclarePairedDelimiter{\set}{\{}{\}}
\DeclarePairedDelimiter{\card}{\lvert}{\rvert}
\DeclareMathOperator{\phull}{posi}
\newcommand*{\shull}[1]{\overline{#1}}
\DeclareMathOperator{\lhull}{span}
\DeclareMathOperator{\close}{cl}
\newcommand{\cls}[1]{\close_{#1}}
\DeclareMathOperator{\extend}{ext}
\newcommand{\ext}[1]{\extend_{#1}}

\DeclareMathOperator{\average}{avg}
\newcommand*{\conj}{\wedge}

\newcommand*{\then}{\Rightarrow}
\newcommand*{\neht}{\Leftarrow}
\renewcommand*{\iff}{\Leftrightarrow}
\newcommand*{\after}{\circ}
\newcommand*{\pspace}{\varOmega}
\newcommand*{\someEs}{\mathcal{E}}
\newcommand*{\othersomeEs}{\mathcal{F}}
\newcommand*{\Gs}[2][]{\mathcal{G}_{#1}(#2)}
\newcommand*{\linGs}{\mathcal{L}}
\newcommand*{\someGs}{\mathcal{K}}
\newcommand*{\ray}[1]{\bar{#1}}
\newcommand*{\cones}{\mathrm{C}}

\newcommand*{\lineals}{\mathrm{L}}
\newcommand*{\geqrvf}[1]{\mathrel{\geq_{#1}}}
\newcommand*{\gtrrvf}[1]{\mathrel{>_{#1}}}
\newcommand*{\lessrvf}[1]{\mathrel{<_{#1}}}
\newcommand*{\eqrvf}[1]{\mathrel{=_{#1}}}
\DeclarePairedDelimiter{\Adelimhelper}{\langle}{\rangle}
\newcommand*{\Adelim}[3][]{\Adelimhelper[#1]{#2;#3}}
\newcommand*{\A}{\mathcal{A}} 
\newcommand*{\otherA}{\mathcal{B}}
\newcommand*{\As}{\mathbf{A}}
\newcommand*{\someAs}{\mathbf{B}}
\newcommand*{\somemoreAs}{\mathbf{C}}
\newcommand*{\ncAs}{\mathbb{A}}
\newcommand*{\dcAs}{\ncAs\mkern-5mu\vphantom{ncAs}^{\smash[b]{+}}}
\newcommand*{\somencAs}{\mathbb{B}}

\newcommand*{\maxsomeAs}{\hat{\someAs}}
\newcommand*{\maxsomemoreAs}{\hat{\somemoreAs}}
\newcommand*{\maxncAs}{\hat{\ncAs}}
\newcommand*{\maxdcAs}{\maxncAs\mkern-7mu\vphantom{ncAs}^{\smash[b]{+}}}
\newcommand*{\D}{\mathcal{D}}
\newcommand*{\otherD}{\mathcal{C}}
\newcommand*{\Ds}{\mathbf{D}}

\newcommand*{\ncDs}{\mathbb{D}}

\newcommand*{\maxncDs}{\hat{\ncDs}}
\newcommand*{\M}{\mathcal{M}}
\newcommand*{\otherM}{\mathcal{N}}
\newcommand*{\Ms}{\mathbf{M}}

\newcommand*{\ncMs}{\mathbb{M}}
\newcommand*{\somencMs}{\mathbb{K}}

\newcommand*{\maxncMs}{\hat{\ncMs}}
\newcommand*{\bgM}{\mathcal{S}}
\newcommand*{\dotbgM}{\Adelim{\linGs_\gtrdot}{\linGs_\lessdot}}
\newcommand*{\zeroM}{\mathcal{O}}

\newcommand*{\aff}[1]{\underline#1\vphantom{#1}}
\newcommand*{\affdcAs}{\aff\ncAs\mkern-7mu\vphantom{ncAs}^{\smash[b]{+}}}
\newcommand*{\fif}[1]{\undertilde#1\vphantom{#1}}
\newcommand*{\fifdcAs}{\fif\ncAs\mkern-7mu\vphantom{ncAs}^{\smash[b]{+}}}
\newcommand*{\fav}{\closedsucc}
\newcommand*{\nfav}{\not\fav}
\newcommand*{\rej}{\prec}
\newcommand*{\nrej}{\nprec}
\newcommand*{\acc}{\succeq}
\newcommand*{\nacc}{\nsucceq}
\newcommand*{\rejnacc}{{\rej,\nacc}}
\newcommand*{\accnrej}{{\acc,\nrej}}
\newcommand*{\indiff}{\simeq}
\newcommand*{\nindiff}{\nsimeq}
\newcommand*{\accnindiff}{{\acc,\nindiff}}
\newcommand*{\unres}{\smile}
\newcommand*{\confus}{{\acc,\rej}}
\newcommand*{\incomp}{\shortparallel}

\newcommand*{\srej}{\vartriangleleft}
\newcommand*{\sacc}{\trianglerighteq}

\newcommand*{\pr}{P}
\newcommand*{\otherpr}{Q}
\newcommand*{\prs}{\mathcal{P}}
\newcommand*{\lpr}{{\underline{\pr}}}
\newcommand*{\upr}{{\overline{\pr}}}
\newcommand*{\margs}[1]{\mathcal{G}_{#1}}
\newcommand*{\permuts}{\varPi}
\newcommand*{\transfos}{\mathcal{T}}
\newcommand*{\avg}[1]{\average_{#1}}
\newcommand*{\invars}{\mathcal{I}}
\newcommand{\demph}[1]{\textbf{\emph{#1}}}
\theoremstyle{plain}
\newtheoremstyle{theorem}
  {}
  {}
  {}
  {}
  {\itshape\bfseries}
  {.}
  {.5em}
  {}
\theoremstyle{theorem}
\newtheorem{inneraxiom}{Axiom}
\newenvironment{axiom}[2][]
  {\renewcommand\theinneraxiom{#2}\inneraxiom[#1]\quad}
  {\endinneraxiom}
\newtheorem{theorem}{Theorem}[section]
\newtheoremstyle{proposition}
  {}
  {}
  {}
  {}
  {\itshape}
  {.}
  {.5em}
  {}
\theoremstyle{proposition}
\newtheorem{innercondition}{Condition}
\newenvironment{condition}[2][]
  {\renewcommand\theinnercondition{#2}\innercondition[#1]\quad}
  {\endinnercondition}
\newtheorem{proposition}[theorem]{Proposition}
\newtheorem{corollary}[theorem]{Corollary}
\newtheorem{lemma}{Lemma}

\begin{document}
\begin{frontmatter}
  \title{Accept \& Reject Statement-Based Uncertainty Models} 
  \author[systems,cwi]{Erik Quaeghebeur\corref{cor}\fnref{thanks}}
  \ead{Erik.Quaeghebeur@cwi.nl}
  \cortext[cor]{Corresponding author}
  \fntext[thanks]{
    This work was first submitted while Erik Quaeghebeur was a visiting scholar at the Department of Information and Computing Sciences of Utrecht University.
    Revision of this work was carried out while Erik Quaeghebeur was an ERCIM “Alain Bensoussan” Fellow at the Centrum Wiskunde \& Informatica, a program receiving funding from the European Union Seventh Framework Programme (FP7/2007-2013) under grant agreement n° 246016.}
  \author[systems]{Gert de Cooman}
  \ead{Gert.deCooman@UGent.be}
  \author[systems]{Filip Hermans}
  \address[systems]{SYSTeMS Research Group, Ghent University, Technologiepark 914, 9052 Zwijnaarde, Belgium}
  \address[cwi]{Centrum Wiskunde \& Informatica, Postbus 94079, 1090~GB~~Amsterdam, The Netherlands}

  \begin{abstract}
    We develop a framework for modelling and reasoning with uncertainty based on accept and reject statements about gambles.
    It generalises the frameworks found in the literature based on statements of acceptability, desirability, or favourability and clarifies their relative position.
    Next to the statement-based formulation, we also provide a translation in terms of preference relations, discuss---as a bridge to existing frameworks---a number of simplified variants, and show the relationship with prevision-based uncertainty models.
    We furthermore provide an application to modelling symmetry judgements.
  \end{abstract}

\end{frontmatter}

\section{Introduction}\label{sec:intro}
\subsection{What \& why}
Probability theory and the statistical tools built on it help us deal with uncertainty, be it caused by the lack of information about or by the variability of some phenomenon.
Probability theory provides a mathematical framework for modelling uncertainty that is centred around the quantification of the uncertainty.
It also provides rules for reasoning under uncertainty, i.e., deductive inference: how to go from an assessment, such as probability values for some events, to conclusions and decisions, e.g., the expected value of some quantity or the selection of an optimal action.

In this paper, we present a mathematical framework for modelling uncertainty that generalises probability theory and that is centred around the categorisation of gambles---i.e., quantities about whose value we are uncertain---into acceptable and rejected ones.
Crudely speaking, gambles whose expectation is assessed to be non-negative are acceptable and those whose expectation is assessed to be negative are rejected.
Our mathematical framework includes rules for reasoning under uncertainty, but now the assessments we start from consist of sets of acceptable and rejected gambles.
We do not discuss how the assessments are obtained from domain experts or experimental data and restrict ourselves to the unconditional case.

Why did we develop this theory?
The main goal is to unify many of the existing generalisations of probability theory that are explicitly or implicitly based on the assumption that gamble values can be expressed in a linear precise utility.
Of these generalisations, some essentially express uncertainty using a \emph{strict} (partial) order of a set of gambles of interest, which is useful for decision making applications, and some use a \emph{non-strict} (partial) order, which is practical when expressing indifference judgements such as those arising when modelling symmetry assumptions.
Because of the unifying character of our theory, we can combine the strengths of each type of theory and also express the relative position of the models representable within each of the theories.
A consequence is of course that our theory is more expressive than the theories we unify.

We have already mentioned that the representation in our theory consists of a pair of gamble sets.
This type of representation differs markedly and is far less common than the typical, functional one, in which one works with functions that map events and gambles to probability and expectation values.
It is closely related to, but less popular than the preference order one, where, roughly speaking, events are ranked according to likelihood and gambles according to expectation.
Our experience with working with representations using sets of gambles has convinced us that it deserves more attention: the formulation of theory and derivation of results is mostly simplified to intuitive applications of basic set theory, linear algebra, and convex analysis.
A side goal of this paper is therefore to bring attention to this representation by putting it centre stage.

Although it is not the focus of this paper, it is useful to point out another advantage of the theory we develop:
Because it generalises theories using different representations of uncertainty, the elicitation statements that go along with these representations---probabilities for events, preferences between gambles, etc.---can all be incorporated.

\subsection{Literature context}
We have drawn much inspiration from the foundational works of the theories of imprecise probability \citep{Keynes-1921,Koopman-1940-ams,Good-1952,Smith-1961,Dempster-1967,Suppes-1974,Shafer-1976,Levi-1980,Walley-1991}, which all generalise probability theory in a similar way.
In these theories, one can specify lower and upper probabilities and expectations instead of just unique, precise probability and expectation values.
The introduction of these new concepts is justified by giving them a clear interpretation and by the fact that often not enough information is available to fix unique, precise values.
Our theory possesses the expressiveness to deal with imprecision, as these theories of imprecise probability do.

Probability measures, the models of classical probability theory, correspond to complete preference orders.
This means that all events or gambles can be compared by their unique probability or expectation value.
This is no longer the case in theories of imprecise probabilities, where events and gambles may be uncomparable and whose models may correspond to partial preference orders.
Sometimes strict preference is taken as basic \citep[see, e.g.,][]{Seidenfeld-Schervish-Kadane-1990-decwoord,Walley-1991,Walley-2000-towards} and sometimes non-strict preference is \citep[see, e.g.,][]{Williams-1974,Walley-1991}; the other relation may then be derived from it \citep[see, e.g.,][]{Fishburn-1986}.
We can associate both a strict and a non-strict preference order to each of the models in our theory.
But now neither is basic and each can essentially be specified separately, but both are related to each other in a natural way.
This is not the only way to generalise things; \citet{Seidenfeld-Schervish-Kadane-1995-preference,Seidenfeld-Schervish-Kadane-2010} use strict preference, choose their axioms in view of characterising the preference order in terms of sets of functionals and move to preference between sets of gambles in their coherent choice function approach.

In this paper, we represent uncertainty using sets of gambles.
Although this is as of yet uncommon in the literature, we are not the first.
\Citet[§14]{Smith-1961} uses them as a useful intermediate representation in the course of a proof in the exposition of his theory of imprecise probabilities.
In the work of \citet{Williams-1974,Williams-1975}, they become more prominent---he talks about sets of acceptable bets---, but he still keeps a focus on (non-linear) expectation-type models.
\Citet[§IV]{Seidenfeld-Schervish-Kadane-1990-decwoord} talk about ‘favorable’ gambles while presenting results about the relationship between convex sets of probability measures and strict preference orders.
It is \citet{Walley-1991,Walley-2000-towards}---talking about sets of desirable gambles---who seems to have been the first to discuss them in their own right.
Compared to these authors, our contribution lies in considering a pair of sets instead of just one set, proposing an intuitive axiom of how both sets interact, and developing the resulting theory.
Adding an extra set is what allows us to simultaneously represent strict and non-strict preferences.

In our framework, our first main axiom defines which assessments are irrational and which are not.
The two other main axioms are generative in the sense that we can associate an extension operator with each of them that modifies assessments in such a way that the result satisfies the axiom.
In spirit, this is completely analogous to the approach taken by \citet{DeFinetti-1937,DeFinetti-1974/1975} in his development of probability theory and \citet{Walley-1991} in his development of imprecise probability theory: they call their irrationality axiom `avoiding sure loss', their (single) generative principle `coherence', and the associated extension operator `the fundamental theorem of probability' and `natural extension', respectively.

\looseness-1 The basic setup and the terminology we use in our paper shows the influence of the subjectivist school of \citeauthor{DeFinetti-1974/1975}.
We do not want this to be construed as a constraint.
The results of this paper---its mathematical models and techniques---are applicable under different interpretations, analogously to the role measure theory takes in probability theory.
In this context, the `agent' that provides us with an assessment can, e.g., be seen as a person giving a `subjective' opinion or an algorithm that transforms observations such as sample sequences into `objective' opinions.
We do not discuss how these opinions are elicited or constructed.

\subsection{Overview}\label{sec:overview}
As stated, the prime goal of this paper is to present a mathematical framework for modelling uncertainty.
The basic conceptual set-up and mathematical notation is given in Section~\ref{sec:setup}.
Our presentation starts from the basics and there is no prerequisite knowledge of (imprecise) probability theory, although of course this does not hurt.
However, throughout, a basic familiarity with set theory, linear algebra, and convex analysis is assumed.
Once the framework is established in Section~\ref{sec:framework}, we do make the connection with preference relations in Section~\ref{sec:relations} and (imprecise) probability frameworks in Sections \ref{sec:simplified-frameworks} and~\ref{sec:previsions}.
Also here specific prerequisite knowledge is useful but not necessary.
Before concluding with some final remarks, musings and topics for further investigation in Section~\ref{sec:conclusions}, we present a small theoretical application of our framework in Section~\ref{sec:xch}, i.e., we show how it can be effectively used to model symmetry judgements.
To improve the readability of the paper, we have gathered the proofs in an appendix.

How do we build up our framework?
We start by giving a detailed description of the nature of the assessments we consider, the pairs of sets of accepted and rejected gambles (Section~\ref{sec:accept+reject}).
In three subsequent sections, we introduce the main rationality criteria of our framework: No Confusion, which determines the assessments that are irrational (Section~\ref{sec:no-confusion}); Deductive Closure, which makes explicit the consequences of assuming values are expressed in a linear precise utility scale (Section~\ref{sec:dedcls}); No Limbo, which states that gambles whose acceptance would lead to an irrational model must be rejected.
The latter two axioms are accompanied by operators that extend an assessment to a full fledged model that incorporates both the assessments and the requirements of these axioms.
At this point the conceptual groundwork is complete.

The central question then becomes: given an assessment, what are the properties of the model obtained by extending it, i.e., by performing deductive inference?
In particular: which assessments give rise to rational models?
To be able to provide general answers to these questions, we first perform an order-theoretic analysis of the sets of all assessments and models (Section~\ref{sec:order}), where assessments are ordered according to the resolve they encode, i.e., the amount of accepted and rejected gambles they contain.
Both general answers are given (in Section~\ref{sec:dominating-models}) and answers for the important specific case in which a background model---e.g., expressing that uniformly negative gambles should be rejected and uniformly positive ones should be accepted---is assumed (in Section~\ref{sec:background}).
This analysis of the framework ends with a characterisation of the important specific case with Indifference to Status Quo, i.e., in which the zero gamble is assumed to be accepted.

From this point onward the focus shifts away from our framework itself to its relationships with other frameworks.
This is worked out in most detail for its relationship to \citeauthor{DeFinetti-1974/1975}'s \cite{DeFinetti-1974/1975} theory of (coherent linear) previsions (in Section~\ref{sec:linprevs}) and \citeauthor{Walley-1991}'s \cite{Walley-1991} theory of coherent lower previsions (in Section~\ref{sec:lowprevs}).
This shows how---the in practice important---(imprecise) probabilistic models are specific instances of the models of our framework.
For readers interested in preference order frameworks, we give a translation of
the concepts of our framework and the main characterisation result in terms of gamble relations in Section~\ref{sec:relations}.
For particular applications, it may well be that a framework more general than (imprecise) probability theory is appropriate, but for which the full generality of our framework is nevertheless overkill.
Therefore, we present a number of simplified variants:
The first two of these still preserve the ability to express two separate preference relations (Sections~\ref{sec:accept-favour} and~\ref{sec:favindiff}).
The last two lose this ability and have been included to make correspondences with the frameworks based on sets of gambles that we have encountered in the literature (Section~\ref{sec:favourability} and~\ref{sec:acceptability}).

\subsection{Basic set-up and notation}\label{sec:setup}
We consider an agent faced with uncertainty, e.g., about the outcome of some experiment.
We assume it is possible to construct a \emph{possibility space}~$\pspace$ of mutually exclusive \emph{elementary events}, e.g., a set of different experimental outcomes, one of which is guaranteed to occur.
The possibility space~$\pspace$ can be infinite, but---apart from the Axiom of Choice---we do not consider anything specific for infinite-dimensional spaces such as their topology or the conglomerability of uncertainty models \citep[see, e.g.,][\S6.8]{Walley-1991}.
The paper is best read using intuition about finite-dimensional spaces.

Formally, a \demph{gamble} is a real-valued function on the possibility space; as suggested by its name, it represents a---positive or negative---pay-off that is uncertain in the sense that it depends on the unknown outcome.
These pay-offs are assumed to be expressed in units of a linear precise utility.
The set of all gambles $\Gs{\pspace}$, combined with point-wise  addition of gambles and point-wise multiplication with real numbers constitutes a real vector space.
We give the name \emph{status quo} to the origin of this space, the zero gamble~$0$.
We assume the agent is interested in a linear subspace of gambles $\linGs\subseteq\Gs{\pspace}$.

\begin{wrapfigure}[5]{r}{0pt}
  \begin{tikzpicture}[scale=2]
    \draw[->] (-.5,0) coordinate (xl) -- (1,0) coordinate (xu);
    \draw[->] (0,-.3) coordinate (yl) -- (0,.8) coordinate (yu);
    \coordinate[vec,label={above right:$f$}] (f) at (.5,.5);
    \coordinate[rectangle,inner sep=0pt,fill,minimum width=1pt,minimum height=4pt,label={below:$f(\omega)$}] (fx) at (f |- xu);
    \coordinate[rectangle,inner sep=0pt,fill,minimum width=4pt,minimum height=1pt,label={left:$f(\varpi)$}] (fy) at (f -| yu);
    \draw[dotted] (f) -- (fx) (f) -- (fy);
  \end{tikzpicture}
\end{wrapfigure}
We illustrate these concepts for a possibility space ${\pspace\defeq\set{\omega,\varpi}}$, and take the linear space of gambles to be $\linGs\defeq\Gs{\pspace}$, the two-dimensional plane.
A gamble $f$ is a vector with two components, $f(\omega)$ and $f(\varpi)$.
This example format will be used throughout the paper.
The examples should not give the impression that particular basis we choose for the gamble space is fundamentally relevant; however, our choice to have the positive orthant upper-right and the negative orthant lower-left connects with existing experience and intuition.

The following concepts and notation prove convenient.
First, those concerning operations on (sets of) gambles: let~$f$ be a gamble in~$\linGs$ and let $\someGs$ and $\someGs'$ be subsets of $\linGs$, then
\begin{enumerate}[label=(\alph*)]
  \item the \emph{complement} of $\someGs$ relative to $\linGs$ is $\compl{\someGs}\defeq\linGs\setminus\someGs$,
  \item the \emph{negation} of $\someGs$ is $-\someGs\defeq\cset{-g}{g\in\someGs}$,
  \item the \emph{ray} through~$f$ is
    $\ray{f}\defeq\cset{\lambda\cdot f}{\lambda\in\reals_{>0}}$,
  \item the \emph{positive scalar hull} of $\someGs$ is
    $\shull{\someGs}\defeq\Union_{f\in\someGs}\ray{f}$,
  \item the \emph{Minkowski sum} of $\someGs$ and $\someGs'$ is
    $\someGs+\someGs'\defeq\cset{g+h}{g\in\someGs \conj h\in\someGs'}$, so in particular $\someGs+\emptyset=\emptyset$ and, using negation, $\someGs-\someGs'=\someGs+(-\someGs')$,
  \item the \demph{positive linear hull} of $\someGs$ is
    $
      \phull\someGs \defeq \Union\cset[\big]{\sum_{g\in\someGs''}\ray{g}}{\someGs''\subseteq\someGs \conj \card{\someGs''}\in\naturals},
    $
    where $\sum_{g\in\someGs''}\ray{g}$ is a Minkowski sum of rays, and so the set of all convex cones in~$\linGs$ is $\cones\defeq\cset{\someGs\subseteq\linGs}{\phull\someGs=\someGs}$,
  \item the \demph{linear span} of $\someGs$ is $\lhull{\someGs}\defeq\phull(\someGs\union-\someGs\union\{0\})$, the smallest linear space including $\someGs$, and so the set of all linear subspaces of~$\linGs$ is $\lineals\defeq\cset{\someGs\subseteq\linGs}{\lhull\someGs=\someGs}$.
\end{enumerate}
Secondly, concepts and notation for the comparison of gambles, i.e., vector inequalities: let~$Q$ be a real-valued operator on~$\linGs$ and let~$f$ and~$g$ be gambles in $\linGs$, then
\begin{enumerate}
  \item $f\geqrvf{Q}g$ if and only if $Q\group{f-g}\geq0$,
  \item $f\gtrrvf{Q}g$ if and only if $Q\group{f-g}>0$,
  \item $f\eqrvf{Q}g$ if and only if $f\geqrvf{Q}g$ and $g\geqrvf{Q}f$,
  \item $f\geq g$ if and only if $f\geqrvf{\inf}g$,
  \item $f>g$ if and only if $f\geq g$ and $f\neq g$,
  \item $f\gtrdot g$ if and only if $f\gtrrvf{\inf}g$.
\end{enumerate}
With each of these gamble relations---let us denote them generically by $\mathrel{\Box}$---we can associate a specific subset of~$\linGs$, $\linGs_{\mathrel{\Box}}\defeq\cset{f\in\linGs}{f\mathrel{\Box}0}$.

\section{Accept \& Reject Statement-Based Uncertainty Models}\label{sec:framework}
In this section, we introduce the objects used to represent assessments and models for uncertainty, and the basic axioms that order them.
This gives rise to a number of important characterisations and provides us with a simple but powerful framework that leaves room for problem-specific a priori assumptions.

\subsection{Accepting \& Rejecting Gambles, and the Resulting Assessments}\label{sec:accept+reject}
We envisage an elicitation procedure where the agent is asked to state whether he accepts or rejects different gambles, or remains unresolved.
The \demph{acceptability} of a gamble~$f$ on~$\pspace$ implies a commitment to engage in the following transaction:
\begin{enumerate}
  \item the experiment's outcome $\omega\in\pspace$ is determined,
  \item the agent gets the---possibly negative---pay-off $f(\omega)$.
\end{enumerate}
By \demph{rejecting} a gamble, the agent expresses that he considers accepting that gamble unreasonable, or, in other words, deems the gamble unacceptable.
Such statements are, for example, relevant when combining his statements with those of another agent: when a gamble is accepted by one, but rejected by the other, there is a clear conflict.
The agent is not forced to state either acceptance or rejection for a given gamble, but may choose to remain unresolved, e.g., because of a lack of information about the experiment.
Note the difference in nature of these two statement types:
Acceptability is operationally linked to gamble pay-offs.
For rejection this link is only indirect, namely by referral to acceptability.

The set of gambles the agent finds acceptable is denoted by~$\A_\acc\subseteq\linGs$; the set of gambles he rejects is similarly denoted  by~$\A_\rej\subseteq\linGs$.
Combined, they form his \demph{assessment} $\A\defeq\Adelim{\A_\acc}{\A_\rej}$; the set of all assessments is $\As\defeq\powset{\linGs}\times\powset{\linGs}$.

\begin{wrapfigure}{r}{0pt}
  \begin{tikzpicture}[scale=1.5]
    \draw[->] (-1,0) coordinate (xl) -- (1,0) coordinate (xu);
    \draw[->] (0,-.7) coordinate (yl) -- (0,1) coordinate (yu);
    \node (a1) at (30:.7) {$\oplus$};
    \node (a2) at (60:.7) {$\oplus$};
    \node (r1) at (-30:.7) {$\ominus$};
    \node (r2) at (120:.7) {$\ominus$};
  \end{tikzpicture}
  \hspace{1em}
  \begin{tikzpicture}[scale=1.5]
    \draw[->] (-1,0) coordinate (xl) -- (1,0) coordinate (xu);
    \draw[->] (0,-.7) coordinate (yl) -- (0,1) coordinate (yu);
    \node (a) at (45:.7) {$\oplus$};
    \node (r1) at (-45:.7) {$\ominus$};
    \node (r2) at (135:.7) {$\ominus$};
  \end{tikzpicture}
  \hspace{1em}
  \begin{tikzpicture}[scale=1.5]
    \draw[->] (-1,0) coordinate (xl) -- (1,0) coordinate (xu);
    \draw[->] (0,-.7) coordinate (yl) -- (0,1) coordinate (yu);
    \node (a1) at (30:.7) {$\oplus$};
    \node (a2) at (135:.7) {$\oplus$};
    \node (r1) at (60:.7) {$\ominus$};
    \node (r2) at (150:.7) {$\ominus$};
  \end{tikzpicture}
\end{wrapfigure}
We represent finite assessments graphically in our two-dimensional example format using $\oplus$ for accepted gambles and $\ominus$ for rejected ones.
On the right, we give three examples, which will be extended further on.
Each of these examples has been chosen to exhibit interesting behaviour at some later point in this discussion, so the accepted and rejected gambles have not been chosen arbitrarily.
One thing we can already indicate here is that the choice of rejected gambles in the leftmost example will effectively mean that relative losses above a certain threshold are deemed unacceptable.

\subsection{Unresolved Gambles, Confusing Gambles, and No Confusion}\label{sec:no-confusion}
Based on the statements that have been made about a gamble~$f$, it can fall into one of four categories.
It can be only accepted, only rejected, neither accepted nor rejected, or both accepted and rejected.

When the gamble is neither accepted nor rejected, it is called \demph{unresolved}; the set of unresolved gambles is $\A_\unres\defeq\compl{\group{\A_\acc\union\A_\rej}}$.
When the gamble is both accepted and rejected, it is said to be \demph{confusing}; the set of confusing gambles is $\A_\confus\defeq\A_\acc\intersection\A_\rej$.
Similarly, we denote the set of non-confusing acceptable gambles by $\D_\accnrej\coloneqq\D_\acc\setminus\D_\rej$ and the set of non-confusing rejected gambles by $\D_\rejnacc\coloneqq\D_\rej\setminus\D_\acc$.

To formalise the meaning attached to rejection---that acceptance is unreasonable---, we must judge confusion to be a situation that has to be avoided.
This corresponds to the following rationality axiom:
\begin{axiom}[No Confusion]{NC}\label{eq:no-confusion}
  $\A_\confus=\A_\acc\intersection\A_\rej=\emptyset$.
\end{axiom}
\noindent
The set of assessments without confusion is $\ncAs\defeq\cset{\A\in\As}{\A_\confus=\emptyset}$.
Assessments $\A$ in $\ncAs$ partition the space~$\linGs$ of all gambles of interest into $\set{\A_\acc,\A_\rej,\A_\unres}$.
(We allow partition elements to be empty.)

Although sources of confusion should ideally be investigated, it is possible to automatically remove confusion from assessments in a number of ways, of which we mention only a few here:
\begin{proposition}\label{prop:removing-confusion-from-assessments}
  Confusion can be removed from an assessment $\A$ in $\As$ by removing the confusing gambles from the accepted gambles, from the rejected gambles, or from both.
  Formally:
  $
  \set[\big]{
    \Adelim{\A_\accnrej}{\A_\rej},
    \Adelim{\A_\acc}{\A_\rejnacc},
    \Adelim{\A_\accnrej}{\A_\rejnacc}
  }\subseteq\ncAs$.
\end{proposition}

In our graphical examples, given that only a small finite number of statements have been made, all gambles are unresolved apart from the ones involved in these statements.
The resolved ones belong to either $\D_\accnrej$ or $\D_\rejnacc$, i.e., are either accepted or rejected.

\subsection{Indifference, Favourability, and Uncomparability}\label{sec:derived}
Given an assessment $\A$ in $\As$, we can introduce three other types of statements an agent can make about gambles by considering both a gamble and its (point-wise) negation.
Making a statement about one gamble in such a pair does not restrict the statements that can be made about the other.
However, useful interpretations can be given to certain combinations of statements.

We say that the agent is \demph{indifferent} about a gamble~$f$ if he finds both it and its negation $-f$ acceptable.
The set of indifferent gambles is $\A_\indiff\defeq\A_\acc\intersection-\A_\acc$.

We say that the agent finds a gamble~$f$ \demph{favourable} if he finds it acceptable, but rejects its negation~$-f$.
The reason for this terminology is that in such a configuration the agent accepts the gamble~$f$ and moreover effectively states non-indifference about it---unless he is in a state of confusion, of course.
The set of favourable gambles is~$\A_\fav\defeq\A_\acc\intersection-\A_\rej$.
The zero gamble cannot be favourable without being confusing because $-0=0$.

We say that a gamble~$f$ is \demph{indeterminate} if neither it nor its negation~$-f$ are acceptable, so that no predetermined operational procedure is associated to it.
The set of indeterminate gambles is~$\A_\incomp\defeq\compl{(\A_\acc\union-\A_\acc)}$.

In our graphical examples, no indifferent or favourable gambles appear; these categories will be illustrated later.
Of course, again given that only a small finite number of statements has been made, most pairs of gambles related by negation consist of unresolved gambles and these gambles are therefore indeterminate.
In the middle example, there is however one pair of rejected gambles related by negation; the gambles involved, even though resolved, are also called indeterminate because no accept statement is involved, so no commitments are determined.

\subsection{Deductive Extension and Deductive Closure}\label{sec:dedcls}
Based on the assumption that the gamble pay-offs are expressed in a linear precise utility scale, statements of acceptance imply other statements, generated by positive scaling and combination: if~$f$ is judged acceptable, then $\lambda\cdot f$ should be as well for all real~$\lambda>0$; if~$f$ and~$g$ are judged acceptable, then $f+g$ should be as well.
This is called \demph{deductive extension}.
Deductive extension can be succinctly expressed using the positive linear hull operator~$\phull$, which generates convex cones and was introduced in Section~\ref{sec:setup}.
The assumptions about the utility scale have no \emph{direct} consequences for reject statements; their \emph{indirect} impact will be derived in Section~\ref{sec:nolimbo}.

So, starting from an assessment~$\A$ in~$\As$, its deductive extension $\ext{\Ds}\A\defeq\Adelim{\phull\A_\acc}{\A_\rej}$, which we call a \emph{deductively closed} assessment, can be derived.
Deductively closed assessments~$\D$ satisfy the following rationality axiom:
\begin{axiom}[Deductive Closure]{DC}\label{eq:deduction}
    $\ext{\Ds}\D=\D$ \quad or, equivalently, \quad $\D_\acc\in\cones$,
\end{axiom}
\noindent
which can also be expressed as the combination of 
\begin{axiom}[Positive Scaling]{PS}\label{eq:scaling}
  $\lambda>0 \conj f\in\D_\acc \then \lambda\cdot f\in\D_\acc$
  \quad or, equivalently, \quad
  $\reals_{>}\cdot\D_\acc\subseteq\D_\acc$
\end{axiom}
\noindent
and
\begin{axiom}[Combination]{C}\label{eq:combination}
  $f,g\in\D_\acc \then f+g\in\D_\acc$
  \quad or, equivalently, \quad
  $\D_\acc+\D_\acc\subseteq\D_\acc$.
\end{axiom}
\noindent
The subset of $\As$ consisting of all deductively closed assessments is---not surprisingly---denoted by~$\Ds$ and those without confusion by $\ncDs\defeq\Ds\intersection\ncAs$.
Not all assessments without confusion remain so after deductive extension; those that do are called \emph{deductively closable} and form the set $\dcAs\defeq\cset{\A\in\ncAs}{\ext{\Ds}\A\in\ncDs}$, where we have made use of the fact that $\ext{\Ds}$ never removes statements and therefore cannot remove confusion.

It is useful to have an explicit criterion on hand to test whether an assessment is deductively closable or not:
\begin{theorem}\label{thm:dcAs-zerocrit}
  An assessment~$\A$ in~$\As$ is deductively closable---i.e., $\A\in\dcAs$---if and only if\/ $0\notin\A_\rej-\phull\A_\acc$.
\end{theorem}
\noindent
This criterion is a feasibility problem.
When the assessment consists of only a finite number of statements, the feasible space $\A_\rej-\phull\A_\acc$ is a union of convex cones $\Union_{f\in\A_\rej}(\set{f}-\phull\A_\acc)$ and the problem becomes a disjunctive linear program.
It reduces to a plain linear program when, for example, $\A_\rej$ is also convex.

Again, it is possible to automatically remove confusion from deductively closed assessments, but there is less flexibility than for assessments because not all modified assessments suggested in Proposition~\ref{prop:removing-confusion-from-assessments} are deductively closable:
\begin{proposition}\label{prop:removing-confusion-from-deductions}
  While ensuring the resulting assessment is still deductively closed, confusion can be removed from a deductively closed assessment $\D$ in $\Ds$, by removing the confusing gambles from the rejected gambles or by removing them from both the accepted and rejected gambles and then taking the deductive extension.
  So formally we have:
  $
    \set{
      \Adelim{\D_\acc}{\D_\rejnacc},
      \ext{\Ds}\Adelim{\D_\accnrej}{\D_\rejnacc}
    }\subseteq\ncDs.
  $
\end{proposition}

The set of indifferent gambles~$\D_\indiff=\D_\acc\intersection-\D_\acc$ of a deductively closed assessment~$\D$ in~$\Ds$ is the negation invariant part of the convex cone~$\D_\acc$, so it is either empty or a linear space, the cone's so-called \emph{lineality space}.
The lineality space can be used in a couple of useful results:
\begin{proposition}\label{prop:lineality}
  Given a deductively closed assessment $\D$ in~$\Ds$ with a non-empty set of indifferent gambles $\D_\indiff$, denote the set of non-indifferent acceptable gambles by $\D_\accnindiff\coloneqq\D_\acc\setminus\D_\indiff$, then
  \begin{enumerate}
    \item\label{item:lineality-isq} there is indifference to status quo: $0\in\D_\indiff$,
    \item\label{item:lineality-invariance}
      the set $\D_\accnindiff$ is invariant under Minkowski addition with the indifferent gambles: $\D_\accnindiff+\D_\indiff=\D_\accnindiff$, and
    \item\label{item:lineality-cone}
       the set $\D_\accnindiff$ is a cone: $\D_\accnindiff\in\cones$.
  \end{enumerate}
\end{proposition}
\noindent
Results such as Claim~\ref{item:lineality-cone} in which some set is found to be a cone are useful, because they alert us to the fact that, in principle, working with these sets should often amount to solving convex optimisation problems.
Decomposition results such as Claim~\ref{item:lineality-invariance}, where some set is the Minkowski sum of a linear space and another set, may allow for a lower-dimensional representation of that set, much like a cylinder can be characterised by its axis and some cylindrical section.
The idea is to use a complement of the lineality space in $\linGs$:
For example, if $\someGs\subseteq\linGs$ is a linear space such that $\linGs=\D_\indiff\oplus\someGs$, where $\oplus$ denotes the direct sum of linear spaces, then we can use $\D_\accnindiff\intersection\someGs$ as a lower-dimensional representation of~$\D_\accnindiff$.

\begin{wrapfigure}{r}{0pt}
  \begin{tikzpicture}[scale=1.5]
    \draw[->] (-1,0) coordinate (xl) -- (1,0) coordinate (xu);
    \draw[->] (0,-.7) coordinate (yl) -- (0,1) coordinate (yu);
    \node (a1) at (30:.7) {$\oplus$};
    \node (a2) at (60:.7) {$\oplus$};
    \node (r1) at (-30:.7) {$\ominus$};
    \node (r2) at (120:.7) {$\ominus$};
    \begin{pgfonlayer}{background}
      \draw[border] (0,0) -- (intersection of 0,0--a1 and xu--{xu|-yu}) coordinate (a1away);
      \draw[border] (0,0) -- (intersection of 0,0--a2 and yu--{xu|-yu}) coordinate (a2away);
      \fill[lightgray] (0,0) -- (a1away) |- (a2away) --cycle;
    \end{pgfonlayer}
  \end{tikzpicture}
  \hspace{1em}
  \begin{tikzpicture}[scale=1.5]
    \draw[->] (-1,0) coordinate (xl) -- (1,0) coordinate (xu);
    \draw[->] (0,-.7) coordinate (yl) -- (0,1) coordinate (yu);
    \node (a) at (45:.7) {$\oplus$};
    \node (r1) at (-45:.7) {$\ominus$};
    \node (r2) at (135:.7) {$\ominus$};
    \begin{pgfonlayer}{background}
      \draw[border] (0,0) -- (intersection of 0,0--a and xu--{xu|-yu});
    \end{pgfonlayer}
  \end{tikzpicture}
  \hspace{1em}
  \begin{tikzpicture}[scale=1.5]
    \draw[->] (-1,0) coordinate (xl) -- (1,0) coordinate (xu);
    \draw[->] (0,-.7) coordinate (yl) -- (0,1) coordinate (yu);
    \node (a1) at (30:.7) {$\oplus$};
    \node (a2) at (135:.7) {$\oplus$};
    \node (r1) at (60:.7) {$\ominus$};
    \node (r2) at (150:.7) {$\ominus$};
    \begin{pgfonlayer}{background}
      \draw[border] (0,0) -- (intersection of 0,0--a1 and xu--{xu|-yu}) coordinate (a1away);
      \draw[border] (0,0) -- (intersection of 0,0--a2 and yu--{xu|-yu}) coordinate (a2away);
      \fill[lightgray] (0,0) -- (a1away) |- (a2away) --cycle;
    \end{pgfonlayer}
  \end{tikzpicture}
\end{wrapfigure}
Applying deductive extension~$\ext{\Ds}$ to the three example assessments given in Section~\ref{sec:accept+reject} results in the deductively closed assessments pictured on the right.
The area filled light grey is the set of accepted gambles generated by deductive  extension, black lines indicate included border rays.
The conical nature of these sets is immediately apparent, as well as the fact that working with accepted gambles, or with the rays they generate, amounts to the same thing.
Because these deductively closed sets result from direct, finite assessments, the cones are all closed; this is not necessarily so in general, e.g., in Section~\ref{sec:previsions} we will see how even finitary prevision assessments generate infinite acceptability assessments that lead to non-closed cones.
Also visible is the fact that the set of rejected gambles is not affected by deductive closure.

We see that the assessment in the third example is not deductively closable, because after deductive extension one rejected gamble becomes confusing.
Both procedures for removing confusion in Proposition~\ref{prop:removing-confusion-from-deductions} lead to the same deductively closed assessment: the one resulting from removing the rejection statement for the single confusing gamble.

\subsection{Limbo, Reckoning Extension, No Limbo, and Models}\label{sec:nolimbo}
Deductive Closure does have more of an impact than is apparent at first sight.
Consider a deductively closed assessment~$\D$ in~$\Ds$ that is the deductive extension of the agent's assessment.
Furthermore consider an unresolved gamble~$f$, i.e., $f\in\D_\unres$.
What happens if the agent makes a statement about this gamble to augment his assessment?

Were~$f$ to be rejected, then we would be interested in $\ext{\Ds}\Adelim{\D_\acc}{\D_\rej\union\set{f}}$, which is just $\Adelim{\D_\acc}{\D_\rej\union\set{f}}$.
Consequently, there is no increase in confusion.
On the other hand, were~$f$ to be accepted, then we would have to focus on $\ext{\Ds}\Adelim{\D_\acc\union\set{f}}{\D_\rej}$, which is equal to $\Adelim{\D_\acc\union(\ray{f}+\D_\acc)\union\ray{f}}{\D_\rej}$.
This new deductively closed assessment may exhibit an increase in confusion:
\begin{proposition}\label{prop:limbo}
  Given a deductively closed assessment $\D$ in $\Ds$ and an unresolved gamble $f$ in $\D_\unres$, then
  \begin{enumerate}
    \item\label{item:limbo-condition}
      there is no increase in confusion, i.e.,
      $
        \group[\big]{\ext{\Ds}\Adelim{\D_\acc\union\set{f}}{\D_\rej}}_\confus
        \subseteq \D_\confus,
      $
      if and only if
      $
        f \notin \shull{\D_\rejnacc}\union\group{\shull{\D_\rejnacc}-\D_\acc},
      $
    \item\label{item:limbo-disjoint}
      the set $\D_\acc$ is disjoint with both  $\shull{\D_\rejnacc}$ and $\shull{\D_\rejnacc}-\D_\acc$.
  \end{enumerate}
\end{proposition}
\noindent
We say that the gambles in
$
  \group[\big]{
    \shull{\D_\rejnacc}\union\group{\shull{\D_\rejnacc}-\D_\acc}
  } \setminus \D_\rej
$
are in \demph{limbo} and we call this set the limbo of~$\D$.
We use this imagery because there is no real choice for the gambles in this set: although they are not rejected yet, the only thing to do is to reject them, if an increase in confusion is to be avoided. 
Proposition~\ref{prop:limbo} tells us that under Deductive Closure gambles in limbo have exactly the same effect as gambles in $\D_\rej$: considering them as acceptable increases confusion.
So---echoing the meaning attached to reject statements---, accepting them would be unreasonable.

When the deductively closed assessment we start from satisfies No Confusion, it holds that $\D_\rejnacc=\D_\rej$ and the expression for limbo simplifies:
\begin{corollary}\label{cor:limbo-noconfusion}
  Given a deductively closed assessment without confusion $\D$ in $\ncDs$ and an unresolved gamble $f$ in $\D_\unres$, then
  \begin{enumerate}
    \item\label{item:cor:limbo-noconfusion:condition}
      no confusion is created, i.e.,
      $
        \group[\big]{\ext{\Ds}\Adelim{\D_\acc\union\set{f}}{\D_\rej}}_\confus
        = \emptyset,
      $
      if and only if
      $
        f \notin \shull{\D_\rej}\union\group{\shull{\D_\rej}-\D_\acc},
      $
    \item\label{item:cor:limbo-noconfusion:disjoint}
      the set $\D_\acc$ is disjoint with both $\shull{\D_\rej}$ and $\shull{\D_\rej}-\D_\acc$, and
    \item
      the limbo of\/ $\D$ is
      $
        \group[\big]{
          \shull{\D_\rej}\union\group{\shull{\D_\rej}-\D_\acc}
        } \setminus \D_\rej.
      $
  \end{enumerate}
\end{corollary}

Starting from a deductively closed assessment~$\D$ in~$\Ds$, additionally rejecting the gambles that are in its limbo---i.e., those that would lead to an increase in confusion if added instead to $\D_\acc$---results in its \demph{reckoning extension}
$
  \ext{\Ms}\D \defeq \Adelim[\big]{
    \D_\acc
  }{
    \D_\rej \union \shull{\D_\rejnacc}
            \union (\shull{\D_\rejnacc} - \D_\acc)
  },
$
which we call a \demph{model}.
Models~$\M$ are deductively closed assessments that satisfy the following rationality axiom:
\begin{axiom}[No Limbo]{NL}\label{eq:no-limbo}
  $\ext{\Ms}\M=\M$
  \quad or, equivalently, \quad
  $
    \shull{\M_\rejnacc} \union (\shull{\M_\rejnacc} - \M_\acc)
    \subseteq \M_\rej.
  $
\end{axiom}
\noindent
The subset of $\As$ consisting of all models is denoted by~$\Ms$ and those without confusion by $\ncMs\defeq\Ms\intersection\ncAs$.
By definition, reckoning extension cannot increase or create confusion for deductively closed assessments.
This means that for an assessment $\D$ that is deductively closed and avoids confusion, its reckoning extension $\ext{\Ms}\D$ is a model without confusion:
\begin{proposition}\label{prop:reckoning-for-closed-noconf}
  Given a deductively closed assessment without confusion $\D$ in $\ncDs$, then its reckoning extension is a model without confusion:
  $
    \ext{\Ms}\D =
    \Adelim{\D_\acc}{\shull{\D_\rej}\union\group{\shull{\D_\rej}-\D_\acc}}
    \in \ncMs.
  $
\end{proposition}
We can take this result one step further, starting again from an assessment that does not need to be deductively closed:
\begin{corollary}\label{cor:dedclostomod}
  Any deductively closable assessment $\A$ can be extended to a model without confusion.
  Formally: if $\A\in\dcAs$, then $\ext{\Ms}(\ext{\Ds}\A)\in\ncMs$.
\end{corollary}

When one wants to automatically remove confusion from models, they may be treated as deductively closed assessments, meaning that Proposition~\ref{prop:removing-confusion-from-deductions} provides the appropriate answers.
In any case, however they are constructed, models without confusion have some useful additional properties:
\begin{proposition}\label{prop:model+hulls}
  Given a model without confusion $\M$ in $\ncMs$, then
  \begin{enumerate}
    \item\label{item:prop:model+hulls:shullrej}
      the structure of the set of rejected gambles is described by $\shull{\M_\rej}=\M_\rej= \M_\rej\union(\M_\rej-\M_\acc)$,
    \item\label{item:prop:model+hulls:favexpr}
      the set of favourable gambles is $\M_\fav=\group{\M_\acc-\M_\rej} \intersection \M_\acc$, and
    \item\label{item:prop:model+hulls:faviscone}
      the set of favourable gambles is a cone: $\M_\fav\in\cones$.
  \end{enumerate}
\end{proposition}
\noindent
Adding anything favourable to something acceptable sweetens the deal to something favourable:
\begin{corollary}[Sweetened Deals]\label{cor:sweetened-deals}
  Given a model without confusion $\M$ in $\ncMs$, then $\M_\acc+\M_\fav=\M_\fav$.
\end{corollary}

\begin{wrapfigure}{r}{0pt}
  \begin{tikzpicture}[scale=1.5]
    \draw[->] (-1,0) coordinate (xl) -- (1,0) coordinate (xu);
    \draw[->] (0,-.7) coordinate (yl) -- (0,1) coordinate (yu);
    \node (a1) at (30:.7) {$\oplus$};
    \node (a2) at (60:.7) {$\oplus$};
    \node (r1) at (-30:.7) {$\ominus$};
    \node (r2) at (120:.7) {$\ominus$};
    \begin{pgfonlayer}{background}
      \draw[border] (0,0) -- (intersection of 0,0--a1 and xu--{xu|-yu}) coordinate (a1away);
      \draw[border] (0,0) -- (intersection of 0,0--a2 and yu--{xu|-yu}) coordinate (a2away);
      \fill[lightgray] (0,0) -- (a1away) |- (a2away) --cycle;
      \draw[border] (0,0) -- (intersection of 0,0--r1 and xu--{xu|-yu}) coordinate (r1away);
      \draw[border] (0,0) -- (intersection of 0,0--r2 and yu--{xu|-yu}) coordinate (r2away);
      \fill[gray] (0,0) -- (r1away) |- (xl|-yl) |- (r2away) --cycle;
      \draw[favborder] (0,0) -- (a1away) (0,0) -- (a2away);
    \end{pgfonlayer}
  \end{tikzpicture}
  \hspace{1em}
  \begin{tikzpicture}[scale=1.5]
    \draw[->] (-1,0) coordinate (xl) -- (1,0) coordinate (xu);
    \draw[->] (0,-.7) coordinate (yl) -- (0,1) coordinate (yu);
    \node (a) at (45:.7) {$\oplus$};
    \node (r1) at (-45:.7) {$\ominus$};
    \node (r2) at (135:.7) {$\ominus$};
    \begin{pgfonlayer}{background}
      \draw[border] (0,0) -- (intersection of 0,0--a and xu--{xu|-yu});
      \draw[border] (0,0) -- (intersection of 0,0--r1 and yl--{xu|-yl}) coordinate (r1away);
      \draw[border] (0,0) -- (intersection of 0,0--r2 and yu--{xu|-yu}) coordinate (r2away);
      \fill[gray] (0,0) -- (r1away) |- (xl|-yl) |- (r2away) --cycle;
      \draw[border,white] (0,0) -- (intersection of 0,0--a and yl--{yl-|xl});
      \draw[exclborder] (0,0) -- (intersection of 0,0--a and yl--{yl-|xl});
    \end{pgfonlayer}
  \end{tikzpicture}
  \hspace{1em}
  \begin{tikzpicture}[scale=1.5]
    \draw[->] (-1,0) coordinate (xl) -- (1,0) coordinate (xu);
    \draw[->] (0,-.7) coordinate (yl) -- (0,1) coordinate (yu);
    \node (a1) at (30:.7) {$\oplus$};
    \node (a2) at (135:.7) {$\oplus$};
    \node (r1) at (60:.7) {$\ominus$};
    \node (r2) at (150:.7) {$\ominus$};
    \begin{pgfonlayer}{background}
      \draw[border] (0,0) -- (intersection of 0,0--a1 and xu--{xu|-yu}) coordinate (a1away);
      \draw[border] (0,0) -- (intersection of 0,0--a2 and yu--{xu|-yu}) coordinate (a2away);
      \fill[lightgray] (0,0) -- (a1away) |- (a2away) --cycle;
      \draw[border] (0,0) -- (intersection of 0,0--r2 and xl--{xl|-yl}) coordinate (r2away);
      \draw[exclborder] (0,0) -- (intersection of 0,0--a2 and yl--{xu|-yl}) coordinate (a2anotherway);
      \fill[gray] (0,0) -- (a2anotherway) |- (xl|-yl) |- (r2away) --cycle;
      \draw[favborder] (0,0) -- (a1away);
    \end{pgfonlayer}
  \end{tikzpicture}
\end{wrapfigure}
Applying reckoning extension~$\ext{\Ms}$ to the three deductively closed example assessments given in Section~\ref{sec:dedcls} results in the models depicted on the right.
The area filled dark grey is the set of rejected gambles implied by No Limbo, dashed black lines emphasise excluded rays.

The first two examples illustrate that a model's set of rejected gambles need not be convex.
In the first case, we can interpret this as follows: the agent effectively finds accepting gambles with a gain to loss ratio below a certain threshold---here smaller than one---unreasonable.
In the second example this set is even disconnected---arguably due to a degenerate setup.

The last example illustrates that reckoning extension only acts on the unconfused parts of a model's set of rejected gambles.
Making abstraction of such confusing gambles, it is visible that reckoning extension transfers the practical equivalence of working with gambles or rays from the set of accepted gambles to the set of rejected gambles.

In the first example, all acceptable rays turn out to be favourable as well; for the border rays, this is indicated by dotting them.
In the second example, there are no favourable gambles.
In the third example, all acceptable rays are favourable except for one---undotted---border ray; note how the opposing ray---a border ray of the set of rejected gambles---is not included in the set of rejected gambles.

\subsection{Order-Theoretic Considerations: Resolve, Intersection Structures, and Closure Operators}\label{sec:order}
(\Citet{Davey-Priestley-1990} provide a good supporting reference for the material in this section.)

The typical set-theoretic operations---e.g., union~$\Union$, intersection~$\Intersection$, and set difference~$\setminus$---are extended to assessments by component-wise application.
Pairs of assessments can be compared component-wise using their set-theoretic  inclusion relationship.
We say that an assessment~$\A$ is at most as \demph{resolved} as an assessment~$\otherA$ if the former's components are included in those of the latter: $\A\subseteq\otherA$ if and only if $\A_\acc \subseteq \otherA_\acc$ and $\A_\rej \subseteq \otherA_\rej$.
The resolve terminology is based on the fact that more resolved assessments contain less unresolved gambles.
We work with component-wise inclusion of assessments and not with inclusion of sets of resolved gambles because we will gratefully exploit the additional mathematical structure it provides.

We also introduce terminology for the comparison of assessments based on one component only:
We say that an assessment~$\A$ is at most as \emph{committal} as an assessment~$\otherA$ if the former's set of acceptable gambles is included in the latter's: $\A_\acc\subseteq\otherA_\acc$.
Effectively, more committal assessments contain more commitments to engage in transactions of the type described in Section~\ref{sec:accept+reject}.
We furthermore say that an assessment~$\A$ is at most as \emph{restrictive} as an assessment~$\otherA$ if the former's set of rejected gambles is included in the latter's: $\A_\rej\subseteq\otherA_\rej$.
More restrictive assessments restrict the potential commitments to a smaller set.

\begin{wrapfigure}{r}{0pt}
  $
    \begin{tikzpicture}[scale=1.5,baseline=-.55ex]
      \draw[->] (-.8,0) coordinate (xl) -- (1,0) coordinate (xu);
      \draw[->] (0,-.7) coordinate (yl) -- (0,1) coordinate (yu);
      \useasboundingbox (xl|-yl) rectangle (xu|-yu);
      \node (a) at (45:.7) {$\oplus$};
      \node (r1) at (-45:.7) {$\ominus$};
      \node (r2) at (135:.7) {$\ominus$};
      \begin{pgfonlayer}{background}
        \draw[border] (0,0) -- (intersection of 0,0--a and xu--{xu|-yu});
        \draw[border] (0,0) -- (intersection of 0,0--r1 and yl--{xu|-yl}) coordinate (r1away);
        \draw[border] (0,0) -- (intersection of 0,0--r2 and xl--{xl|-yu}) coordinate (r2away);
        \fill[gray] (0,0) -- (r1away) |- (xl|-yl) |- (r2away) --cycle;
        \draw[border,white] (0,0) -- (intersection of 0,0--a and yl--{yl-|xl});
        \draw[exclborder] (0,0) -- (intersection of 0,0--a and yl--{yl-|xl});
      \end{pgfonlayer}
    \end{tikzpicture}
    \;\subset\;
    \begin{tikzpicture}[scale=1.5,baseline=-.55ex]
      \draw[->] (-.8,0) coordinate (xl) -- (1,0) coordinate (xu);
      \draw[->] (0,-.7) coordinate (yl) -- (0,1) coordinate (yu);
      \useasboundingbox (xl|-yl) rectangle (xu|-yu);
      \node (a1) at (30:.7) {$\oplus$};
      \node (a2) at (60:.7) {$\oplus$};
      \node (r1) at (-30:.7) {$\ominus$};
      \node (r2) at (120:.7) {$\ominus$};
      \begin{pgfonlayer}{background}
        \draw[border] (0,0) -- (intersection of 0,0--a1 and xu--{xu|-yu}) coordinate (a1away);
        \draw[border] (0,0) -- (intersection of 0,0--a2 and yu--{xu|-yu}) coordinate (a2away);
        \fill[lightgray] (0,0) -- (a1away) |- (a2away) --cycle;
        \draw[border] (0,0) -- (intersection of 0,0--r1 and xu--{xu|-yu}) coordinate (r1away);
        \draw[border] (0,0) -- (intersection of 0,0--r2 and yu--{xu|-yu}) coordinate (r2away);
        \fill[gray] (0,0) -- (r1away) |- (xl|-yl) |- (r2away) --cycle;
        \draw[favborder] (0,0) -- (a1away) (0,0) -- (a2away);
      \end{pgfonlayer}
    \end{tikzpicture}
    \;\lower1ex\hbox{\shortstack{$\nsubseteq$\\$\nsupseteq$}}\;
    \begin{tikzpicture}[scale=1.5,baseline=-.55ex]
      \draw[->] (-.8,0) coordinate (xl) -- (1,0) coordinate (xu);
      \draw[->] (0,-.7) coordinate (yl) -- (0,1) coordinate (yu);
      \useasboundingbox (xl|-yl) rectangle (xu|-yu);
      \node (a1) at (30:.7) {$\oplus$};
      \node (a2) at (135:.7) {$\oplus$};
      \node (r1) at (60:.7) {$\ominus$};
      \node (r2) at (150:.7) {$\ominus$};
      \begin{pgfonlayer}{background}
        \draw[border] (0,0) -- (intersection of 0,0--a1 and xu--{xu|-yu}) coordinate (a1away);
        \draw[border] (0,0) -- (intersection of 0,0--a2 and xl--{xl|-yu}) coordinate (a2away);
        \fill[lightgray] (0,0) -- (a1away) |- (yu) -| (a2away) --cycle;
        \draw[border] (0,0) -- (intersection of 0,0--r2 and xl--{xl|-yl}) coordinate (r2away);
        \draw[exclborder] (0,0) -- (intersection of 0,0--a2 and yl--{xu|-yl}) coordinate (a2anotherway);
        \fill[gray] (0,0) -- (a2anotherway) |- (xl|-yl) |- (r2away) --cycle;
        \draw[favborder] (0,0) -- (a1away);
      \end{pgfonlayer}
    \end{tikzpicture}
  $
\end{wrapfigure}
On the right, the resolvedness relation is illustrated using the example models we encountered at the end of Section~\ref{sec:nolimbo}:
The middle model is more resolved than the left one, but neither more nor less resolved than the right one.
This is a consequence of the fact that the models increase qua commitments from left to right, but that the middle model is more restrictive than the outer ones.

Under the `at most as resolved as'-relation $\subseteq$, the set of assessments constitutes a complete lattice $(\As,\subseteq)$, where the union operator $\Union$ plays the role of supremum and the intersection operator $\Intersection$ that of infimum.
Its bottom is $\bot\defeq\Adelim{\emptyset}{\emptyset}$ and its top is $\top\defeq\Adelim{\linGs}{\linGs}$.

The derived posets $(\ncAs,\subseteq)$, $(\dcAs,\subseteq)$, $(\Ds,\subseteq)$, $(\ncDs,\subseteq)$, and $(\ncMs,\subseteq)$ all have an interesting order-theoretic nature, as they are \demph{intersection structures}:
\begin{proposition}\label{prop:intersection}
  The sets $\ncAs$, $\dcAs$, $\Ds$, $\ncDs$, and\/ $\ncMs$ are closed under arbitrary non-empty intersections:
    let $\someAs$ denote any one of these sets and let $\somemoreAs\subseteq\someAs$, then $\Intersection\somemoreAs\in\someAs$.
\end{proposition}
\noindent
All posets $(\ncAs,\subseteq)$, $(\dcAs,\subseteq)$, $(\Ds,\subseteq)$, $(\ncDs,\subseteq)$, and $(\ncMs,\subseteq)$ are therefore complete infimum-semi-lattices where~$\Intersection$ plays the role of infimum.
They have a common bottom~$\bot$, but only $(\Ds,\subseteq)$ has a top, $\top$.
The others have multiple maximal elements, respectively forming the sets~$\maxncAs$, $\maxdcAs$, $\maxncDs$, and $\maxncMs$.
Given $\someAs\subseteq\As$, then the \emph{set of maximal elements} of $(\someAs,\subseteq)$ is
$
  \maxsomeAs \defeq
    \cset[\big]{\A\in\someAs}{(\forall\otherA\in\someAs)\A\nsubset\otherA},
$
with $\A\subset\otherA$ if and only if $\A\subseteq\otherA$ and $\A\neq\otherA$.
The explicit form of these maximal elements can be found:
\begin{proposition}\label{prop:maximalncasss}
  The set of maximal assessments without confusion is
  $
    \maxncAs = \cset{\Adelim{\someGs}{\linGs\setminus\someGs}}{\someGs\subseteq\linGs}.
  $
  So an assessment without confusion $\A$ in $\ncAs$ is maximal if and only if it has no unresolved gambles: $\A_\unres=\emptyset$.
\end{proposition}
\begin{proposition}\label{prop:maximalncmods}
  The set of all maximal models without confusion is
  $
    \maxncMs
    = \maxncAs\intersection\ncDs
    = \cset{\Adelim{\someGs}{\linGs\setminus\someGs}}{\someGs\in\cones}.
  $
  It moreover coincides with the sets of all maximal deductively closed or closable assessments without confusion: $\maxncMs=\maxncDs=\maxdcAs$.
\end{proposition}

\begin{wrapfigure}{r}{0pt}
  $
    \begin{tikzpicture}[scale=1.5,baseline=-.7ex]
    \draw[->] (-.7,0) coordinate (xl) -- (1,0) coordinate (xu);
    \draw[->] (0,-.7) coordinate (yl) -- (0,1) coordinate (yu);
      \node (a1) at (135:.7) {$\oplus$};
      \node (a2) at (-45:.7) {$\oplus$};
      \node (r1) at (-30:.7) {$\ominus$};
      \node (r2) at (120:.7) {$\ominus$};
      \begin{pgfonlayer}{background}
        \draw[border] (0,0) -- (intersection of 0,0--a1 and xl--{xl|-yu}) coordinate (a1away);
        \draw[border] (0,0) -- (intersection of 0,0--a2 and yl--{xu|-yl}) coordinate (a2away);
        \fill[lightgray] (0,0) -- (a1away) |- (xu|-yu) |- (a2away) --cycle;
      \end{pgfonlayer}
    \end{tikzpicture}
    \intersection
    \begin{tikzpicture}[scale=1.5,baseline=-.7ex]
    \draw[->] (-.7,0) coordinate (xl) -- (1,0) coordinate (xu);
    \draw[->] (0,-.7) coordinate (yl) -- (0,1) coordinate (yu);
      \node (a1) at (0:.7) {$\oplus$};
      \node (a2) at (90:.7) {$\oplus$};
      \node (r1) at (-30:.7) {$\ominus$};
      \node (r2) at (120:.7) {$\ominus$};
      \begin{pgfonlayer}{background}
        \draw[border] (0,0) -- (intersection of 0,0--a1 and xu--{xu|-yu}) coordinate (a1away);
        \draw[border] (0,0) -- (intersection of 0,0--a2 and yu--{xu|-yu}) coordinate (a2away);
        \fill[lightgray] (0,0) -- (a1away) |- (a2away) --cycle;
        \draw[border] (0,0) -- (intersection of 0,0--r1 and xu--{xu|-yu}) coordinate (r1away);
        \draw[border] (0,0) -- (intersection of 0,0--r2 and yu--{xu|-yu}) coordinate (r2away);
        \fill[gray] (0,0) -- (r1away) |- (xl|-yl) |- (r2away) --cycle;
      \end{pgfonlayer}
    \end{tikzpicture}
    =
    \begin{tikzpicture}[scale=1.5,baseline=-.7ex]
    \draw[->] (-.7,0) coordinate (xl) -- (1,0) coordinate (xu);
    \draw[->] (0,-.7) coordinate (yl) -- (0,1) coordinate (yu);
      \node (a1) at (0:.7) {$\oplus$};
      \node (a2) at (90:.7) {$\oplus$};
      \node (r1) at (-30:.7) {$\ominus$};
      \node (r2) at (120:.7) {$\ominus$};
      \begin{pgfonlayer}{background}
        \draw[border] (0,0) -- (intersection of 0,0--a1 and xu--{xu|-yu}) coordinate (a1away);
        \draw[border] (0,0) -- (intersection of 0,0--a2 and yu--{xu|-yu}) coordinate (a2away);
        \fill[lightgray] (0,0) -- (a1away) |- (a2away) --cycle;
      \end{pgfonlayer}
    \end{tikzpicture}
  $
\end{wrapfigure}
The poset $(\Ms,\subseteq)$ is not an intersection structure.
This can be shown using the graphical counterexample on the right.
The resulting intersection is a deductively closed assessment, but not a model; the corresponding model here---which can be obtained by applying reckoning extension---is actually the second intersection factor.
Nevertheless, the poset's bottom is $\bot$ and its top is~$\top$, the same as for $(\As,\subseteq)$ and $(\Ds,\subseteq)$.

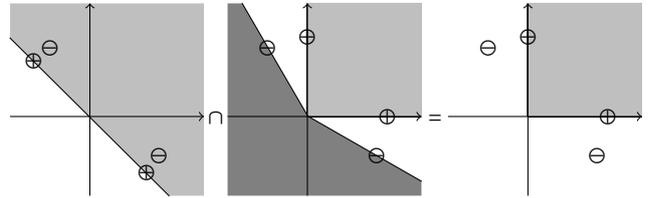
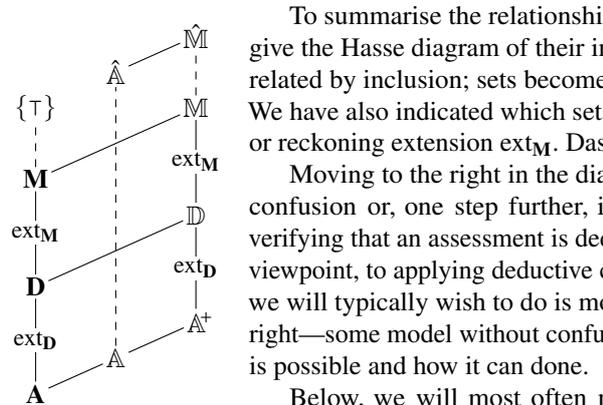
\begin{wrapfigure}{l}{0pt}
  \begin{tikzpicture}
    \matrix[column sep={3em,between origins},row sep={3ex,between origins},inner sep=1pt,fill=white] {
      && \node (mmd) {$\maxncMs$};\\
      & \node (mncass) {$\maxncAs$};\\
      \node (top) {$\set\top$}; && \node (ncmd) {$\ncMs$};\\
      \\
      \node (md) {$\Ms$}; \\
      && \node (nccp) {$\ncDs$};\\
      \\
      \node (cp) {$\Ds$};\\
      && \node (dcass) {$\dcAs\mspace{-7mu}$};\\
      & \node (ncass) {$\ncAs$}; \\
      \node (ass) {$\As$}; \\
    };
    \path (ass) edge node[pos=.5,inner sep=1pt,fill=white] {\small$\ext{\Ds}$} (cp);
    \path (ass) edge (ncass);
    \path (ncass) edge (dcass);
    \path (ncass) edge[dashed] (mncass);
    \path (dcass) edge node[pos=.5,inner sep=1pt,fill=white] {\small$\ext{\Ds}$} (nccp);
    \path (cp) edge (nccp);
    \path (cp) edge node[pos=.5,inner sep=1pt,fill=white] {\small$\ext{\Ms}$} (md);
    \path (md) edge[dashed] (top);
    \path (md) edge (ncmd);
    \path (nccp) edge node[pos=.5,inner sep=1pt,fill=white] {\small$\ext{\Ms}$} (ncmd);
    \path (ncmd) edge[dashed] (mmd);
    \path (mncass) edge (mmd);
  \end{tikzpicture}
\end{wrapfigure}
To summarise the relationships between the sets of assessments we have encountered, we give the Hasse diagram of their inclusion-based partial ordering on the left:
Full lines link sets related by inclusion; sets become smaller as we go towards the top of the page in the diagram.
We have also indicated which sets are direct images of a superset under deductive closure $\ext{\Ds}$ or reckoning extension $\ext{\Ms}$.
Dashed lines lead to sets of maximal elements.

Moving to the right in the diagram corresponds to verifying that an assessment is without confusion or, one step further, is deductively closable.
Moving straight up corresponds to verifying that an assessment is deductively closed or a model, respectively, or, from an inference viewpoint, to applying deductive closure $\ext{\Ds}$ and reckoning extension $\ext{\Ms}$, respectively.
What we will typically wish to do is move from the bottom left---some assessment in ~$\As$---to the top right---some model without confusion in $\ncMs$---; many of the results below characterise when this is possible and how it can done.

Below, we will most often not be working in the Hasse diagram pictured here, but in daughter diagrams where $\As$ is replaced by a subset of interest.
In Section~\ref{sec:simplified-frameworks}, these subsets will be characterised by one or more interesting simplifying conditions.
In the remainder of this section, these subsets are defined by replacing the bottom~$\bot$ with some other, non-trivial assessment.

Let $\someAs_\A\defeq\cset{\otherA\in\someAs}{\A\subseteq\otherA}$ be the subset of assessments in $\someAs\subseteq\As$ dominating the assessment~$\A$ in~$\As$.
With every intersection structure $(\someAs,\subseteq)$ there can be associated an operator $\cls{\someAs}$ from $\As$ to $\someAs\union\set{\top}$, defined for any assessment~$\A$ in~$\As$ by $\cls{\someAs}\A\defeq\Intersection\someAs_\A$.
For a correct understanding of this definition, recall that $\Intersection\emptyset=\top$. 
An operator $\cls{*}$ on $\As$ is a \demph{closure operator} if for all $\A$ and $\otherA$ in $\As$ it is
\begin{itemize}
  \item extensive: $\A\subseteq\cls{*}\A$,
  \item idempotent: $\cls{*}(\cls{*}\A)=\cls{*}\A$, and
  \item increasing: $\A\subseteq\otherA\then\cls{*}\A\subseteq\cls{*}\otherA$.
\end{itemize}
The operators we have defined satisfy these criteria:
\begin{proposition}\label{prop:closure}
  Given an intersection structure $(\someAs,\subseteq)$, with $\someAs\subseteq\As$, then
  \begin{enumerate}
    \item\label{item:prop:closure:cls}
      $\cls{\someAs}$ is a closure operator, and
    \item\label{item:prop:closure:id}
      $\cls{\someAs}\A=\A$ if and only if~$\A\in\someAs\union\set{\top}$.
  \end{enumerate}
\end{proposition}
\noindent
A closure operator relative to an intersection structure $(\someAs,\subseteq)$ effects the most conservative inference relative to~$\someAs$ in the sense that it generates the least resolved dominating assessment in~$\someAs$ or---if there is no such dominating assessment---returns $\top$.

With the intersection structures encountered above, there correspond the closure operators $\cls{\As}$, $\cls{\ncAs}$, $\cls{\dcAs}$, $\cls{\Ds}$, $\cls{\ncDs}$, and $\cls{\ncMs}$.
For each of these, we can give a constructive formulation for practical calculations:
\begin{proposition}\label{prop:cls-formulae}
  Next to the general result of Proposition~\ref{prop:closure}, we have that
  \begin{enumerate}
    \item\label{item:prop:cls-formulae:ncAs}
      $\cls{\ncAs}$ returns $\top$ outside of\/ $\ncAs$,
    \item\label{item:prop:cls-formulae:dcAs}
      $\cls{\dcAs}$, $\cls{\ncDs}$, and $\cls{\ncMs}$ return $\top$ outside of\/ $\dcAs$,
    \item\label{item:prop:cls-formulae:Ds}
      $\cls{\Ds}=\ext{\Ds}$,
    \item\label{item:prop:cls-formulae:ncDs}
      $\cls{\ncDs}=\ext{\Ds}$ on $\dcAs$, and
    \item\label{item:prop:cls-formulae:ncMs}
      $\cls{\ncMs}=\ext{\Ms}\after\ext{\Ds}$ on $\dcAs$ and, specifically, $\cls{\ncMs}=\ext{\Ms}$ on $\ncDs$.
  \end{enumerate}
\end{proposition}

Concatenating the closure operators so defined with the supremum operator~$\Union$ of $(\As,\subseteq)$, which can be applied to arbitrary---i.e., also infinite---families of assessments, gives us the supremum operators of the complete lattices $(\ncAs\union\set{\top},\subseteq)$, $(\dcAs\union\set{\top},\subseteq)$, $(\Ds,\subseteq)$, $(\ncDs\union\set{\top},\subseteq)$, and $(\ncMs\union\set{\top},\subseteq)$ so formed.
Supremum operators of special interest are the \emph{deductive union} $\dedUnion\defeq\cls{\Ds}\after\Union$ and the \emph{reckoning union} $\reckUnion\defeq\cls{\ncMs}\after\Union$.

\subsection{Models Dominating Assessments}\label{sec:dominating-models}
We want to represent the agent's uncertainty using models in~$\ncMs$.
(From now on in this section, unless indicated otherwise, when we talk about models, we mean unconfused ones.)
The agent provides an assessment.
Unless it is deductively closable, i.e., an element of~$\dcAs$, it is impossible to derive a most conservative model from it using the closure operator $\cls{\ncMs}$---or the operators~$\ext{\Ds}$ and~$\ext{\Ms}$---or confusion would ensue (cf. Proposition~\ref{prop:cls-formulae}).
We are therefore interested in characterisations of~$\dcAs$, the set of assessments that can be turned into models.

We consider the sets of models~$\ncMs_\A$ and maximal models~$\maxncMs_\A\defeq\maxncMs\intersection\ncMs_\A$ dominating the assessment~$\A$.
Both sets are empty if the assessment is not in~$\dcAs$; but if it is, their elements all dominate $\cls{\ncMs}\A=\ext{\Ms}\group{\ext{\Ds}\A}$, which is the bottom of $(\ncMs_\A,\subseteq)$.
These observations can be strengthened into the following characterisation of~$\dcAs$:
\begin{theorem}\label{thm:dedclosable-nonemptymax}
  An assessment $\A$ in $\As$ is deductively closable if and only if its set of maximal dominating models is non-empty.
  Formally: $\A\in\dcAs$ if and only if $\maxncMs_\A\neq\emptyset$.
\end{theorem}
\noindent
This is our counterpart to the lower envelope theorem in the theory of coherent lower previsions \citep[see, e.g.,][\S3.3.3(a)]{Walley-1991}.
It leads to a first result about the combination of assessments:
\begin{corollary}\label{cor:reckoningunionmodel}
  If all assessments in some family $\someAs\subseteq\As$ are dominated by a common model in $\ncMs$, then their reckoning union $\reckUnion\someAs$ is a model in $\ncMs$.
\end{corollary}

The maximal models dominating an assessment can also be used for inference purposes \citep[cfr.][\S3.3.3(b)]{Walley-1991}:
\begin{proposition}\label{prop:inf-by-max}
  Given an assessment $\A$ in $\As$, then its closure in the set of models is $\cls{\ncMs}\A=\Intersection\maxncMs_\A$.
\end{proposition}
\noindent
This result is practically applicable whenever $\maxncMs_\A$ has some special structure we can exploit.
For example, when it has a finite set of ‘extreme’ elements $\maxncMs_\A^*$ such that $\Intersection\maxncMs_\A^*=\Intersection\maxncMs_\A$.
The analogon in the theory of imprecise probabilities is the set of extreme points of a credal set \citep[cf., e.g.,][]{Levi-1980}.

The results in this section guarantee that unconfused models constitute an instance of what are called \emph{strong belief structures} by \citet{DeCooman-2005-order}.
This implies in particular that the whole apparatus developed there for dealing with AGM-style belief change and belief revision \citep{Alchourron+GM-JSL85p}, is also available for the unconfused models we are dealing with here.

\subsection{Background Models, Respect, Natural Extension, Coherence, and Indifference to Status Quo}\label{sec:background}
So far, we have not dealt with any structural a priori assumptions about the gambles in~$\linGs$ or the experiment.
Many of these can be captured by positing a so-called \demph{background model}~$\bgM$ in $\ncMs$ to replace the trivial smallest model~$\bot$.
In such a context, attention is evidently restricted to models in~$\ncMs_\bgM$, and when doing so, all the results of the preceding sections remain valid, mutatis mutandis.
An intuitively appealing background model is $\Adelim{\linGs_\geq}{\linGs_<}$.
Using this background model amounts to taking for granted that all non-negative gambles should be accepted, and all negative gambles rejected.\footnote{
  One reason for not generally including $\Adelim{\linGs_\geq}{\linGs_<}$ or $\Adelim{\linGs_\gtrdot}{\linGs_\lessdot}$ in the background model, would be that conditioning on an event with probability one---modelled using de Finetti's interpretation (cf. Section~\ref{sec:linprevs})---then results in a conditional model with confusion.
  (As the conditioning rule, we here take the standard one from the theory of sets of desirable gambles, namely restriction to the linear subspace corresponding to the conditioning event \citep[cf., e.g.,][\S1.3.3]{2013-Quaeghebeur-itip}.)
  We do not discuss this issue further, as we do not treat conditional models in this paper.
}
For other examples, we refer to Sections~\ref{sec:favourability}, \ref{sec:acceptability}, and~\ref{sec:previsions}.

We say that an assessment~$\A\in\As$ \demph{respects} the background model~$\bgM$ if they share a common maximal model; i.e., if 
$
  \maxncMs_\A\intersection\maxncMs_\bgM=\maxncMs_{\A\union\bgM}\neq\emptyset,
$
or, by Theorem~\ref{thm:dedclosable-nonemptymax}, that $\A\union\bgM\in\dcAs$.
The \demph{natural extension} of an assessment~$\A$ in~$\As$ is its reckoning union with the background model, $\A\reckunion\bgM$.
Corollary~\ref{cor:reckoningunionmodel} then leads to:
\begin{corollary}\label{cor:respect-natext}
  The natural extension of an assessment $\A$ in $\As$ is a model if and only if the assessment respects the background model $\bgM$ in $\ncMs$.
  Formally: $\A\reckunion\bgM\in\ncMs$ if and only if $\maxncMs_{\A\union\bgM}\neq\emptyset$ or $\A\union\bgM\in\dcAs$.
\end{corollary}
\noindent
So, more colloquially, the assessment~$\A$ respects the background model~$\bgM$ if their combination leads to a model without confusion.
We furthermore say that an assessment $\A$ in $\dcAs$ is \demph{coherent} if it coincides with its natural extension: $\A=\A\reckunion\bgM$, or equivalently $\bgM\subseteq\cls{\ncMs}\A=\A$.
To make explicit what the linear space~$\linGs$ of gambles of interest is or what the background model~$\bgM$ is, these can be used as a prefix: $(\linGs,\bgM)$-coherent, $\linGs$-coherent, $\bgM$-coherent.

Given the interpretation attached to accept statements, we judge it reasonable to always let the zero gamble~$0$---also called \emph{status quo}---be acceptable, and therefore indifferent.
This corresponds to the following rationality axiom:
\begin{axiom}[Indifference to Status Quo]{ISQ}\label{eq:indifference-to-status-quo}
  $\zeroM\subseteq\bgM$, \quad with $\zeroM\defeq\Adelim{\set{0}}{\emptyset}$.
\end{axiom}
\noindent
The set of assessments that express Indifference to Status Quo is $\As_\zeroM$, i.e., all assessments dominating~$\zeroM$.
Under Indifference to Status Quo, the limbo expression simplifies yet further (cf.~Corollary~\ref{cor:limbo-noconfusion}):
\begin{corollary}\label{cor:limbo-noconfusion+statusquo}
  Given a deductively closed assessment~$\D$ in $\ncDs_\zeroM$, so without confusion and with indifference to status quo, and an unresolved gamble~$f$ in $\D_\unres$, then
  \begin{enumerate}
    \item
      no confusion is created, i.e.,
      $
        \ext{\Ds}\Adelim{\D_\acc\union\set{f}}{\D_\rej}_\confus = \emptyset,
      $
      if and only if
      $
        f \notin \shull{\D_\rej} - \D_\acc,
      $
    \addtocounter{enumi}{1}
    \item
      the limbo of\/ $\D$ is
      $
        \group{
          \shull{\D_\rej} - \D_\acc
        } \setminus \D_\rej.
      $
  \end{enumerate}
\end{corollary}

It is possible to give a compact characterisation of models coherent with a background model that is indifferent to status quo:
\begin{theorem}\label{thm:char-AR}
  An assessment~$\M$ in~$\As$ is a model that respects a background model with indifference to status quo $\bgM$ in~$\ncMs_\zeroM$, i.e., $\M\in\ncMs_\bgM$, if and only if
  \begin{enumerate}[label=\upshape{(AR\arabic*)},leftmargin=*,widest=AR0]
    \item\label{item:ARbg} it includes the background model: $\bgM\subseteq\M$,
    \item\label{item:ARzr} it does not reject status quo: $0\notin\M_\rej$,
    \item\label{item:ARcn} its acceptable gambles form a cone: $\M_\acc\in\cones$, and
    \item\label{item:ARss} it has no limbo: $\shull{\M_\rej}-\M_\acc\subseteq\M_\rej$.
  \end{enumerate}
\end{theorem}
\noindent
We will give similar characterisation results when formulating our theory in terms of gamble relations in Section~\ref{sec:relations} and when looking at simplified versions of our theory in Section~\ref{sec:simplified-frameworks}.
These results allow for easy high-level comparison of the different representations and theory variants.

\pagebreak
\subsection{Gamble Space Partitions Induced by Assessments with No Confusion}
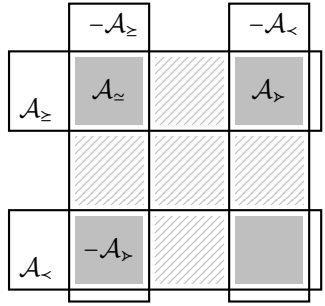
\label{sec:partitions}
\begin{wrapfigure}{l}{0pt}
  \begin{tikzpicture}
    \matrix [nodes={regular polygon, regular polygon sides=4,pattern=north east lines,pattern color=lightgray,inner sep=-1ex,minimum size=8ex,outer sep=.5ex},column sep={1ex,between borders},row sep={1ex,between borders}] {
      \node[fill=lightgray] (apam) {$\A_\indiff$}; & \node (ap) {}; & \node[fill=lightgray] (aprm) {$\A_\fav$}; \\
      \node (am) {}; & \node (none) {}; & \node (rm) {}; \\
      \node[fill=lightgray] (rpam) {$-\A_\fav$}; & \node (rp) {}; & \node[fill=lightgray] (rprm) {}; \\
    };
    \draw[thick] ([xshift=1ex]aprm.north east) rectangle ([xshift=-5ex]apam.south west) node[above right] {$\A_\acc$};
    \draw[thick] ([yshift=-1ex]rpam.south west) rectangle ([yshift=4ex]apam.north east) node[below left] {$-\A_\acc$};
    \draw[thick] ([xshift=-5ex]rpam.south west) node[above right] {$\A_\rej$} rectangle ([xshift=1ex]rprm.north east);
    \draw[thick] ([yshift=-1ex]rprm.south west) rectangle ([yshift=4ex]aprm.north east) node[below left] {$-\A_\rej$};
  \end{tikzpicture}
\end{wrapfigure}
In our \emph{accept-reject framework}, an assessment~$\A$ in~$\ncAs$---so with no confusion---partitions the linear space of gambles of interest~$\linGs$ into nine classes, each of which is defined by whether its constituent gambles and their negations are acceptable, rejected, or unresolved.
Some of these classes may be empty.
This partitioning is illustrated on the left.
We have explicitly labelled the partition elements corresponding to the sets of indifferent and favourable gambles.
To make the partition's gamble negation-based cross-diagonal symmetry visible, we have also labelled the partition element corresponding to the negation of the set of favourable gambles.
The set of indifferent gambles is invariant under this symmetry, i.e., gamble negation.

From Proposition~\ref{prop:maximalncmods}, we know that for maximal assessments all gambles in~$\linGs$ are either accepted or rejected.
Because of this, for maximal models, some partition classes are empty for sure; these 
\par
\begin{wrapfigure}{r}{0pt}
  \begin{tikzpicture}
    \matrix [nodes={regular polygon, regular polygon sides=4,inner sep=-1ex,minimum size=8ex,outer sep=.5ex},column sep={1ex,between borders},row sep={1ex,between borders}] {
      \node (apam) {}; & \node (ap) {}; & \node (aprm) {}; \\
      \node (am) {}; & \node (none) {}; & \node (rm) {}; \\
      \node (rpam) {}; & \node (rp) {}; & \node (rprm) {}; \\
    };
    \draw[thick] ([xshift=1ex]aprm.north east) rectangle ([xshift=-5ex]apam.south west) node[above right] {$\A_\acc$};
    \draw[thick] ([yshift=-1ex]rpam.south west) rectangle ([yshift=4ex]apam.north east) node[below left] {$-\A_\acc$};
    \draw[thick] ([xshift=-5ex]rpam.south west) node[above right] {$\A_\rej$} rectangle ([xshift=1ex]rprm.north east);
    \draw[thick] ([yshift=-1ex]rprm.south west) rectangle ([yshift=4ex]aprm.north east) node[below left] {$-\A_\rej$};
    \fill[lightgray] ([shift={(.5ex,-.5ex)}]none.north west) rectangle ([shift={(-.5ex,.5ex)}]rprm.south east);
    \node at ($(none)!.5!(rprm)$) {$\A_\incomp$};
  \end{tikzpicture}
  \quad
  \begin{tikzpicture}
    \matrix [nodes={regular polygon, regular polygon sides=4,inner sep=-1ex,minimum size=8ex,outer sep=.5ex},column sep={1ex,between borders},row sep={1ex,between borders}] {
      \node (apam) {}; & \node (ap) {}; & \node (aprm) {}; \\
      \node (am) {}; & \node (none) {}; & \node (rm) {}; \\
      \node (rpam) {}; & \node (rp) {}; & \node (rprm) {}; \\
    };
    \draw[thick] ([xshift=1ex]aprm.north east) rectangle ([xshift=-5ex]apam.south west) node[above right] {$\A_\acc$};
    \draw[thick] ([yshift=-1ex]rpam.south west) rectangle ([yshift=4ex]apam.north east) node[below left] {$-\A_\acc$};
    \draw[thick] ([xshift=-5ex]rpam.south west) node[above right] {$\A_\rej$} rectangle ([xshift=1ex]rprm.north east);
    \draw[thick] ([yshift=-1ex]rprm.south west) rectangle ([yshift=4ex]aprm.north east) node[below left] {$-\A_\rej$};
    \fill[pattern=north east lines,pattern color=lightgray] ([shift={(.5ex,-.5ex)}]am.north west) rectangle ([shift={(-.5ex,.5ex)}]rm.south east);
    \node at ($(am)!.5!(rm)$) {$\A_\unres$};
  \end{tikzpicture}
\end{wrapfigure}
\noindent
have been hatched instead of filled.
Whenever a background model~$\bgM$ has been posited, the picture stays the same, but the background model~$\bgM$ constrains some of the partition classes to be non-empty.

It is useful to see where the other assessment categories we have encountered lie as contextualised by the above partition.
We have done this on the right:
In the first figure, we show the set of indeterminate gambles, which is invariant under the cross-diagonal symmetry mentioned.
In the second figure, we show the set of unresolved gambles---hatched, because it is empty for maximal assessments.

\subsection{Summary of Concepts \& Notation}\label{sec:summary}
This is a good point to provide a summary of the most important concepts and notation introduced up until now, so as to have a convenient reference for quick reminders when reading the remainder of the paper:
\begin{enumerate}[label=\arabic*.]
  \item
    We consider \emph{assessments} $\A$, which are pairs of sets $\Adelim{\A_\acc}{\A_\rej}$ of gambles, where $\A_\acc$ contains \emph{acceptable} gambles and $\A_\rej$ \emph{rejected} gambles.
    The set of all assessments is $\As$.
  \item
    When a gamble is both accepted and rejected, we call it confusing.
    Our first axiom is \emph{No Confusion}~\eqref{eq:no-confusion}: $\A_\confus\defeq\A_\acc\intersection\A_\rej=\emptyset$.
    The set of all assessments without confusion is $\ncAs$.
    (We consistently use blackboard bold instead of plain bold to denote a set of assessments without confusion.)
  \item
    The most important derived statement types are \emph{indifference} ($\A_\indiff\coloneqq\A_\acc\intersection-\A_\acc$), accepting both a gamble and its negation,  and \emph{favourability} ($\A_\fav\coloneqq\A_\acc\intersection-\A_\rej$), accepting a gamble and rejecting its negation.
  \item
    Our second axiom is \emph{Deductive Closure}~\eqref{eq:deduction}: all positive linear combinations of acceptable gambles must be (made) acceptable.
    From an assessment $\A$ we can obtain a \emph{deductively closed} assessment $\D\coloneqq\ext{\Ds}\A\coloneqq\Adelim{\phull\A_\acc}{\A_\rej}$.
    The set of all deductively closed assessments is $\Ds$ and of those without confusion $\ncDs$.
    The set of all assessments that are \emph{deductively closable} without generating confusion is $\dcAs$.
  \item
    Our third axiom is \emph{No Limbo}~\eqref{eq:no-limbo}: all gambles that would cause confusion if made acceptable must be rejected.
    From a deductively closed assessment without confusion $\D$ we can obtain a \emph{model} $\M\coloneqq\ext{\Ms}\D\coloneqq\Adelim[\big]{\D_\acc}{\shull{\D_\rej}\union(\shull{\D_\rej}-\D_\acc)}$.
    The set of all models without confusion is $\ncMs$.
  \item
    Consider a set of assessments $\someAs$.
    The subset of elements that are \emph{maximal} under the ‘at most as \emph{resolved} as’ relation $\subseteq$ is $\maxsomeAs$.
    The subset of elements that $\subseteq$-\emph{dominate} $\A$ is $\someAs_\A$.
  \item
    Typically, a \emph{background model} $\bgM$ in $\ncMs$ is given, which provides a default set of accepted and rejected gambles and which should be combined with a specific assessment $\A$ in $\dcAs$.
    If $\A\union\bgM\in\dcAs$, then $\A$ \emph{respects} $\bgM$ and its \emph{natural extension} is
    $
      \A\reckunion\bgM \coloneqq \cls{\ncMs}(\A\union\bgM)
                                   \coloneqq \Intersection\maxncMs_{\A\union\bgM}
                                   = \ext{\Ms}\group[\big]{\ext{\Ds}(\A\union\bgM)}.
    $
    An assessment is \emph{coherent} if it coincides with its natural extension.
  \item
    Our fourth and last axiom is Indifference to Status Quo~\eqref{eq:indifference-to-status-quo}: $0\in\bgM_\acc$, the zero gamble is acceptable.
    This assumption simplifies many expressions in the accept-reject framework.
\end{enumerate}

\section{Gamble Relations}\label{sec:relations}
We associate a number of gamble relations on $\linGs$ with each model with Indifference to Status Quo in our accept-reject framework.
So, fix a model~$\M$ in~$\ncMs_\zeroM$ and consider the following defining equivalences:
\begin{equation}\label{eq:gamble-rel-def}
  f\acc g \iff f-g\in\M_\acc
  \quad\text{ and }\quad
  f\rej g \iff f-g\in\M_\rej.
\end{equation}
The former can be read as `$f$ is \demph{accepted} in exchange for $g$', the latter as `$f$ is \demph{unpreferred} to $g$'.
Please not that, despite what may be suggested by the notation, $f\rej g$ on its own does \emph{not} imply $g\acc  f$.

The nature of these gamble relations follows from the axioms of the accept-reject framework: No Confusion~\eqref{eq:no-confusion}, Deductive Closure~\eqref{eq:deduction}, No Limbo~\eqref{eq:no-limbo}, and Indifference to Status Quo~\eqref{eq:indifference-to-status-quo}.
We give a translation of these axioms for gamble relations under the form of a characterisation in the vein of Theorem~\ref{thm:char-AR}:
\begin{theorem}\label{thm:char-AD}
  Gamble relations $\acc$ and $\rej$ on $\linGs$ are equivalent under Definition~\eqref{eq:gamble-rel-def} to a model $\M$ in~$\ncMs_\zeroM$, so without confusion and with indifference to status quo, if and only if for all gambles $f$, $g$, and~$h$ in~$\linGs$ and all mixture coefficients $0<\mu\leq1$ it holds that
  \begin{enumerate}[label=\upshape{(AD\arabic*)},leftmargin=*,widest=AD0,series=AD]
    \item\label{item:AD-accrefl} Accept Reflexivity: $f\acc f$,
    \item\label{item:AD-rejirrefl} Reject Irreflexivity: $f\nrej f$,
    \item\label{item:AD-acctrans}
      Accept Transitivity: $f\acc g \conj g\acc h \then f\acc h$,
    \item\label{item:AD-mixtrans}
      Mixed Transitivity: $f\rej g \conj h\acc g \then f\rej h$, and
    \item\label{item:AD-accmixindep} Accept Mixture Independence:
      $
        f\acc g \iff \mu\cdot f+(1-\mu)\cdot h\acc\mu\cdot g+(1-\mu)\cdot h.
      $
    \item\label{item:AD-rejmixindep} Reject Mixture Independence:
      $
        f\rej g \iff \mu\cdot f+(1-\mu)\cdot h\rej\mu\cdot g+(1-\mu)\cdot h.
      $
  \end{enumerate}
\end{theorem}
\noindent
So acceptability is reflexive and transitive, which makes it a non-strict pre-order, also a vector ordering.
Unpreference is irreflexive.
Both gamble relations are linked together by Mixed Transitivity.

The two definitions of Equation~\eqref{eq:gamble-rel-def} engender three other useful gamble relations:
We say that the agent is \demph{indifferent} between two gambles~$f$ and~$g$ if he accepts~$f$ in exchange for~$g$ and vice versa:
\begin{equation}\label{eq:gamble-rel-indiff}
  f\indiff g \iff f\acc g \conj g\acc f \iff f-g\in\M_\indiff.
\end{equation}
Because it is the symmetrisation of the reflexive and transitive acceptability~$\acc$, indifference is reflexive, transitive, and symmetric, which makes it an equivalence relation.

We say that the agent \demph{prefers} a gamble~$f$ over a gamble~$g$ if he both accepts $f$ in exchange for $g$ and unprefers $g$ to~$f$:
\begin{equation}\label{eq:gamble-rel-fav}
  f\fav g \iff f\acc g \conj g\rej f \iff f-g\in\M_\fav.
\end{equation}
Because of how it is derived from unpreference~$\rej$ and acceptability~$\acc$, and because of Proposition~\ref{prop:model+hulls} and Corollary~\ref{cor:sweetened-deals}, preference satisfies the following properties for all $f$, $g$, and~$h$ in~$\linGs$ and $0<\mu\leq1$:
\begin{alignat}{2}
  \text{Weakening: }&\quad& f\fav g \then& f\acc g,\label{eq:weakening}\\
  \text{Favour Irreflexivity: }&\quad& &f\nfav f,\label{eq:fav-irreflexivity}\\
  \text{Favour Transitivity: }&\quad& f\fav g \conj g\fav h \then& f\fav h,\label{eq:fav-transitivity}\\
  \text{Mixed Transitivity: }&\quad& f\fav g \conj g\acc h \then& f\fav h,\label{eq:fav-mixed-transitivity}\\
  \text{Mixture Independence: }&\quad&
    f\fav g \iff& \mu\cdot f+(1-\mu)\cdot h\fav\mu\cdot g+(1-\mu)\cdot h.
    \label{eq:fav-mix-independence}
\end{alignat}
So preference is irreflexive and transitive (and therefore also antisymmetric), which makes it a strict partial ordering.
This, together with the interpretation attached to the accept and reject type statements from which it derives, makes it ideally suited for decision making.

We say that two gambles~$f$ and~$g$ are \demph{uncomparable} when neither is accepted in exchange for the other:
\begin{equation}\label{eq:gamble-rel-incomp}
  f\incomp g \iff f-g\in\M_\incomp.
\end{equation}
Uncomparability is by definition symmetric, but is irreflexive because of Accept Reflexivity.
Moreover, in general it will not be transitive.
We wish to stress that uncomparability of two gambles corresponds to the absence of an acceptability relation between them.
So, for example, two gambles that are mutually unpreferred are uncomparable, but arguably in a different way than two gambles between which there is no relationship whatsoever, as in this latter case uncomparability is caused by a lack of statements instead of by the nature of the applicable statements.

For all the gamble relations introduced above, denote them generically by $\mathrel{\Box}$, the following property follows from Mixture Independence with $h=-g$ and $\mu=\frac12$:
\begin{equation}\label{eq:cancellation}
  \text{Cancellation: }\quad
    f\mathrel{\Box}g\iff f-g\mathrel{\Box}0.
\end{equation}
This property can be considered as a conceptual intermediate step when moving between gamble relations and models, using Equations~\eqref{eq:gamble-rel-def}, \eqref{eq:gamble-rel-indiff}, \eqref{eq:gamble-rel-fav}, and \eqref{eq:gamble-rel-incomp}.

\begin{wrapfigure}{r}{0pt}
  \shortstack[c]{
    \begin{tikzpicture}[scale=1.9]
      \draw[->] (-2.2,0) coordinate (xl) -- (2.2,0) coordinate (xu);
      \draw[->] (0,-1) coordinate (yl) -- (0,1) coordinate (yu);
      \coordinate (a1) at (30:.7);
      \coordinate (a2) at (135:.7);
      \coordinate (r1) at (60:.7);
      \coordinate (r2) at (163:.7);
      \begin{pgfonlayer}{background}
        \draw[border] (0,0) -- (intersection of 0,0--a1 and yu--{xu|-yu}) coordinate (a1away);
        \draw[border] (0,0) -- (intersection of 0,0--a2 and yu--{xu|-yu}) coordinate (a2away);
        \fill[lightgray] (0,0) -- (a1away) |- (a2away) --cycle;
        \draw[border] (0,0) -- (intersection of 0,0--r2 and xl--{xl|-yu}) coordinate (r2away);
        \draw[exclborder] (0,0) -- (intersection of 0,0--a2 and yl--{xu|-yl}) coordinate (a2anotherway);
        \fill[gray] (0,0) -- (a2anotherway) |- (xl|-yl) |- (r2away) --cycle;
        \draw[favborder] (0,0) -- (a1away);
      \end{pgfonlayer}
      \coordinate[vec,label={[xshift=-6pt]above right:$f-2\cdot g$}] (fm2g) at (intersection of a2--{a2 -| 0,0} and 0,0--a2);
      \coordinate[vec,label={[white,xshift=6pt]below left:$2\cdot g-f$}] (mfm2g) at ($-1*(fm2g)$);
      \coordinate[vec,label=right:$f$,xshift=2ex] (f) at (intersection of a2--{a2 -| 0,0} and 0,0--a1);
      \coordinate[white,vec,label={[white]left:$-f$}] (mf) at ($-1*(f)$);
      \coordinate[vec,label={[xshift=-6pt]above right:$f-g$}] (fm1g) at ($(f)!.5!(fm2g)$);
      \coordinate[white,vec,label={[white,xshift=6pt]below left:$g-f$}] (mfm1g) at ($-1*(fm1g)$);
      \coordinate[vec,label=below:$g$] (g) at ($.5*(f)-.5*(fm2g)$);
      \coordinate[white,vec,label={[white]below:$-g$}] (mg) at ($-1*(g)$);
      \coordinate[vec,label={[xshift=-6pt]above right:$f-3\cdot g$}] (fm3g) at ($(fm2g)-(g)$);
      \coordinate[vec,label={[xshift=6pt]below left:$3\cdot g-f$}] (mfm3g) at ($-1*(fm3g)$);
      \coordinate[white,vec,label={[white,xshift=-6pt]below right:$f-4\cdot g$}] (fm4g) at ($(fm3g)-(g)$);
      \coordinate[vec,label={[xshift=6pt]below left:$4\cdot g-f$}] (mfm4g) at ($-1*(fm4g)$);
    \end{tikzpicture}
  \\[1ex]
    \begin{tikzpicture}[scale=1.9]
      \draw[->] (-2.2,0) coordinate (xl) -- (2.2,0) coordinate (xu);
      \draw[->] (0,-1) coordinate (yl) -- (0,1) coordinate (yu);
      \coordinate (a1) at (150:.6);
      \begin{pgfonlayer}{background}
        \draw[border]
          (intersection of 0,0--a1 and yu--{xu|-yu}) coordinate (a1away1) -- (intersection of 0,0--a1 and yl--{xu|-yl}) coordinate (a1away2);
        \fill[lightgray]  (a1away1) -- (a1away2) -| (yl -| xu) -- (yu -| xu) --cycle;
      \end{pgfonlayer}
      \coordinate[vec,label=above:$f'$,xshift=5ex] (f) at (intersection of a1away1--a1away2 and 0,.4--1,.4);
      \coordinate[vec,label=right:$g'$,xshift=5ex] (g) at (intersection of a1away1--a1away2 and 0,-.3--1,-.3);
      \coordinate[vec,label=below left:$f'-g'$] (fmg) at ($(f)-1*(g)$);
      \coordinate[vec,label=below left:$g'-f'$] (gmf) at ($(g)-1*(f)$);
    \end{tikzpicture}
  }
\end{wrapfigure}
The connection between models and the gamble relations of this section is illustrated on the right.
In the top figure, we have that
\begin{description}
  \item[$f\fav g$:]
    $f$ is preferred to $g$, because $f-g$ is acceptable and $g-f$ is rejected;
  \item[$f\acc2\cdot g$:]
    $f$ is accepted in exchange for ${2\cdot g}$, because $f-2\cdot g$ is acceptable and $2\cdot g-f$ is unresolved;
  \item[$f\incomp 3\cdot g$:]
    $f$ and $3\cdot g$ are uncomparable, because both $f-3\cdot g$ and $3\cdot g-f$ are unresolved;
  \item[$f\rej 4\cdot g$:]
    $f$ is unpreferred to $4\cdot g$, because $f-4\cdot g$ is rejected and $4\cdot g-f$ is unresolved;
  \item[$f\rej0$, $g\rej0$:]
    both $f$ and $g$ are rejected.
\end{description}
In the bottom figure we have that
\begin{description}
  \item[$f'\indiff g'$:]
    indifference between $f'$ and $g'$, because both $f'-g'$ and $g'-f'$ are acceptable;
  \item[$f'\acc0$, $g'\acc0$:]
    both $f'$ and $g'$ are acceptable.
\end{description}

Let us denote the gamble relations corresponding to a background model~$\bgM$ in~$\ncMs_\zeroM$ by $\sacc$ for acceptability and~$\srej$ for unpreference.
This provides a baseline assessment that the agent can augment, resulting---if all goes well---in a model~$\M$ in~$\ncMs_\bgM$ with which, as above, we associate the gamble relations~$\acc$ and~$\rej$.
How should such a background gamble relation pair $(\sacc,\srej)$ impact the agent's gamble relation pair $(\acc,\rej)$?
\begin{proposition}\label{prop:char-ADbgM}
  Given pairs of gamble relations $(\sacc,\srej)$ and $(\acc,\rej)$ on $\linGs$ equivalent under Definition~\eqref{eq:gamble-rel-def} respectively to models with indifference to status quo $\bgM$ and $\M$ in~$\ncMs_\zeroM$, then~$\M$ is $\bgM$-coherent if and only if for all gambles $f$ and $g$ in~$\linGs$ it holds that
  \begin{enumerate}[AD]
    \item\label{item:AD-monoton}
      Monotonicity: $f\sacc g \then f\acc g$ and $f\srej g \then f\rej g$.
  \end{enumerate}
\end{proposition}
\noindent
For example, when using $\bgM\defeq\Adelim{\linGs_\geq}{\linGs_<}$, we have ${\sacc}={\geq}$ and ${\srej}={<}$, which supports the name `Monotonicity' even more than the suggestive notation used.

\section{Simplified Frameworks}\label{sec:simplified-frameworks}
The accept-reject framework of Section~\ref{sec:framework} may in many situations be more general than needed.
Therefore, it is interesting to have a look at simplified versions of this framework.
By `simplified', we mean that we add additional restrictions on the statements the agent is allowed to make, so that the models that result become easier to work with or characterise.

The most simplified frameworks we consider here essentially restrict assessments in terms of either favourability or acceptability statements.
They allow us to make the connection with other frameworks in the literature for modelling uncertainty that are also based on statements about gambles.

\pagebreak
\subsection{The Accept-Favour Framework}
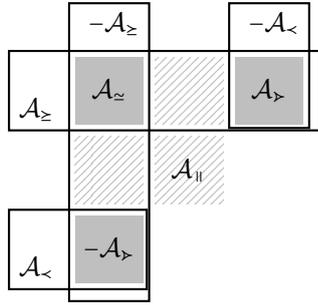
\label{sec:accept-favour}
\begin{wrapfigure}[10]{l}{0pt}
  \begin{tikzpicture}
    \matrix [nodes={regular polygon, regular polygon sides=4,inner sep=-1ex,minimum size=8ex,outer sep=.5ex},column sep={1ex,between borders},row sep={1ex,between borders}] {
      \node[fill=lightgray] (apam) {$\A_\indiff$}; & \node[pattern=north east lines,pattern color=lightgray] (ap) {}; & \node[fill=lightgray] (aprm) {$\A_\fav$}; \\
      \node[pattern=north east lines,pattern color=lightgray] (am) {}; & \node[pattern=north east lines,pattern color=lightgray] (none) {$\A_\incomp$}; & \node (rm) {}; \\
      \node[fill=lightgray] (rpam) {$-\A_\fav$}; & \node (rp) {}; & \node (rprm) {}; \\
    };
    \draw[thick] ([xshift=1ex]aprm.north east) rectangle ([xshift=-5ex]apam.south west) node[above right] {$\A_\acc$};
    \draw[thick] ([yshift=-1ex]rpam.south west) rectangle ([yshift=4ex]apam.north east) node[below left] {$-\A_\acc$};
    \draw[thick] ([xshift=-5ex]rpam.south west) node[above right] {$\A_\rej$} rectangle ([xshift=-.2ex]rpam.north east);
    \draw[thick] ([yshift=.2ex]aprm.south west) rectangle ([yshift=4ex]aprm.north east) node[below left] {$-\A_\rej$};
  \end{tikzpicture}
\end{wrapfigure}
On the left we give the illustration of the six-element partition that results if we simplify our framework by restricting reject statements to negated acceptable gambles by imposing
\begin{condition}{AF}\label{eq:accfav-condition}
  $-\A_\rej\subseteq\A_\acc$.
\end{condition}
\noindent
In such a context, the rejection of a gamble~$f$ in~$\linGs$ can be viewed as an explicit statement of favourability about the gamble's negation~$-f$, because then $\A_\fav=\A_\acc\intersection-\A_\rej=-\A_\rej$.
It is therefore immaterial whether we specify an assessment by providing the sets $\A_\acc$ and $\A_\rej$, or by the sets $\A_\acc$ and~$\A_\fav$; in this situation we say we are using the \emph{accept-favour framework}.

We should not impose Condition~\eqref{eq:accfav-condition} as a universal rationality requirement.
For example, it would preclude one from rejecting all gambles with maximum winnings that are not at least twice the maximum losses without creating confusion; e.g., in a two-dimensional context $\set{(-1,1),(1,-1)}\subset\A_\confus$.

Given a set of assessments $\someAs\subseteq\As$, we define its subset of assessments satisfying Condition~\eqref{eq:accfav-condition} by $\aff{\someAs}\defeq\cset{\A\in\someAs}{-\A_\rej\subseteq\A_\acc}$.
The basic results of Sections~\ref{sec:no-confusion} to~\ref{sec:nolimbo} about assessments, deductively closed assessments, and models of course also remain valid when restricting attention to the accept-favour assessments in~$\aff{\As}$.
It is useful to state a more specific version of our criterion for deductive closability:
\begin{theorem}[cf. Theorem~\ref{thm:dcAs-zerocrit}]\label{thm:affdcAs-zerocrit}
  An accept-favour assessment~$\A$ in~$\aff\As$ is deductively closable---i.e., $\A\in\affdcAs$---if and only if\/ $0\notin\A_\fav+\phull\A_\acc$.
\end{theorem}

We furthermore wish the results of Sections~\ref{sec:order} to~\ref{sec:background} about intersection structures, maximal elements, and closure operators to have counterparts in the accept-favour framework.
The following four results take care of the basics:
\begin{proposition}[cf. Proposition~\ref{prop:intersection}]\label{prop:accfav-intersection}
  Condition~\eqref{eq:accfav-condition} is preserved under arbitrary non-empty intersections:
  If $(\someAs,\subseteq)$, with $\someAs\subseteq\As$, is an intersection structure, then so is $(\aff{\someAs},\subseteq)$.
\end{proposition}
\begin{proposition}[cf. Proposition~\ref{prop:maximalncasss}]\label{prop:accfav-maximalncasss}
  The set of maximal accept-favour assessments is
  \begin{equation*}
    \aff\maxncAs
    = \aff\ncAs\intersection\maxncAs
    = \cset{\Adelim{\someGs}{\linGs\setminus\someGs}}
          {\someGs\subseteq\linGs
            \conj -(\linGs\setminus\someGs)\subseteq\someGs},
  \end{equation*}
  and every such assessment $\A$ in~$\aff\maxncAs$ satisfies indifference to status quo: $\zeroM\subseteq\A$.
\end{proposition}
\begin{proposition}[cf. Proposition~\ref{prop:maximalncmods}]\label{prop:accfav-maximalncmods}
  The set of all maximal accept-favour models is\footnote{
    The Axiom of Choice is assumed for infinite-dimensional~$\linGs$ to prove the existence of the cones~$\someGs$ (cf. Lemma~\ref{lem:kakutani-for-cones} in the Proofs Appendix).
  }
  \begin{equation*}
    \aff\maxncMs
    = \aff\maxncAs\intersection\aff\ncDs
    = \cset{\Adelim{\someGs}{\linGs\setminus\someGs}}
          {\someGs\in\cones
            \conj -(\linGs\setminus\someGs)\subseteq\someGs}.
  \end{equation*}
  It moreover coincides with the sets of all maximal deductively closable or closed accept-favour assessments without confusion: $\aff\maxncMs=\aff\maxncDs=\aff\maxdcAs$.
\end{proposition}
\begin{proposition}[cf. Propositions~\ref{prop:closure} and~\ref{prop:cls-formulae}]\label{prop:accfav-cls}
  Given an intersection structure $(\someAs,\subseteq)$ with $\someAs$ in $\set{\As,\ncAs,\dcAs,\Ds,\ncDs,\ncMs}$, then the closure operator on the intersection structure $(\aff\someAs,\subseteq)$ coincides with the one on $(\someAs,\subseteq)$: $\cls{\aff{\someAs}}=\cls{\someAs}$ on $\aff{\As}$.
\end{proposition}

We also consider the set of maximal models~$\aff{\maxncMs}_\A\defeq\aff{\maxncMs}\intersection\aff{\ncMs}_\A$ that dominate the assessment~$\A$ in~$\aff{\As}$ and satisfy Condition~\eqref{eq:accfav-condition}.
This leads to a variant of our alternative characterisation of deductive closability:
\begin{theorem}[cf. Theorem~\ref{thm:dedclosable-nonemptymax}]\label{thm:accfav-dedclosable-nonemptymax}
  An accept-favour assessment $\A$ in $\aff\As$ is deductively closable if and only if its set of maximal dominating accept-favour models is non-empty.
  Formally: $\A\in\affdcAs$ if and only if\/ $\aff\maxncMs_\A\neq\emptyset$.
\end{theorem}

We would also like to prove that here too, the maximal models dominating an assessment can be used for inference purposes, in a result similar to Proposition~\ref{prop:inf-by-max}.
But since in the accept-favour framework every maximal model satisfies Indifference to Status Quo by Proposition~\ref{prop:accfav-maximalncasss}, no model without Indifference to Status Quo can ever be the intersection of the maximal models in $\aff\maxncMs$ that dominate it.
Therefore our variant restricts attention to assessments with indifference to status quo:
\begin{proposition}[cf. Proposition~\ref{prop:inf-by-max}]\label{prop:accfav-inf-by-max}
  Given an accept-favour assessment with indifference to status quo $\A$ in $\aff{\As}_\zeroM$, then its closure in the set of accept-favour models is $\cls{\aff\ncMs}\A=\cls{\ncMs}\A=\Intersection\aff{\maxncMs}_\A$.
\end{proposition}

Again we can give a compact characterisation of models coherent with a background model that is indifferent to status quo.
But now, in the accept-favour framework, we can refocus attention from rejected to favoured gambles.
\begin{theorem}[cf. Theorem~\ref{thm:char-AR}]\label{thm:char-AF}
  An accept-favour assessment~$\M$ in~$\aff\As$ is a model that respects an accept-favour background model with indifference to status quo $\bgM$ in~$\aff\ncMs_\zeroM$, i.e., $\M\in\aff\ncMs_\bgM$, if and only if
  \begin{enumerate}[label=\upshape{(AF\arabic*)},leftmargin=*,widest=AF0]
    \item\label{item:AFbg} it includes the background model: $\bgM\subseteq\M$,
    \item\label{item:AFzr} it does not favour status quo: $0\notin\M_\fav$,
    \item\label{item:AFcn}
      both its acceptable and favourable gambles form a cone: $\M_\acc,\M_\fav\in\cones$, and
    \item\label{item:AFss}
      it favours sweetened deals: $\M_\acc+\M_\fav\subseteq\M_\fav$.
  \end{enumerate}
\end{theorem}

It is possible to consider the `dual' of Condition~\eqref{eq:accfav-condition}, $-\A_\acc\subseteq\A_\rej$, i.e., restrict accept statements to negated rejected gambles.
Then any acceptable gamble is also favourable: $\A_\fav=\A_\acc$; we end up with a \emph{reject-favour framework}.
This framework is incompatible with Indifference to Status Quo in the sense that their combination would lead to confusion.
So status quo would effectively be in limbo, leading us to reject it by default within this framework.
We see no compelling reason to exclude the reject-favour framework from consideration, but we do not investigate it here further, focussing rather on models compatible with Indifference to Status Quo.

\subsection{The Favour-Indifference Framework}
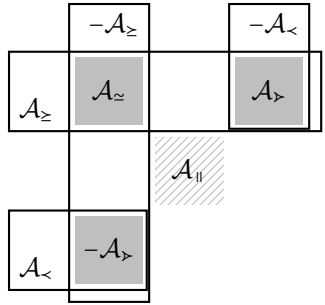
\label{sec:favindiff}
\begin{wrapfigure}[10]{l}{0pt}
  \begin{tikzpicture}
    \matrix [nodes={regular polygon, regular polygon sides=4,inner sep=-1ex,minimum size=8ex,outer sep=.5ex},column sep={1ex,between borders},row sep={1ex,between borders}] {
      \node[fill=lightgray] (apam) {$\A_\indiff$}; & \node (ap) {}; & \node[fill=lightgray] (aprm) {$\A_\fav$}; \\
      \node (am) {}; & \node[pattern=north east lines,pattern color=lightgray] (none) {$\A_\incomp$}; & \node (rm) {}; \\
      \node[fill=lightgray] (rpam) {$-\A_\fav$}; & \node (rp) {}; & \node (rprm) {}; \\
    };
    \draw[thick] ([xshift=1ex]aprm.north east) rectangle ([xshift=-5ex]apam.south west) node[above right] {$\A_\acc$};
    \draw[thick] ([yshift=-1ex]rpam.south west) rectangle ([yshift=4ex]apam.north east) node[below left] {$-\A_\acc$};
    \draw[thick] ([xshift=-5ex]rpam.south west) node[above right] {$\A_\rej$} rectangle ([xshift=-.2ex]rpam.north east);
    \draw[thick] ([yshift=.2ex]aprm.south west) rectangle ([yshift=4ex]aprm.north east) node[below left] {$-\A_\rej$};
  \end{tikzpicture}
\end{wrapfigure}
On the left we give the illustration of the four-element partition that results if we further simplify the accept-favour framework by restricting accept statements to either favourability statements or indifference statements by imposing, in addition to Condition~\eqref{eq:accfav-condition} that:
\begin{condition}{FI}\label{eq:favindiff-condition}
  $\A_\acc=\A_\fav\union\A_\indiff$.
\end{condition}
\noindent
In such a context, all unresolved gambles have an unresolved negation, because indifference and favourability statements both say something about a gamble~$f$ and its negation~$-f$ concurrently.
It is therefore immaterial whether we specify an assessment by providing the sets $\A_\acc$ and $\A_\rej$, or by the sets $\A_\fav$ and~$\A_\indiff$; in this situation we say we are using the \emph{favour-indifference framework}.

Given $\someAs\subseteq\aff\As$, then define $\fif\someAs\defeq\cset{\A\in\someAs}{-\A_\rej\subseteq\A_\acc\conj\A_\acc=\A_\fav\union\A_\indiff}\subseteq\aff\someAs$ as its subset of assessments satisfying Conditions~\eqref{eq:accfav-condition} and~\eqref{eq:favindiff-condition}.
Again, the results of Sections~\ref{sec:no-confusion} to~\ref{sec:nolimbo} about assessments, deductively closed assessments, and models also remain valid when restricting attention to~$\fif{\As}$.
It is useful to state a more specific version of our criterion for deductive closability:
\begin{theorem}[cf. Theorem~\ref{thm:affdcAs-zerocrit}]\label{thm:fifdcAs-zerocrit}
  A favour-indifference assessment~$\A$ in~$\fif\As$ is deductively closable---i.e., $\A\in\fifdcAs$---if and only if\/ $0\notin\phull\A_\fav+\lhull\A_\indiff$.
\end{theorem}

\begin{wrapfigure}{r}{0pt}
  $
    \begin{tikzpicture}[scale=1.5,baseline=-.7ex]
      \draw[->] (-.7,0) coordinate (xl) -- (1,0) coordinate (xu);
      \draw[->] (0,-.7) coordinate (yl) -- (0,1) coordinate (yu);
      \node (a1) at (150:.7) {};
      \begin{pgfonlayer}{background}
        \draw[border] (intersection of 0,0--a1 and xl--{xl|-yu}) coordinate (a1away) -- (intersection of 0,0--a1 and xu--{yl-|xu}) coordinate (a1away2);
        \fill[lightgray] (a1away) |- (xu|-yu) |- (a1away2) --cycle;
        \fill[gray] (a1away) |- (xl|-yl) -| (a1away2) --cycle;
      \end{pgfonlayer}
    \end{tikzpicture}
    \intersection
    \begin{tikzpicture}[scale=1.5,baseline=-.7ex]
      \draw[->] (-.7,0) coordinate (xl) -- (1,0) coordinate (xu);
      \draw[->] (0,-.7) coordinate (yl) -- (0,1) coordinate (yu);
      \node (a2) at (120:.7) {};
      \begin{pgfonlayer}{background}
        \draw[border] (intersection of 0,0--a2 and yu--{xl|-yu}) coordinate (a2away) -- (intersection of 0,0--a2 and yl--{yl-|xu}) coordinate (a2away2);
        \fill[lightgray] (a2away) |- (xu|-yu) |- (a2away2) --cycle;
        \fill[gray] (a2away) -| (xl) |- (a2away2) --cycle;
      \end{pgfonlayer}
    \end{tikzpicture}
    =
    \begin{tikzpicture}[scale=1.5,baseline=-.7ex]
      \draw[->] (-.7,0) coordinate (xl) -- (1,0) coordinate (xu);
      \draw[->] (0,-.7) coordinate (yl) -- (0,1) coordinate (yu);
      \node (a1) at (150:.7) {};
      \node (a2) at (120:.7) {};
      \begin{pgfonlayer}{background}
        \path (intersection of 0,0--a1 and xl--{xl|-yu}) coordinate (a1away) -- (intersection of 0,0--a1 and xu--{yl-|xu}) coordinate (a1away2);
        \path (intersection of 0,0--a2 and yu--{xl|-yu}) coordinate (a2away) --   (intersection of 0,0--a2 and yl--{yl-|xu}) coordinate (a2away2);
        \draw[border] (a1away2) -- (0,0) -- (a2away);
        \draw[exclborder] (a2away2) -- (0,0) -- (a1away);
        \fill[lightgray] (0,0) -- (a2away) |- (xu|-yu) |- (a1away2) --cycle;
        \fill[gray] (0,0) -- (a1away) |- (xl|-yl) -| (a2away2) --cycle;
      \end{pgfonlayer}
    \end{tikzpicture}
  $
\end{wrapfigure}
The poset $(\fif\As,\subseteq)$ is not an intersection structure.
This can be shown using the graphical counterexample on the right.
The resulting intersection---of two elements of $\fif\ncMs\subseteq\fif\As$ with topologically closed sets of acceptable gambles---does not satisfy Condition~\eqref{eq:favindiff-condition} any more: it has (border) gambles that are acceptable without being either favourable or indifferent.
However, by investigating the effect of~$\cls{\ncMs}$, some interesting conclusions can still be drawn:
\begin{theorem}\label{thm:favindiff}
  Given a favour-indifference assessment~$\A$ in~$\fifdcAs\!\vphantom{\dcAs}_\zeroM$, so without confusion and with indifference to status quo, then  the corresponding model $\M\defeq\cls{\ncMs}\A$ is the element of $\fif\ncMs_\zeroM$ with defining components $\M_\indiff=\lhull\A_\indiff$ and $\M_\fav=\phull\A_\fav+\lhull\A_\indiff$.
\end{theorem}

As we did in the accept-favour framework, it is possible to give a compact characterisation of models coherent with a background model that is indifferent to status quo.
But now, in the favour-indifference framework, we can refocus attention from accepted to indifferent gambles.
\begin{theorem}[cf. Theorem~\ref{thm:char-AF}]\label{thm:char-FI}
  A favour-indifference assessment~$\M$ in~$\fif\As$ is a model that respects a favour-indifference background model with indifference to status quo $\bgM$ in~$\fif\ncMs_\zeroM$, i.e., $\M\in\fif\ncMs_\bgM$, if and only if
  \begin{enumerate}[label=\upshape{(FI\arabic*)},leftmargin=*,widest=FI0]
    \item\label{item:FIbg} it includes the background model: $\bgM\subseteq\M$,
    \item\label{item:FIzr} it does not favour status quo: $0\notin\M_\fav$,
    \item\label{item:FIcn}
      its favourable gambles form a cone and its indifferent gambles form a linear space: $\M_\fav\in\cones$ and $\M_\indiff\in\lineals$, and
    \item\label{item:FIss}
      it favours sweetened deals: $\M_\fav+\M_\indiff\subseteq\M_\fav$.
  \end{enumerate}
\end{theorem}

\subsection{The Favourability Framework \& its Appearance in the Literature}\label{sec:favourability}
To work towards types of models encountered in the literature, we look at the special case where the agent only makes favourability statements forming a set $\A_\fav\subset\linGs$, but where there is a favour-indifference background model~$\bgM$ in~$\fif\ncMs_\zeroM$, so one that satisfies \ref{item:FIbg}--\ref{item:FIss}.
The resulting \emph{favourability framework} is a restriction of the favour-indifference framework.
To know what the models in this framework look like, we specialise Theorems~\ref{thm:fifdcAs-zerocrit} and~\ref{thm:favindiff} to obtain a specialised criterion for deductive closability and an explicit expression for natural extension:
\begin{theorem}[cf. Theorems~\ref{thm:fifdcAs-zerocrit} and~\ref{thm:favindiff}]\label{thm:ffdcAs-zerocrit}
  Given a set of favourable gambles $\A_\fav\subseteq\linGs$ and a favour-indifference background model with indifference to status quo $\bgM$ in~$\fif\ncMs_\zeroM$, then the favourability assessment $\A\defeq\Adelim{\A_\fav}{-\A_\fav}$ respects the background model~$\bgM$ if and only if\/ $0\notin\bgM_\indiff+\phull(\bgM_\fav\union\A_\fav)$.
  In that case, the natural extension $\M\defeq\A\reckunion\bgM\in\fif\ncMs_\zeroM$ has defining components $\M_\indiff\defeq\bgM_\indiff$ and $\M_\fav\defeq\bgM_\indiff+\phull(\bgM_\fav\union\A_\fav)$, so the resulting model is $\M=\Adelim{\bgM_\indiff\union\M_\fav}{-\M_\fav}$.
\end{theorem}

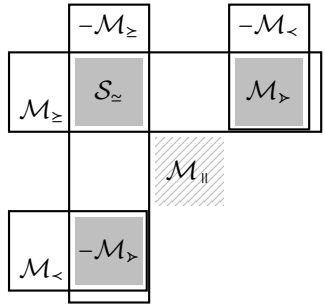
\begin{wrapfigure}[10]{l}{0pt}
  \begin{tikzpicture}
    \matrix [nodes={regular polygon, regular polygon sides=4,inner sep=-1ex,minimum size=8ex,outer sep=.5ex},column sep={1ex,between borders},row sep={1ex,between borders}] {
      \node[fill=lightgray] (apam) {$\bgM_\indiff$}; & \node[white] (ap) {}; & \node[fill=lightgray] (aprm) {$\M_\fav$}; \\
      \node[white] (am) {}; & \node[pattern=north east lines,pattern color=lightgray] (none) {$\M_\incomp$}; & \node (rm) {}; \\
      \node[fill=lightgray] (rpam) {$-\M_\fav$}; & \node (rp) {}; & \node (rprm) {}; \\
    };
    \draw[thick] ([xshift=1ex]aprm.north east) rectangle ([xshift=-5ex]apam.south west) node[above right] {$\M_\acc$};
    \draw[thick] ([yshift=-1ex]rpam.south west) rectangle ([yshift=4ex]apam.north east) node[below left] {$-\M_\acc$};
    \draw[thick] ([xshift=-5ex]rpam.south west) node[above right] {$\M_\rej$} rectangle ([xshift=-.2ex]rpam.north east);
    \draw[thick] ([yshift=.2ex]aprm.south west) rectangle ([yshift=4ex]aprm.north east) node[below left] {$-\M_\rej$};
  \end{tikzpicture}
\end{wrapfigure}
The illustration of the four-element partition of gamble space corresponding to a model~$\M$ in the favourability framework with background model~$\bgM$ is given on the left.
The reason for working with a model here when describing the partition instead of with a general assessment, as in the previous section, is that it highlights the specific difference with favour-indifference models---namely that $\M_\indiff=\bgM_\indiff$.

The focus in this framework's characterisation result lies on the set of favourable gambles:
\begin{theorem}[cf.~Theorem~\ref{thm:char-FI}]\label{thm:char-F}
  Given a set of favourable gambles $\M_\fav\subseteq\linGs$ and a favour-indifference background model with indifference to status quo $\bgM$ in~$\fif\ncMs_\zeroM$, then the assessment $\M\defeq\Adelim{\bgM_\indiff\union\M_\fav}{-\M_\fav}$ is a model, i.e., $\M\in\fif\ncMs_\bgM$, if and only if
  \begin{enumerate}[label=\upshape{(F\arabic*)},leftmargin=*,widest=F0]
    \item\label{item:Fbg}
      it favours background-favourable gambles: $\bgM_\fav\subseteq\M_\fav$,
    \item\label{item:Fzr} it does not favour status quo: $0\notin\M_\fav$,
    \item\label{item:Fcn} its favourable gambles form a cone: $\M_\fav\in\cones$, and
    \item\label{item:Fss}
      it favours sweetened deals: $\M_\fav+\bgM_\indiff\subseteq\M_\fav$.
  \end{enumerate}
\end{theorem}

The simplest case occurs when we take $0$ to be the single background indifferent gamble: $\bgM_\indiff=\set{0}$; the background favourable gambles $\bgM_\fav$ should then by \ref{item:FIbg}--\ref{item:FIss} be a convex cone that does not contain the zero gamble.
This case results in \ref{item:Fss} being trivially satisfied; \ref{item:Fbg}--\ref{item:Fcn} then reduce to the conditions for what \citet{DeCooman-Quaeghebeur-2012-Kyburg} have called coherence relative to the convex cone $\bgM_\fav$.
The notion of Bernstein coherence of a set of polynomials that they discuss there, is an interesting and useful special case.
Elsewhere in the literature, only the case $\bgM_\fav=\linGs_>$ is considered.
We give a brief chronological overview, but do give an axiom-for-axiom correspondence with the most general of these frameworks:
\begin{itemize}
  \item
    \Citet[§14]{Smith-1961} talks about an open cone of `exchange vectors' (he works with a finite possibility space~$\pspace$ and the linear space of all gambles $\linGs=\Gs{\pspace}$); his notion of preference fits in the favourability framework.
    Furthermore, he imposes that~$\M_\fav$ should be an open set.
  \item
    \Citet[§IV]{Seidenfeld-Schervish-Kadane-1990-decwoord} talk about ‘favorable’ gambles (again on a finite possibility space~$\pspace$ and with the linear space of all gambles $\linGs=\Gs{\pspace}$); their work fits right in the favourability framework.
  \item
    \Citet[§3.7.8]{Walley-1991} discusses ‘strictly desirable’ gambles (with $\linGs$ any linear space containing constant gambles).
    An openness axiom is added for prevision-equivalence (cf. Section~\ref{sec:previsions}): $\M_\fav\setminus\bgM_\fav\subseteq\M_\fav+\reals_{>0}$.
  \item
    \Citet[§6]{Walley-2000-towards} advocates a desirability framework that is more elaborately discussed, but essentially equivalent to \citeauthor{Seidenfeld-Schervish-Kadane-1990-decwoord}'s \cite{Seidenfeld-Schervish-Kadane-1990-decwoord} `favorable' gambles, but without the focus on finite possibility spaces~$\pspace$.
    This framework also corresponds to his `strictly desirable' gambles without the extra openness axiom.

    Let us take take a slightly more detailed look at this framework by exhibiting the correspondence between \Citeauthor{Walley-2000-towards}'s axioms and ours, as formulated in Theorem~\ref{thm:char-F}:
    He calls a set of desirable gambles $\M_\fav\subseteq\linGs$ coherent if and only if for all gambles $f$ and $g$ in $\linGs$ it satisfies
    \begin{enumerate}[label=\upshape{(D\arabic*)},leftmargin=*,widest=D0]
      \item\label{item:Dzr} $0\notin\M_\fav$,
      \item\label{item:Dbg} if $f>0$, then $f\in\M_\fav$,
      \item\label{item:Dps}
        if $f\in\M_\fav$, then $\lambda\cdot f\in\M_\fav$ for all positive scaling factors $\lambda\in\reals_>$,
      \item\label{item:Dcb}
        if $f,g\in\M_\fav$, then $f+g\in\M_\fav$.
    \end{enumerate}
    Recall that we consider \ref{item:Fss} trivially satisfied because our notion of acceptability and thus indifference does not exist in \Citeauthor{Walley-2000-towards}'s framework.
    We furthermore see that \ref{item:Dzr} and \ref{item:Fzr} are identical, as are \ref{item:Dbg} and \ref{item:Fbg} with $\bgM_\fav=\linGs_>$.
    Finally, \ref{item:Dps} corresponds to Positive Scaling~\eqref{eq:scaling} and \ref{item:Dcb} to Combination~\eqref{eq:combination}, and therefore together to Deductive Closure~\eqref{eq:deduction}, i.e., \ref{item:Fcn}.
\end{itemize}

\citet{DeCooman-Quaeghebeur-2012-Kyburg} use and extend the framework of \Citet[§6]{Walley-2000-towards} to study exchangeability.
Actually, we feel the accept-reject framework is a more natural setting for such a study.
Therefore we repeat it in Section~\ref{sec:xch} for the restricted finite exchangeability case \citep[cf.][]{DeCooman-Quaeghebeur-2009-ISIPTA} as an illustrative application, where both $\bgM_\fav$ and $\bgM_\indiff$ are non-trivial.

The unconfused models of a favourability framework with a fixed background model actually form an intersection structure, in contrast to the situation for the general favour-indifference framework of the previous section.
Order-theoretic results and results about maximal elements can be found in the work of \citet{DeCooman-Quaeghebeur-2012-Kyburg} and \citet{2011-Couso+Moral-desir}.

\subsection{The Acceptability Framework \& its Appearance in the Literature}\label{sec:acceptability}
Again to work towards types of models encountered in the literature, we look at the special case where the agent only makes acceptability statements forming a set $\A_\acc\subset\linGs$, but where there is an accept-reject background model~$\bgM$ in~$\ncMs_\zeroM$, so one that satisfies \ref{item:ARbg}--\ref{item:ARss}.
The resulting \emph{acceptability framework} is a restriction of the accept-reject framework.
To know what the models in this framework look like, we derive a result similar to Theorem~\ref{thm:ffdcAs-zerocrit}, with a specialised criterion for deductive closability and an explicit expression for natural extension:
\begin{theorem}\label{thm:afdcAs-zerocrit}
  Given a set of acceptable gambles $\A_\acc\subseteq\linGs$ and a background model with indifference to status quo $\bgM$ in~$\ncMs_\zeroM$, then the acceptability assessment $\A\defeq\Adelim{\A_\acc}{\emptyset}$ respects the background model $\bgM$ if and only if\/ $0\notin\bgM_\rej-\phull(\bgM_\acc\union\A_\acc)$.
  In that case, the natural extension is
  $
    \M \defeq \A\reckunion\bgM
       = \Adelim{\M_\acc}{\bgM_\rej-\M_\acc},
  $
  where $\M_\acc=\phull(\bgM_\acc\union\A_\acc)$ is the resulting set of acceptable gambles.
\end{theorem}

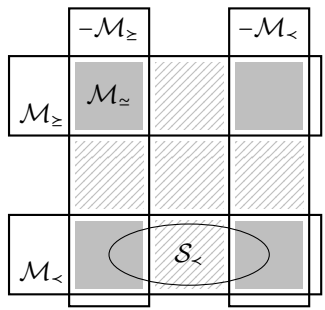
\begin{wrapfigure}[10]{l}{0pt}
  \begin{tikzpicture}
    \matrix [nodes={regular polygon, regular polygon sides=4,pattern=north east lines,pattern color=lightgray,inner sep=-1ex,minimum size=8ex,outer sep=.5ex},column sep={1ex,between borders},row sep={1ex,between borders}] {
      \node[fill=lightgray] (apam) {$\M_\indiff$}; & \node (ap) {}; & \node[fill=lightgray] (aprm) {}; \\
      \node (am) {}; & \node (none) {}; & \node (rm) {}; \\
      \node[fill=lightgray] (rpam) {}; & \node (rp) {}; & \node[fill=lightgray] (rprm) {}; \\
    };
    \draw[thick] ([xshift=1ex]aprm.north east) rectangle ([xshift=-5ex]apam.south west) node[above right] {$\M_\acc$};
    \draw[thick] ([yshift=-1ex]rpam.south west) rectangle ([yshift=4ex]apam.north east) node[below left] {$-\M_\acc$};
    \draw[thick] ([xshift=-5ex]rpam.south west) node[above right] {$\M_\rej$} rectangle ([xshift=1ex]rprm.north east);
    \draw[thick] ([yshift=-1ex]rprm.south west) rectangle ([yshift=4ex]aprm.north east) node[below left] {$-\M_\rej$};
    \node[ellipse,minimum width=6em,draw,inner sep=4pt] at (rp) {$\bgM_\rej$};
  \end{tikzpicture}
\end{wrapfigure}
The illustration of the nine-element partition of gamble space corresponding to a model~$\M$ in the acceptability framework with background model~$\bgM$ is given on the left.
The reason for working with a model here when describing the partition instead of with a general assessment, is to draw attention to the specific difference with accept-reject models (cf. Section~\ref{sec:partitions})---namely that the set of rejected gambles~$\M_\rej$ is completely determined by $\bgM_\rej$ and $\M_\acc$.

The focus in this framework's characterisation result lies on the set of acceptable gambles:
\begin{theorem}[cf. Theorem~\ref{thm:char-AR}]\label{thm:char-A}
  Given a set of acceptable gambles~$\M_\acc\subseteq\linGs$ and a background model with indifference to status quo $\bgM$ in~$\ncMs_\zeroM$, then the assessment $\M\defeq\Adelim{\M_\acc}{\bgM_\rej-\M_\acc}$ is a model, i.e., $\M\in\ncMs_\bgM$, if and only if
  \begin{enumerate}[label=\upshape{(A\arabic*)},leftmargin=*,widest=A0]
    \item\label{item:Abg}
      it accepts background-acceptable gambles: $\bgM_\acc\subseteq\M_\acc$,
    \item\label{item:Azr} it does not reject status quo: $0\notin\bgM_\rej-\M_\acc$, and
    \item\label{item:Acn} its acceptable gambles form a cone: $\M_\acc\in\cones$.
  \end{enumerate}
\end{theorem}

When $\bgM_\rej\subseteq-\bgM_\acc$, the resulting models will also belong to the accept-favour framework (cf. Section~\ref{sec:accept-favour}).
This is actually the case for all acceptability-type frameworks we are aware of in the literature.
Again, we give a brief chronological overview, but do give an axiom-for-axiom correspondence with the most general of these frameworks:
\begin{itemize}
 \item
    \Citet[§IV]{Williams-1974} talks about `acceptable bets' (with as a linear space $\linGs$ the set of simple functions with as a basis the indicator functions of the elements of a set~$\someEs$ of subsets of $\pspace$, including $\pspace$).
    The background model~$\bgM$ he uses is $\Adelim{\linGs_\gtrdot}{\linGs_\lessdot}$.
    He does not require Indifference to Status Quo.
  \item
    \Citet{Williams-1975} extends the previous work to include conditional models.
    (He now lets~$\linGs$ be any linear space including constant gambles.)
    To deal with conditional models nicely, he uses a larger background model $\Union_{\emptyset\subset E\subseteq\pspace}\bgM_E$, where each $\bgM_E$ is defined by $(\bgM_E)_\acc\defeq\cset[\big]{f\in\linGs}{\inf{f(E)}>0\conj f(\pspace\setminus E)=\set{0}}$ and $\bgM_\rej=-\bgM_\acc$.
    This union of models is still a model because $(\bgM_E)_\acc+(\bgM_F)_\acc\subset(\bgM_{E\union F})_\acc$ for all non-empty events~$E$ and~$F$.
  \item
    \Citet[§3.7.3]{Walley-1991} discusses ‘almost desirable’ gambles (with $\linGs$ any linear space containing constant gambles).
    The background model here is $\Adelim{\linGs_\geq}{\linGs_\lessdot}$.
    A closure axiom is added for prevision-equivalence (cf. Section~\ref{sec:previsions}): $f+\reals_{>0}\subseteq\M_\acc\then f\in\M_\acc$.
  \item
    \Citet[App.~F]{Walley-1991} talks about ‘really desirable’ gambles (with $\linGs$ any linear space containing constant gambles).
    The background model here is $\Adelim{\linGs_\geq}{\linGs_<}$.
    Normally, given a not necessarily finite partition~$\someEs$ of~$\pspace$ and some~$f$ in~$\linGs$ such that $fI_E\in\M_\acc$ for all~$E$ in~$\someEs$, then Deductive Closure implies $fI_{\Union\othersomeEs}\in\M_\acc$ for all finite subsets $\othersomeEs$ of~$\someEs$.
    (Here~$I_E$ is the indicator of the event~$E$, i.e., the gamble that is~$1$ on~$E$ and~$0$ elsewhere.)
    He imposes a conglomerability axiom that says that under those conditions $f\in\M_\acc$ should hold.

    Let us take take a slightly more detailed look at this framework by explicitly showing the correspondence between \Citeauthor{Walley-1991}'s axioms---ignoring conglomerability---and ours, as formulated in Theorem~\ref{thm:char-A}:
    He calls a set of desirable gambles $\M_\acc\subseteq\linGs$ coherent if and only if for all gambles $f$ and $g$ in $\linGs$ it satisfies
    \begin{enumerate}[label=\upshape{(D\arabic*)},leftmargin=*,widest=D00]
      \addtocounter{enumi}{1}
      \item\label{item:RDps}
        if $f\in\M_\acc$, then $\lambda\cdot f\in\M_\acc$ for all positive scaling factors $\lambda\in\reals_>$ (positive homogeneity),
      \item\label{item:RDcb}
        if $f,g\in\M_\acc$, then $f+g\in\M_\acc$ (addition),
      \addtocounter{enumi}{6}
      \item\label{item:RDapl} if $f<0$, then $f\notin\M_\acc$ (avoiding partial loss),
      \item\label{item:RDapg} if $f\geq0$, then $f\in\M_\acc$ (accepting sure gains).
    \end{enumerate}
    We see that \ref{item:RDps} corresponds to Positive Scaling~\eqref{eq:scaling} and \ref{item:RDcb} to Combination~\eqref{eq:combination}, and therefore together to Deductive Closure~\eqref{eq:deduction}, i.e., \ref{item:Acn}.
    Furthermore, \ref{item:RDapl} is equivalent to $\M_\acc\intersection\linGs_<=\emptyset$, which is in turn equivalent to \ref{item:Azr} with $\bgM_\rej=\linGs_<$.
    Finally, \ref{item:RDapg} is identical to \ref{item:Abg} with $\bgM_\acc=\linGs_\geq$.
  \item
    \Citet[\S2]{Artzner-etal-1999} talk about sets of acceptable `future net worths' (they work with a finite~$\pspace$ and $\linGs=\Gs{\pspace}$).
    The background models~$\bgM$ they use are $\Adelim{\linGs_\geq}{\linGs_\lessdot}$ and $\Adelim{\linGs_\geq}{\linGs_<}$.
    The context of their work is mathematical finance, and more specifically risk measures, which can be seen as scaled negations of lower previsions (cf. Section~\ref{sec:lowprevs}).
    \Citet{Follmer-Schied-2002} generalise \citeauthor{Artzner-etal-1999}'s work and introduce \emph{convex} risk measures:
    They weaken Deductive Closure~\eqref{eq:deduction} to a convexity requirement, so their theory falls outside scope of our frameworks.
\end{itemize}

\section{Linear \& Lower Previsions}\label{sec:previsions}
In this section, we discuss how our framework can be connected to that most popular framework for modelling uncertainty: probability theory.
Given the fact that we have been dealing with gambles, i.e., real-valued functions on the possibility space $\pspace$, and not with events, the most natural way to make this connection is via expectation operators and not probability measures \citep[cf.][]{Whittle-1992}.
Because of our own background, we will call expectation operators \emph{previsions} \citep[cf.][]{DeFinetti-1937,DeFinetti-1974/1975}.

As has been argued by many, restricting attention to precise probabilities (and thus previsions) limits our expression power in modelling uncertainty \citep[e.g.,][and references therein]{Keynes-1921,Koopman-1940-ams,Good-1952,Smith-1961,Dempster-1967,Suppes-1974,Shafer-1976,Levi-1980,Walley-1991}.
It turns out that our framework is sufficiently general to connect it with the theories that have been proposed to provide added expressiveness; we focus on the so-called imprecise-probabilistic theory of coherent lower previsions \citep[for more information, see][]{Walley-1991,Miranda-2008-survey}.

Throughout this section it is convenient to assume that the linear space of gambles of interest $\linGs$ contains the constant gambles, which are identified by their constant value for notational convenience.

\subsection{Linear Previsions}\label{sec:linprevs}
In \citeauthor{DeFinetti-1937}'s theory, the agent's prevision $\pr{f}$ for a gamble~$f$ is a real number seen as his fair price for it.
This can be interpreted in two ways:
The first---which \citet[\S3.1.4]{DeFinetti-1974/1975} seems to have had in mind---is that $f\indiff\pr{f}$, i.e., that the agent is indifferent between~$f$ and $\pr{f}$, which is seen as a constant gamble; this means that he is willing to exchange either one for the other.
The second---which \citet[\S2.3.6, \S2.3.1]{Walley-1991} seems to have kept in mind---is that $\set{f-\pr{f}}+\reals_>$ and $\set{\pr{f}-f}+\reals_>$ are sets of acceptable gambles for the agent, which means that he is willing to buy $f$ for any price strictly lower than $\pr{f}$ and sell it for any price strictly higher.
\citeauthor{Walley-1991}'s ‘willingness’ translates to ‘commitment’ in our terminology (cf.~Section~\ref{sec:accept+reject}).

\Citeauthor*{DeFinetti-1974/1975}'s coherent previsions, when defined on the whole of~$\linGs$, can be characterised as real linear functionals that satisfy $\inf{f}\leq\pr{f}\leq\sup{f}$ for any~$f$ in~$\linGs$ \citep[\S3.1.5]{DeFinetti-1974/1975}. 
This means any such prevision~$\pr$ partitions~$\linGs$ into
\begin{enumerate}
  \item $\linGs_{\gtrrvf{\pr}}\defeq\cset{h\in\linGs}{\pr{h}>0}$, which always includes $\linGs_\gtrdot\defeq\cset{h\in\linGs}{\inf{h}>0}$,
  \item $\linGs_{\lessrvf{\pr}}\defeq\cset{h\in\linGs}{\pr{h}<0}$, which always includes  $\linGs_\lessdot\defeq\cset{h\in\linGs}{\sup{h}<0}$, and
  \item $\linGs_{\eqrvf{\pr}}\defeq\cset{h\in\linGs}{\pr{h}=0}$, which always contains~$0$; its elements are called marginal gambles and $f-\pr{f}$ and $\pr{f}-f$ are the \emph{marginal gambles} corresponding to $f$.
\end{enumerate}
\par
\noindent Because of the linearity of~$\pr$, $\linGs_{\gtrrvf{\pr}}$ and $\linGs_{\lessrvf{\pr}}$ are convex cones related by negation---i.e., $\linGs_{\gtrrvf{\pr}}=-\linGs_{\lessrvf{\pr}}$---and~$\linGs_{\eqrvf{\pr}}$ is  a linear space; the line segment joining any element of $\linGs_{\gtrrvf{\pr}}$ to any element of $\linGs_{\lessrvf{\pr}}$ always intersects~$\linGs_{\eqrvf{\pr}}$.
The associated mental image is that of a hyperplane~$\linGs_{\eqrvf{\pr}}$ separating the (open) positive orthant~$\linGs_\gtrdot$ from the (open) negative one~$\linGs_\lessdot$.
\begin{wrapfigure}{r}{0pt}
  \begin{tikzpicture}[scale=1.5,baseline=-.7ex]
    \draw[->] (-.7,0) coordinate (xl) -- (1,0) coordinate (xu);
    \draw[->] (0,-1) coordinate (yl) -- (0,1) coordinate (yu);
    \node (a1) at (135:.7) {};
    \begin{pgfonlayer}{background}
      \draw[border] (intersection of 0,0--a1 and xl--{xl|-yu}) coordinate[label={[xshift=3pt]right:$\linGs_{\eqrvf{\pr}}$}] (a1away) -- (intersection of 0,0--a1 and xu--{yl-|xu}) coordinate (a1away2);
    \end{pgfonlayer}
  \end{tikzpicture}
  \hspace{1em}
  \begin{tikzpicture}[scale=1.5,baseline=-.7ex]
    \draw[->] (-.7,0) coordinate (xl) -- (1,0) coordinate (xu);
    \draw[->] (0,-1) coordinate (yl) -- (0,1) coordinate (yu);
    \node (a1) at (135:.7) {};
    \begin{pgfonlayer}{background}
      \draw[exclborder] (intersection of 0,0--a1 and xl--{xl|-yu}) coordinate (a1away) -- (intersection of 0,0--a1 and xu--{yl-|xu}) coordinate (a1away2);
      \fill[lightgray] (a1away) |- (xu|-yu) |- (a1away2) --cycle;
      \node at (.5,.5) {$\linGs_{\gtrrvf{\pr}}$};
    \end{pgfonlayer}
  \end{tikzpicture}
  \hspace{1em}
  \begin{tikzpicture}[scale=1.5,baseline=-.7ex]
    \draw[->] (-.7,0) coordinate (xl) -- (1,0) coordinate (xu);
    \draw[->] (0,-1) coordinate (yl) -- (0,1) coordinate (yu);
    \node (a1) at (135:.7) {};
    \begin{pgfonlayer}{background}
      \draw[exclborder] (intersection of 0,0--a1 and xl--{xl|-yu}) coordinate (a1away) -- (intersection of 0,0--a1 and xu--{yl-|xu}) coordinate (a1away2);
      \fill[gray] (a1away) |- (xl|-yl) -| (a1away2) --cycle;
      \node[white] at (-.3,-.4) {$\linGs_{\lessrvf{\pr}}$};
    \end{pgfonlayer}
  \end{tikzpicture}
\end{wrapfigure}
This partitioning is illustrated on the right for the linear prevision $\pr$ modelling a fair coin, so which specifies probability $1/2$ for both elementary events and which has $(\frac{1}{2},-\frac{1}{2})$ and $(-\frac{1}{2},\frac{1}{2})$ as the marginal gambles corresponding to the indicator gambles $(1,0)$ and $(0,1)$, respectively.
We give the hyperplane of marginal gambles $\linGs_{\eqrvf{\pr}}$ in the first figure, the open half-space of gambles with positive prevision~$\linGs_{\gtrrvf{\pr}}$ including $\linGs_\gtrdot$ in the second, and the open half-space of gambles with negative prevision $\linGs_{\lessrvf{\pr}}$ including $\linGs_\lessdot$ in the last figure.

Under the first interpretation---of indifference between $f$ and $\pr{f}$---, the gambles in
$
  \Union_{f\in\linGs}\set{f-\pr{f},\pr{f}-f} = \linGs_{\eqrvf{\pr}}
$
are assessed to be acceptable.
So the assessment corresponding to~$\pr$ is $\Adelim{\linGs_{\eqrvf{\pr}}}{\emptyset}$.
Under the second interpretation, the gambles in the set
$
  \Union_{f\in\linGs}\set{f-\pr{f},\pr{f}-f} + \reals_> = \linGs_{\eqrvf{\pr}} + \reals_>
$
are assessed to be acceptable; this set is equal to $\linGs_{\gtrrvf{\pr}}$ because any gamble~$h$ in $\linGs_{\gtrrvf{\pr}}$ can be decomposed into $h-\pr{h}\in\linGs_{\eqrvf{\pr}}$ and $\pr{h}\in\reals_>$, and any gamble~$-h$ in $\linGs_{\gtrrvf{\pr}}$ can be decomposed into $\pr{h}-h\in\linGs_{\eqrvf{\pr}}$ and $-\pr{h}=\pr(-h)\in\reals_>$.
So the assessment corresponding to~$\pr$ is $\Adelim{\linGs_{\gtrrvf{\pr}}}{\emptyset}$.

From these assessments, we wish to derive models.
For that, we need to choose a minimal background model.
Our choice will be based on what happens if the coherence constraint $\inf{f}\leq\pr{f}\leq\sup{f}$ is violated but linearity is not.
Then either $\sup\group{f-\pr{f}}<0$ or $\inf\group{f-\pr{f}}>0$, i.e., this happens whenever $f-\pr{f}$ is an element of $\linGs_\lessdot$ or~$\linGs_\gtrdot$.
This implies that $\linGs_{\eqrvf{\pr}}$ must not intersect $\linGs_\lessdot$ and~$\linGs_\gtrdot$.
\begin{itemize}
  \item
    Under the first interpretation this means that no element of either may be indifferent, which implies that $\linGs_\lessdot\union\linGs_\gtrdot\subseteq\bgM_\rej\union-\bgM_\rej$ must hold for the background model~$\bgM$.
    In the spirit of \citeauthor{DeFinetti-1937}'s focus on two-sided gambles, we therefore choose all gambles in~$\linGs_\gtrdot$ to be favourable; i.e., $\dotbgM\subseteq\bgM$.
  \item
    For the second interpretation, the constraint may be reformulated as saying that $\linGs_{\gtrrvf{\pr}}$ must not intersect $\linGs_\lessdot$, i.e., no element of $\linGs_\lessdot$ may be acceptable.
    In this case $\Adelim{\emptyset}{\linGs_\lessdot}$ must be part of the background model~$\bgM$.
\end{itemize}

So, under the first interpretation, the model we associate with a coherent prevision~$\pr$ is
\begin{align*}
  \cls{\ncMs}\Adelim{\linGs_{\eqrvf{\pr}}\union\linGs_\gtrdot}{\linGs_\lessdot}
  &= \ext{\Ms}(\ext{\Ds}\Adelim{\linGs_{\eqrvf{\pr}}\union\linGs_\gtrdot}
                                                    {\linGs_\lessdot})
          &&\text{(def. $\cls{\ncMs}$)}\\
  &= \ext{\Ms}\Adelim{\phull(\linGs_{\eqrvf{\pr}}\union\linGs_\gtrdot)}
                                    {\linGs_\lessdot}
          &&\text{(def. $\ext{\Ds}$)}\\
  &= \ext{\Ms}\Adelim{
          \linGs_{\eqrvf{\pr}} \union \linGs_\gtrdot
                                           \union (\linGs_{\eqrvf{\pr}}+\linGs_\gtrdot)
        }{\linGs_\lessdot} &&\text{(def. $\phull$)}\\
  &= \ext{\Ms}\Adelim{
          \linGs_{\eqrvf{\pr}}\union(\linGs_{\eqrvf{\pr}}+\linGs_\gtrdot)
        }{\linGs_\lessdot} &&\text{($0\in\linGs_{\eqrvf{\pr}}$)}\\
  &= \ext{\Ms}\Adelim{\linGs_{\eqrvf{\pr}}\union\linGs_{\gtrrvf{\pr}}}
                                    {\linGs_\lessdot}
        &&\text{(%
          $
            \linGs_\gtrdot \subset \linGs_{\gtrrvf{\pr}}
                                   = \linGs_{\eqrvf{\pr}} + \reals_>
                                   \subseteq \linGs_{\eqrvf{\pr}} + \linGs_\gtrdot
                                   \subseteq \linGs_{\eqrvf{\pr}} + \linGs_{\gtrrvf{\pr}}
                                   \subseteq \linGs_{\gtrrvf{\pr}}
          $
        )}\\
  &= \Adelim{\linGs_{\eqrvf{\pr}}\union\linGs_{\gtrrvf{\pr}}}
                    {\linGs_\lessdot-(\linGs_{\eqrvf{\pr}}\union\linGs_{\gtrrvf{\pr}})}
        &&\text{(def. $\ext{\Ms}$)}\\
  &= \Adelim{\linGs_{\eqrvf{\pr}}\union\linGs_{\gtrrvf{\pr}}}
                    {(\linGs_\lessdot-\linGs_{\eqrvf{\pr}})
                      \union (\linGs_\lessdot-\linGs_{\gtrrvf{\pr}})} \\
  &= \Adelim{\linGs_{\eqrvf{\pr}}\union\linGs_{\gtrrvf{\pr}}}
                  {\linGs_{\lessrvf{\pr}}}.
        &&\text{($\linGs_\lessdot\subset\linGs_{\lessrvf{\pr}}=-\linGs_{\gtrrvf{\pr}}$)}
\end{align*}
This is a maximal model (cf. Proposition~\ref{prop:maximalncmods}).
The model associated with $\pr$ under the second interpretation is
\begin{align*}
  \otherM_\pr \defeq \cls{\ncMs}\Adelim{\linGs_{\gtrrvf{\pr}}}{\linGs_{\lessdot}}
  &= \ext{\Ms}(\ext{\Ds}\Adelim{\linGs_{\gtrrvf{\pr}}}{\linGs_{\lessdot}})
          &&\text{(def. $\cls{\ncMs}$)}\\
  &= \ext{\Ms}\Adelim{\linGs_{\gtrrvf{\pr}}}{\linGs_{\lessdot}}
          &&\text{(def. $\ext{\Ds}$, $\linGs_{\gtrrvf{\pr}}\in\cones$)}\\
  &= \Adelim{\linGs_{\gtrrvf{\pr}}}
                    {\linGs_{\lessdot}\union(\linGs_{\lessdot}-\linGs_{\gtrrvf{\pr}})}
          &&\text{(def. $\ext{\Ms}$)}\\
  &= \Adelim{\linGs_{\gtrrvf{\pr}}}{\linGs_{\lessrvf{\pr}}}.
          &&\text{(%
            $\linGs_\lessdot\subset\linGs_{\lessrvf{\pr}}=-\linGs_{\gtrrvf{\pr}}$
          )}
\end{align*}
The second interpretation leads to a model that is less committal than the one resulting from the first interpretation.\footnote{
  \Citet{Wagner-2007-Smith-Walley} argues for the second interpretation by pointing out that it avoids vulnerability to a weak Dutch book \citep[cf.][]{Shimony-1955}.
  The commitments implied by the different models associated with both interpretations---$\linGs_{\eqrvf{\pr}}\union\linGs_{\gtrrvf{\pr}}$ versus $\linGs_{\gtrrvf{\pr}}$---make the reason for this explicit.
}
(As $\linGs_\gtrdot\subset\linGs_{\gtrrvf{\pr}}$, we see that this does not depend on the difference in restrictions imposed on $\bgM$.)
To not record any commitments the agent might not have wanted to make when specifying the prevision, we continue with the second interpretation.

What is the background model associated with the set~$\prs$ of all coherent previsions on~$\linGs$?
At face value, this question asks what assessment is shared by every model associated with a coherent prevision; the next proposition answers this:
\begin{proposition}\label{prop:pr-bgM-inner}
  The assessment implicit in all models associated with a coherent prevision is $\Intersection_{\pr\in\prs}\otherM_\pr = \dotbgM$.
\end{proposition}
\noindent
Going beyond its literal meaning, the question leads us to consider which models~$\bgM$ are compatible as background models with the set~$\prs$ of all coherent previsions on~$\linGs$.
To wit, which~$\bgM$ are such that for any~$\pr$ in~$\prs$, $\otherM_\pr\reckUnion\bgM$ can still be put into a one-to-one correspondence with~$\pr$?
The remainder of this section is devoted to providing a cogent answer to this question.

There is a one-to-one relation between~$\pr$ and~$\otherM_\pr$ because of the one-to-one relationship with its defining set of favourable gambles $(\otherM_\pr)_\fav=\linGs_{\gtrrvf{\pr}}$.
However, from our discussion on the two interpretations and their impact
we know that there are other models---$\Adelim{\linGs_{\eqrvf{\pr}}\union\linGs_{\gtrrvf{\pr}}}{\linGs_{\lessrvf{\pr}}}$ in particular---that can be associated with~$\pr$.
The model~$\otherM_\pr$ turns out to be the least resolved such model:
\begin{proposition}\label{prop:pr-models-disjoint}
  For any pair of coherent previsions $\pr$ and $\otherpr$ in $\prs$, we have that
  $
    \ncMs_{\otherM_\pr} \intersection \ncMs_{\otherM_\otherpr}
    \neq \emptyset
  $
  if and only if $\pr=\otherpr$.
  Or, in words: a model is uniquely associated to a coherent prevision $\pr$ if and only if it dominates or is equal to $\ncMs_{\otherM_\pr}$.
\end{proposition}
\begin{corollary}\label{cor:pr-models-max}
  Given a coherent prevision~$\pr$, all maximal models in $\maxncMs_{\otherM_\pr}$ are uniquely associated to~$\pr$.
\end{corollary}
\noindent
So any coherent prevision~$\pr$ can be uniquely identified with some element of~$\ncMs_{\otherM_\pr}$.
This means that any~$\bgM$ in~$\ncMs_{\dotbgM}$ is a compatible background model if $\ncMs_{\otherM_\pr}\intersection\ncMs_\bgM$ is non-empty for all~$\pr$ in~$\prs$.
We can make this more concrete:
\begin{proposition}\label{prop:pr-bgM-outer}
  A background model $\bgM$ in $\ncMs_{\dotbgM}$ is compatible with any coherent prevision $\pr$ in~$\prs$, i.e., $\otherM_\pr$ respects~$\bgM$, if and only if $\bgM\subset\Adelim{\linGs_\geq}{\linGs_\leq}$.
\end{proposition}
\noindent
So we find that the background model~$\bgM$ must satisfy $\dotbgM\subseteq\bgM\subset\Adelim{\linGs_\geq}{\linGs_\leq}$ for us to consider it compatible with the theory of coherent previsions.

Up until now, we have looked at previsions defined on the whole of~$\linGs$.
The procedure we used to generate an assessment---calling gambles uniformly dominating marginal gambles acceptable---can be used equally well for previsions defined on a subset~$\someGs$ of~$\linGs$.
The interesting issues that would arise when discussing such non-exhaustively specified linear previsions arise as well when treating coherent lower previsions, which give their user even more control over the commitments specified.
Therefore we move on to that topic.

\subsection{Lower Previsions}\label{sec:lowprevs}
In \citeauthor{Walley-1991}'s theory of coherent lower previsions, the agent's lower prevision $\lpr{f}$ for a gamble~$f$ is a real number seen as his supremum buying price for it \citep[\S2.3.1]{Walley-1991}.
The interpretation is that $\set{f-\lpr{f}}+\reals_>$ is a set of acceptable gambles, which means that he is willing to buy~$f$ for any price strictly lower than $\lpr{f}$ \citep[\S2.3.4]{Walley-1991}.
Again, $f-\lpr{f}$ is called the \emph{marginal gamble} associated with~$f$.
Upper previsions $\upr{f}$ are seen as infimum acceptable buying prices and because of this are conjugate to lower previsions, i.e., $\upr{f}=-\lpr\group{-f}$ \citep[\S2.3.5]{Walley-1991}; therefore it is sufficient to develop the theory in terms of lower previsions.
Linear previsions are lower previsions that are self-conjugate, i.e., that coincide with their conjugate upper prevision.

\citeauthor{Walley-1991}'s coherent lower previsions, when defined on a subset~$\someGs$ of~$\linGs$, can be characterised as follows: a lower prevision~$\lpr$ on~$\someGs$ is \emph{coherent} if and only if $\inf_{h\in\margs{\lpr}}\inf_{g\in\phull\margs{\lpr}}\sup\group{g-h}\geq0$ \citep[adapted from][\S2.5.1]{Walley-1991}, where we conveniently used the set of marginal gambles $\margs{\lpr}\defeq\cset{f-\lpr{f}}{f\in\someGs}$.
This is a weaker criterion than the one for linear previsions; however, it implies, among other things, that $\inf{f}\leq\lpr{f}\leq\sup{f}$ for any~$f$ in~$\someGs$ and that $\lpr$ is point-wise dominated by some linear prevision \citep[\S2.6.1 and \S3.3.3]{Walley-1991}.
So the results of the previous section concerning compatible background models~$\bgM$ in~$\ncMs$ are carried over; they must satisfy $\dotbgM\subseteq\bgM\subset\Adelim{\linGs_\geq}{\linGs_\leq}$, which we assume to be the case from here onward in this section.

So the assessment we associate with a lower prevision~$\lpr$ is~$\A_\lpr\defeq\Adelim{\margs{\lpr}+\reals_>}{\emptyset}$, resulting in a model $\M_\lpr\defeq\A_\lpr\reckunion\bgM$ by natural extension.
It is useful to have a criterion to determine when the assessment~$\A_\lpr$ is deductively closable and, if that is the case, an explicit expression for the natural extension:
\begin{theorem}\label{thm:lpr-model}
  The model $\M_\lpr$ corresponding to a lower prevision~$\lpr$ avoids confusion if and only if $\lpr$ \emph{avoids sure loss}, i.e., $\inf_{g\in\phull\margs{\lpr}}\sup{g}\geq0$.
  In that case
  $
    \M_\lpr = \Adelim{\phull\margs{\lpr}+\linGs_\gtrdot}
                     {\linGs_\lessdot-\phull\margs{\lpr}} \union \bgM.
  $
\end{theorem}

As an aside, notice that $\A_\lpr$ is an assessment consisting purely of acceptability statements; so by taking~$\bgM$ in~$\ncMs_\zeroM$, we see that modelling uncertainty using previsions can be seen as working in the acceptability framework of Section~\ref{sec:acceptability}.
Moreover, notice---making abstraction of $\bgM$---that $\M_\lpr$ consists purely of favourability statements and that we could strengthen every acceptability statement in $\A_\lpr$ to a favourability statement with the resulting model still being equal to $\M_\lpr$.
So modelling uncertainty using previsions can also be seen as working in the favourability framework of Section~\ref{sec:favourability}.

We now know how to associate a model with a lower prevision.
It is also possible to move in the other direction: given a model~$\M$ in $\ncMs_{\dotbgM}$, we can derive a coherent lower prevision~$\lpr_\M$.
The translation rule is based on the interpretation of the lower prevision for a gamble~$f$ in~$\linGs$ as a supremum acceptable buying price:
\begin{equation}\label{eq:Mtolpr}
  \lpr_\M{f} \defeq \sup\cset{\alpha\in\reals}{f-\alpha\in\M_\acc}.
\end{equation}
What kind of prevision does the background model translate to?
\begin{proposition}\label{prop:bgM-to-vac}
  Given a background model~$\bgM$ in~$\ncMs$ such that $\dotbgM\subseteq\bgM\subset\Adelim{\linGs_\geq}{\linGs_\leq}$, then the corresponding lower prevision is \emph{vacuous}: $\lpr_\bgM=\inf$.
\end{proposition}

The rules for moving between models and lower previsions preserve the most resolve possible without adding 
\begin{wrapfigure}{r}{0pt}
  \(
  \begin{tikzpicture}[scale=1.5,baseline]
    \draw[->] (-1,0) coordinate (xl) -- (1,0) coordinate (xu);
    \draw[->] (0,-.7) coordinate (yl) -- (0,1) coordinate (yu);
    \coordinate (a1) at (-10:.7);
    \coordinate (a2) at (135:.7);
    \coordinate (r2) at (150:.7);
    \begin{pgfonlayer}{background}
      \draw[border] (0,0) -- (intersection of 0,0--a1 and xu--{xu|-yu}) coordinate (a1away);
      \draw[border] (0,0) -- (intersection of 0,0--a2 and yu--{xu|-yu}) coordinate (a2away);
      \fill[lightgray] (0,0) -- (a1away) |- (a2away) --cycle;
      \draw[border] (0,0) -- (intersection of 0,0--r2 and xl--{xl|-yl}) coordinate (r2away);
      \draw[exclborder] (0,0) -- (intersection of 0,0--a2 and yl--{xu|-yl}) coordinate (a2anotherway);
      \fill[gray] (0,0) -- (a2anotherway) |- (xl|-yl) |- (r2away) --cycle;
      \draw[favborder] (0,0) -- (a1away);
    \end{pgfonlayer}
    \node at (.5,.5) {$\M_\acc$};
    \node[white] at (-.5,-.4) {$\M_\rej$};
  \end{tikzpicture}
  \;\supseteq\;
  \begin{tikzpicture}[scale=1.5,baseline]
    \draw[->] (-1,0) coordinate (xl) -- (1,0) coordinate (xu);
    \draw[->] (0,-.7) coordinate (yl) -- (0,1) coordinate (yu);
    \coordinate (a1) at (-10:.7);
    \coordinate (a2) at (135:.7);
    \begin{pgfonlayer}{background}
      \draw[exclborder] (0,0) -- (intersection of 0,0--a1 and xu--{xu|-yu}) coordinate (a1away);
      \draw[exclborder] (0,0) -- (intersection of 0,0--a2 and yu--{xu|-yu}) coordinate (a2away);
      \fill[lightgray] (0,0) -- (a1away) |- (a2away) --cycle;
      \draw[exclborder] (0,0) -- (intersection of 0,0--a1 and xl--{xl|-yl}) coordinate (r2away);
      \draw[exclborder] (0,0) -- (intersection of 0,0--a2 and yl--{xu|-yl}) coordinate (a2anotherway);
      \fill[gray] (0,0) -- (a2anotherway) |- (xl|-yl) |- (r2away) --cycle;
    \end{pgfonlayer}
    \node at (.5,.5) {$(\M_{\lpr_\M})_\acc$};
    \node[white,inner sep=0pt] at (-.5,-.4) {$(\M_{\lpr_\M})_\rej$};
  \end{tikzpicture}
  \)
\end{wrapfigure}
any new statements; this is made explicit by the following result:
\begin{proposition}\label{prop:M-pr-M}
  Given a model~$\M$ in~$\ncMs_\bgM$, then
  \begin{enumerate}
    \item\label{item:prop:M-pr-M:coh}
      the corresponding lower prevision $\lpr_\M$ is coherent, and
    \item\label{item:prop:M-pr-M:commitloss}
      it never represents more, but may represent less resolve: $\M_{\lpr_\M}\subseteq\M$.
  \end{enumerate}
\end{proposition}
\noindent
This loss in resolve is illustrated above right for a model~$\M$ in~$\ncMs_{\Adelim{\linGs_\geq}{\linGs_<}}$:
The transition to a coherent lower prevision has erased the part of the set of rejected gambles that is not related to the accepted gambles by negation.
Furthermore, all border structure of both sets has also been erased; only open sets remain.

Now that we can move from lower previsions to models and back, we can derive some standard inference results from the theory of coherent lower previsions using the tools developed in this paper:
\begin{proposition}\label{prop:pr-M-pr}
  Given a non-empty set $\someGs\subseteq\linGs$ and a lower prevision $\lpr$ on~$\someGs$, then
  \begin{enumerate}
    \item\label{item:lprnatex}
      the \emph{natural extension} $\lpr_{\M_\lpr}$ is defined by $\lpr_{\M_\lpr}f=\sup_{g\in\phull\margs{\lpr}}\inf(f-g)$ for all~$f$ in~$\linGs$,
    \item\label{item:natexdomlpr}
      the lower prevision $\lpr$ is dominated by its natural extension: $\lpr_{\M_\lpr}\geq\lpr$ on~$\someGs$, and
    \item\label{item:lpreqnatex} the lower prevision $\lpr$ coincides with its natural extension, i.e., $\lpr_{\M_\lpr}=\lpr$ on~$\someGs$, if and only if $\lpr$~is coherent.
  \end{enumerate}
\end{proposition}
\noindent
For incoherent lower previsions~$\lpr$ on~$\someGs$ that avoid sure loss, we see that there are some gambles~$f$ such that $\lpr_{\M_\lpr}f>\lpr{f}$.
This is a consequence of the fact that the commitments encoded in~$\lpr$ have not been fully taken into account in the specification of gambles such as~$f$; they are in $\lpr_{\M_\lpr}$.

\section{An Application: Dealing with Symmetry}\label{sec:xch}
Consider a set~$\transfos$ of transformations of the possibility space $\pspace$.
With any gamble $f$ and any transformation~$T$ in~$\transfos$, there corresponds a transformed gamble $T^tf$ defined by $T^tf\defeq f\circ T$, so that $(T^tf)(\omega)=f(T\omega)$ for all~$\omega$ in~$\pspace$.
For it to have the properties one would expect of a set of transformations, we assume $\transfos$ with function composition~$\after$ to be a monoid, i.e.,
\begin{itemize}
  \item it is closed: $T'\after T''\in\transfos$,
  \item it is associative: $T'\after(T\after T'')=(T'\after T)\after T''$, and
  \item it has an identity element $T^*$ in $\transfos$: $T^*\after T=T\after T^*=T$,
\end{itemize}
for all transformations $T$, $T'$, and $T''$ in $\transfos$.

In the background, we assume all positive gambles ($\linGs_>$) to be favourable.
We also assume that there is some symmetry, represented by the monoid~$\transfos$, associated with the experiment, which leads the agent to be indifferent between any gamble~$f$ and its transformation~$T^tf$.
This gives rise to the background set of indifferent gambles
$
  \linGs_\transfos \defeq \lhull\cset{f-T^tf}{f\in\linGs\conj T\in\transfos}.
$
What is the background model corresponding to such a background assessment and when does it avoid confusion?
\begin{proposition}\label{prop:symbgM}
  The model 
  $
    \bgM_\transfos \defeq \Adelim{\linGs_\transfos}{\emptyset} \reckunion
                          \Adelim{\linGs_>}{\linGs_<}
                   = \Adelim{\linGs_\transfos\union(\linGs_\transfos+\linGs_>)}
                            {\linGs_\transfos+\linGs_<}
  $
  is a favour-indifference model without confusion and with indifference to status quo, i.e., $\bgM_\transfos\in\fif\ncMs_\zeroM$,  if and only if $f<0$ for no background indifferent gamble~$f$ in~$\linGs_\transfos$.\footnotemark
\end{proposition}
\footnotetext{
  The condition of this proposition is closely related to the necessary and sufficient condition for the so-called left-amenability of the monoid $\transfos$---i.e., the existence of a $\transfos$-invariant linear prevision on $\linGs$---that $\sup{f}\geq0$ for all~$f$ in~$\linGs_\transfos$ \citep[cf.][]{Greenleaf-1969,Walley-1991}.
  The fact that our condition is slightly stronger should come as no surprise, as we have seen in Section~\ref{sec:linprevs} that previsions cannot express the distinction between $\inf{f}\geq0$ and $f>0$.
}

When working in a favourability framework using $\bgM_\transfos$ as the background model, the extension of assessments is governed by the following result:
\begin{corollary}[cf. Theorem~\ref{thm:ffdcAs-zerocrit}]\label{prop:ffsymdcAs-zerocrit}
  Given a set of favourable gambles $\A_\fav\subseteq\linGs$, then the favourability assessment $\A\defeq\Adelim{\A_\fav}{-\A_\fav}$ respects the background model $\bgM_\transfos$ if and only if\/ $0\notin\linGs_\transfos+\phull(\linGs_>\union\A_\fav)$.
  In that case, the natural extension $\M\defeq\A\reckunion\bgM_\transfos\in\fif\ncMs_\zeroM$ has defining components $\M_\indiff=\linGs_\transfos$ and $\M_\fav=\linGs_\transfos+\phull(\linGs_>\union\A_\fav)$, so the resulting model is $\M=\Adelim{\linGs_\transfos\union\M_\fav}{-\M_\fav}$.
\end{corollary}
\noindent
Furthermore, in this context, models can be characterised as follows:
\begin{corollary}[cf.~Theorem~\ref{thm:char-F}]\label{prop:char-Fsym}
  Given a set of favourable gambles $\M_\fav\subseteq\linGs$, then $\M\defeq\Adelim{\linGs_\transfos\union\M_\fav}{-\M_\fav}$ is a model, i.e., $\M\in\fif\ncMs_{\bgM_\transfos}$, if and only if
  \begin{enumerate}[label=\upshape{(F$\transfos$\arabic*)},leftmargin=*,widest=FT0]
    \item\label{item:Fsymbg}
      it favours background-favourable gambles: $\linGs_\transfos+\linGs_>\subseteq\M_\fav$,
    \item\label{item:Fsymzr} it does not favour status quo: $0\notin\M_\fav$,
    \item\label{item:Fsymcn}
      its favourable gambles form a cone: $\M_\fav\in\cones$, and
    \item\label{item:Fsymss}
      it favours sweetened deals: $\M_\fav+\linGs_\transfos\subseteq\M_\fav$.
  \end{enumerate}
\end{corollary}

An interesting special case obtains when the set of transformations is in particular a finite group $\permuts$ of permutations of $\pspace$, in which case the condition of Proposition~\ref{prop:symbgM} always holds.
We give an overview of the important ideas and the conclusions that can be drawn; for details, see our earlier papers on exchangeability \citep{DeCooman-Quaeghebeur-2012-Kyburg,DeCooman-Quaeghebeur-2009-ISIPTA}.
The set $\pspace_\permuts\defeq\cset[\big]{\cset{\pi\omega}{\pi\in\permuts}}{\omega\in\pspace}$ of $\permuts$-invariant atoms is a partition of $\pspace$; these atoms are the smallest subsets of $\pspace$ that are invariant under all permutations $\pi$ in~$\permuts$.
Consequently, a gamble $f$ on $\pspace$ is invariant under all permutations in $\permuts$---or simply, $\permuts$-invariant---if and only if it is constant on this partition.
The linear subspace of all $\permuts$-invariant gambles is $\invars_\permuts\defeq\cset[\big]{f\in\linGs}{\group{\forall A\in\pspace_\permuts:\card{f(A)}=1}}$.

Consider the linear gamble transformation $\avg{\permuts}$ defined for all $f$ in $\linGs$ by $\avg{\permuts}f\defeq\sum_{\pi\in\permuts}\pi^tf/\card{\permuts}$,
then $\avg{\permuts}$ is a projection operator---meaning that $\avg{\permuts}\circ\avg{\permuts}=\avg{\permuts}$---satisfying $\avg{\permuts}\circ\pi^t=\pi^t\circ\avg{\permuts}=\avg{\permuts}$ for all $\pi$ in~$\permuts$.
This in turn (non-trivially) implies that its range is the set $\invars_\permuts$ of all $\permuts$-invariant gambles, and its kernel is the set $\linGs_\permuts$.
The projection operator $\avg{\permuts}$ allows for a simple representation result: if $\M_\fav$ satisfies \ref{item:Fsymbg}--\ref{item:Fsymss}, then $f\in\M_\fav$ if and only if $\avg{\permuts}f\in\M_\fav$.
So the set $\M_\fav$ is completely characterised by its projection $\avg{\permuts}\M_\fav$ on the lower-dimensional linear space $\invars_\permuts$.
(Cfr. our discussion of Proposition~\ref{prop:lineality}\ref{item:lineality-invariance}: $\D_\indiff$ here corresponds to $\linGs_\permuts$ and $\someGs$ here corresponds to~$\invars_\permuts$.)

\citet{DeCooman-Quaeghebeur-2012-Kyburg,DeCooman-Quaeghebeur-2009-ISIPTA} discuss the special case where $\pspace\defeq\mathcal{X}^n$ is the set of length-$n$ sequences of samples in a finite set~$\mathcal{X}$, and $\permuts$ is the group of all permutations of such sequences obtained by permuting their indices.
Stating indifference between a gamble on sequences and all those gambles related by sequence-index permutations corresponds to an exchangeability assumption, and the above-mentioned representation result is then a significant generalisation of \citeauthor{DeFinetti-1937}'s \cite{DeFinetti-1937} representation theorem for finite exchangeable sequences in terms of hypergeometric distributions---sampling from an urn without replacement: indeed, for any gamble $f$ the constant value of $\avg{\permuts}f$ on an invariant atom turns out to be the hypergeometric expectation of $f$ associated with that atom.
This result can be extended to infinite exchangeable sequences as well \citep{DeCooman-Quaeghebeur-2012-Kyburg,DeCooman-Quaeghebeur-2010-IPMU}.
The more general discussion above also includes the case of partial exchangeability.

\section{Conclusions}\label{sec:conclusions}
We started out this paper by claiming that our framework allows us to be more expressive and that it has a unifying character.
This is already apparent in the elicitation step; we can directly incorporate assessments of various natures: of course accept and reject statements, but also statements of indifference and favourability, and the preference relation counterparts of all these statements; even (imprecise) probabilistic statements pose no problem.
Naturally, all these types of statements can also be used on the output side.
Between input and output, we know how to transform mere assessments into models that satisfy a number of---according to our judgement, in many contexts reasonable---rationality requirements, or detect whether the assessments contain inconsistencies that make this impossible.

We hope that the basic theory of this paper will be a starting point for further research and numerous applications, both theoretical and practical.
For practical applications, we of course need computational tools that, given an assessment, produce inferences of the various types described above.
An algorithm that does the core computations needed (on finite possibility spaces) has recently been devised \citep{Quaeghebeur-2012-smps}, so that hurdle has in large part been taken.
Furthermore, conditional and marginal models must be defined, and rules for deriving and combining them must be formulated.
Again, we can expect that quite a bit of work for this has essentially been done already: much can be carried over from the literature on coherent sets of desirable gambles \citep[see, e.g.,][]{2013-Quaeghebeur-itip,DeCooman-Miranda-2011}, which itself builds on the much larger corpus on coherent lower previsions.

As regards comparisons with other frameworks in the literature: the approach using credal sets (sets of linear previsions) deserves attention.
We know that closed convex credal sets are equivalent to coherent lower previsions, but once closure is not required, we can deduce from polytope-theoretic duality properties that only a subclass of such models can be described within our framework \citep[cf.][]{2013-Quaeghebeur-itip}.
It would, however, be useful to know exactly what subclass of credal sets can be equivalently described using our models, so that we may also know what type of information they cannot represent.

On the technical side, investigating topological properties of the objects in our framework would be useful to able to deal with questions prompted by our simple illustrative examples, such as, e.g., `Is the interior of the set of acceptable gambles always favourable whenever the set of rejected gambles is not empty?'
Also, we have not put any restrictions on the possibility space in our exposition, but for finite possibility spaces it should be possible to formulate constructive counterparts for some proofs; e.g., so that we may construct maximal models dominating an assessment, instead of essentially just positing their existence.

To finish, a pair of generalisations of our framework present themselves:
\begin{itemize}
  \item
    The work of \citet{Artzner-etal-1999} on coherent risk measures was generalised by \citet{Follmer-Schied-2002} by weakening Combination~\eqref{eq:combination} plus Positive Scaling~\eqref{eq:scaling} to taking convex combinations---and put into an imprecise probability theory context by \citet{Pelessoni-Vicig-2005}.
    This can also be done with our framework; it would correspond to relaxing the linear utility assumption.
  \item
    Similarly, in some sense and aspects the work of \citet{Seidenfeld-Schervish-Kadane-2010} can be seen as generalising probability theory by working with disjunctions of probability measures.
    This can again also be done with our framework; it would correspond to considering sets of accept-reject models.
\end{itemize}

\section*{Acknowledgements}\addcontentsline{toc}{section}{Acknowledgements}
We wish to thank the reviewers of all versions of this paper, who provided many useful comments, and Teddy Seidenfeld for animated, focussing discussion.

\appendix
\newcommand{\secnumname}[1]{Section~\ref{#1} (\nameref{#1})}
\section*{Proofs}\label{app:proofs}\addcontentsline{toc}{section}{Proofs}
\subsection*{\secnumname{sec:no-confusion}}\addcontentsline{toc}{subsection}{\secnumname{sec:no-confusion}}
\begin{proof}[Proof of Proposition~\ref{prop:removing-confusion-from-assessments}]
  The three modified assessments avoid confusion by construction:
  \begin{align*}
    \Adelim{\A_\accnrej}{\A_\rej}_\confus
      = \A_\accnrej\intersection\A_\rej = \emptyset,
    &&
    \Adelim{\A_\acc}{\A_\rejnacc}_\confus
      = \A_\acc\intersection\A_\rejnacc = \emptyset,
    &&
    \Adelim{\A_\accnrej}{\A_\rejnacc}_\confus
      = \A_\accnrej\intersection\A_\rejnacc = \emptyset.
  \end{align*}
\end{proof}

\subsection*{\secnumname{sec:dedcls}}\addcontentsline{toc}{subsection}{\secnumname{sec:dedcls}}
\begin{proof}[Proof of Theorem~\ref{thm:dcAs-zerocrit}]
  By definition $\A\in\dcAs$ if and only if $\phull\A_\acc\intersection\A_\rej=\emptyset$, and this expression is equivalent to the one we need to prove by Lemma~\ref{lem:minkowski-intersection-basic}.
\end{proof}
\begin{lemma}\label{lem:minkowski-intersection-basic}
  Given $\someGs,\someGs'\subseteq\linGs$, then $0\notin\someGs+\someGs'$ is equivalent to $\someGs'\intersection-\someGs=\emptyset$.
\end{lemma}
\begin{proof}
  We consider the negations:
  \begin{itemize}
    \item[$\neht$]
      Assume $0\in\someGs+\someGs'$, then there are $g$ in $\someGs$ and $f$ in $\someGs'$ such that $f+g=0$; so $f=-g\in\someGs'\intersection-\someGs$.
    \item[$\then$]
      Assume $\someGs'\intersection-\someGs\neq\emptyset$, then there is an $f$ in $\someGs'$ such that $f\in-\someGs$ or, equivalently, $-f\in\someGs$; so $0=f-f\in\someGs+\someGs'$.
    \qedhere
  \end{itemize}
\end{proof}

\begin{proof}[Proof of Proposition~\ref{prop:removing-confusion-from-deductions}]
  The two modified assessments avoid confusion by construction:
  \begin{align*}
    \Adelim{\D_\acc}{\D_\rejnacc}_\confus
      &= \D_\acc\intersection\D_\rejnacc = \emptyset, \\[1ex]
    \group[\big]{
      \ext{\Ds}\Adelim{\D_\accnrej}{\D_\rejnacc}
    }_\confus
      &= \Adelim{\phull\D_\accnrej}{\D_\rejnacc}_\confus
            &&\text{(def. $\ext{\Ds}$)} \\
      &= \phull\D_\accnrej \intersection \D_\rejnacc
      \subseteq \D_\acc\intersection \D_\rejnacc = \emptyset.
        &&\text{($\phull\D_\acc=\D_\acc$)}
  \end{align*}
  The former is deductively closed because only the set of rejected gambles has been changed and the latter is deductively closed by application of $\ext{\Ds}$.
\end{proof}

\begin{proof}[Proof of Proposition~\ref{prop:lineality}]
  We prove the claims one by one:
  \begin{itemize}
    \item[\ref{item:lineality-isq}]
      Because by definition $\D_\indiff\in\lineals$, i.e., is a linear space, it immediately follows that $0\in\D_\indiff$.
  \end{itemize}
  Because $\D_\acc\in\cones$, i.e., is a cone, $\D_\accnindiff\coloneqq\D_\acc\setminus\D_\indiff\subseteq\D_\acc$, and $\D_\indiff\subseteq\D_\acc$, we know that $\D_\accnindiff+\D_\indiff\subseteq\D_\acc$ and $\D_\accnindiff+\D_\accnindiff\subseteq\D_\acc$.
  \begin{itemize}
    \item[\ref{item:lineality-invariance}]
      Because $0\in\D_\indiff$ and $\D_\indiff\subseteq\D_\acc$, we know that $\D_\accnindiff\subseteq\D_\accnindiff+\D_\indiff\subseteq\D_\acc=\D_\accnindiff\union\D_\indiff$.
      Now assume ex absurdo that $\D_\accnindiff+\D_\indiff\nsubseteq\D_\accnindiff$, then
      $
        \group{\D_\accnindiff+\D_\indiff}\intersection\D_\indiff \neq \emptyset,
      $
      from which 
      $
        \D_\accnindiff\intersection\group{\D_\indiff-\D_\indiff} \neq \emptyset
      $
      follows by Lemma~\ref{lem:minkowski-intersection}.
      Because $\D_\indiff\in\lineals$, we have that $\D_\indiff-\D_\indiff=\D_\indiff$ and obtain the contradiction $\D_\accnindiff\intersection\D_\indiff \neq\emptyset$.
    \item[\ref{item:lineality-cone}]
      The set $\D_\accnindiff$ is a cone if it satisfies Positive Scaling~\eqref{eq:scaling} and Combination~\eqref{eq:combination}.
      \begin{itemize}
        \item[\eqref{eq:scaling}]
          Positive scaling is preserved when taking the set difference of the positively scaled sets $\D_\acc$ and $\D_\indiff$.
        \item[\eqref{eq:combination}]
          We know that $\D_\accnindiff\subseteq\D_\accnindiff+\D_\accnindiff\subseteq\D_\acc=\D_\accnindiff\union\D_\indiff$.
          Assume ex absurdo that combination does not hold, then $\D_\accnindiff+\D_\accnindiff\nsubseteq\D_\accnindiff$ or, equivalently,
          $
            \group{\D_\accnindiff+\D_\accnindiff}\intersection\D_\indiff \neq \emptyset,
          $
          from which 
          $
            \D_\accnindiff\intersection\group{\D_\indiff-\D_\accnindiff}\neq\emptyset
          $
          follows by Lemma~\ref{lem:minkowski-intersection}.
          Because of Claim~\ref{item:lineality-invariance} and because $\D_\indiff\in\lineals$, we have that
          $
            \D_\indiff-\D_\accnindiff
            = -\D_\indiff-\D_\accnindiff
            = -\D_\accnindiff
          $
          and so obtain the contradiction
          $
            \D_\accnindiff\intersection-\D_\accnindiff \neq \emptyset.
          $
          \qedhere
      \end{itemize}
  \end{itemize}
\end{proof}
\begin{lemma}\label{lem:minkowski-intersection}
  Given $\someGs,\someGs',\someGs''\subseteq\linGs$, then $(\someGs''+\someGs')\intersection\someGs=\emptyset$ is equivalent to $\someGs'\intersection(\someGs-\someGs'')=\emptyset$.
\end{lemma}
\begin{proof}[Proof of Lemma~\ref{lem:minkowski-intersection}]
  We apply Lemma~\ref{lem:minkowski-intersection-basic} twice:
  $
    (\someGs''+\someGs') \intersection \someGs = \emptyset
    \iff
    0\notin\someGs''+\someGs'-\someGs
    \iff
    \someGs' \intersection (\someGs-\someGs'') = \emptyset.
  $
\end{proof}

\subsection*{\secnumname{sec:nolimbo}}\addcontentsline{toc}{subsection}{\secnumname{sec:nolimbo}}
\begin{proof}[Proof of Proposition~\ref{prop:limbo}]
  We prove the claims one by one:
  \begin{itemize}
    \item[\ref{item:limbo-condition}]
      First,
      $
        \phull\group{\D_\acc\union\set{f}}
        = \D_\acc\union\ray{f}\union(\D_\acc+\ray{f})
      $
      because $\D_\acc\in\cones$.
      So
      \begin{equation*}
        \group[\big]{\ext{\Ds}\Adelim{\D_\acc\union\set{f}}{\D_\rej}}_\confus
        = \group[\big]{\D_\acc\union\ray{f}\union(\D_\acc+\ray{f})} 
            \intersection \D_\rej
        = \D_\confus
          \union (\ray{f}\intersection\D_\rej)
          \union \group[\big]{(\D_\acc+\ray{f})\intersection\D_\rej}
      \end{equation*}
      and we have to find the conditions on $f$ under which $\ray{f}\intersection\D_\rej\subseteq\D_\confus$ and $(\D_\acc+\ray{f})\intersection\D_\rej\subseteq\D_\confus$ or, equivalently, under which $(\ray{f}\intersection\D_\rej)\setminus\D_\confus=\emptyset$ and
      $
        \group[\big]{(\D_\acc+\ray{f})\intersection\D_\rej} \setminus \D_\confus
        = \emptyset.
      $
      Using Lemma~\ref{lem:setdifference-is-intersection} these conditions become $\ray{f}\intersection\D_\rejnacc=\emptyset$ and
      $
        (\D_\acc+\ray{f})\intersection\D_\rejnacc = \emptyset,
      $
      as $\D_\rej\setminus\D_\confus=\D_\rej\setminus\D_\acc=\D_\rejnacc$.
      By Lemma~\ref{lem:minkowski-intersection}, the latter becomes
      $
        \ray{f}\intersection(\D_\rejnacc-\D_\acc) = \emptyset.
      $
      So both conditions can be combined into
      $
        \ray{f} \intersection
        \group[\big]{\D_\rejnacc\union(\D_\rejnacc-\D_\acc)} = \emptyset
      $
      or, equivalently,
      $
        f \notin \shull{\D_\rejnacc}\union\shull{\D_\rejnacc-\D_\acc}
      $
      using the definition of rays and scalar hulls.
      Applying Lemma~\ref{lem:shull-fun-addition} to the second scalar hull in the union finishes the proof: $\shull{\D_\rejnacc-\D_\acc}=\shull{\D_\rejnacc}-\D_\acc$.
    \item[\ref{item:limbo-disjoint}]
      It is claimed that
      $
        \shull{\D_\rejnacc}\intersection \D_\acc = \emptyset
      $
      and
      $
        (\shull{\D_\rejnacc} - \D_\acc)\intersection \D_\acc = \emptyset.
      $
      By Lemma~\ref{lem:minkowski-intersection} the latter expression is equivalent to
      $
        \shull{\D_\rejnacc}\intersection (\D_\acc+\D_\acc) = \emptyset,
      $
      which effectively coincides with the former because $\D_\acc+\D_\acc=\D_\acc\in\cones$.
      This former expression is proven using Lemma~\ref{lem:shull-fun-difference}:
      $
        \shull{\D_\rejnacc}\intersection \D_\acc = \shull{\D_\rej\setminus\D_\acc}\intersection \D_\acc =
        (\shull{\D_\rej}\setminus\D_\acc)\intersection \D_\acc = \emptyset.
      $
      \qedhere
  \end{itemize}
\end{proof}
\begin{lemma}\label{lem:setdifference-is-intersection}
  Given $\someGs,\someGs',\someGs''\subseteq\linGs$, then
  $
    (\someGs'\intersection\someGs'')\setminus\someGs
    = \someGs'\intersection(\someGs''\setminus\someGs).
  $
\end{lemma}
\begin{proof}[Proof of Lemma~\ref{lem:setdifference-is-intersection}]
  The identity's left-hand side is equal to
  $
    (\someGs'\intersection\someGs'')\intersection\compl{\someGs}
    = \someGs'\intersection(\someGs''\intersection\compl{\someGs}),
  $
  which is equal to its right-hand side.
\end{proof}
\begin{lemma}\label{lem:shull-fun-addition}
  Given $\someGs,\someGs'\subseteq\linGs$ such that $\shull{\someGs}=\someGs$, then $\shull{\someGs+\someGs'}=\someGs+\shull{\someGs'}$.
\end{lemma}
 \begin{proof}[Proof of Lemma~\ref{lem:shull-fun-addition}]
  Any left-hand side element can be written as $\lambda\cdot(f+g)$, with $f\in\someGs$, $g\in\someGs'$, and $\lambda>0$.
  This can be rewritten as $\lambda\cdot f+\lambda\cdot g$.
  So because $\lambda\cdot f\in\someGs$, this implies $\shull{\someGs+\someGs'}\subseteq\someGs+\shull{\someGs'}$.
  Inclusion in the other direction follows from $f+\lambda\cdot g=\lambda\cdot(\lambda^{-1}\cdot f+g)$ and the fact that $\lambda^{-1}\cdot f\in\someGs$.
\end{proof}
\begin{lemma}\label{lem:shull-fun-difference}
  Given $\someGs,\someGs'\subseteq\linGs$ such that $\shull{\someGs}=\someGs$, then $\shull{\someGs'\setminus\someGs}=\shull{\someGs'}\setminus\someGs$.
\end{lemma}
\begin{proof}[Proof of Lemma~\ref{lem:shull-fun-difference}]
  Any left-hand side element can be written as $\lambda\cdot f$, with $f\in\someGs'$, $f\notin\someGs$, and $\lambda>0$.
  But then also $\lambda\cdot f\notin\someGs$, for otherwise $f\in\someGs$ by $\shull{\someGs}=\someGs$.
  So $\lambda\cdot f$ is also a right-hand side element.
  Conversely, for any right-hand side element $f$, so with $f\in\shull{\someGs'}$ and $f\notin\someGs$, it holds for any $\lambda>0$ that $\lambda\cdot f\in\shull{\someGs'}$ and $\lambda\cdot f\notin\someGs$ by $\shull{\someGs}=\someGs$.
  So then $f=\lambda^{-1}\cdot(\lambda\cdot f)$ is also a left-hand side element.
\end{proof}

\begin{proof}[Proof of Proposition~\ref{prop:reckoning-for-closed-noconf}]
  The expression for the reckoning extension follows from the fact that $\D_\rejnacc=\D_\rej$ for $\D$ that avoid confusion and the fact that $\D_\rej\subseteq\shull{\D_\rej}$.
  No Confusion of $\ext{\Ms}\D$ follows from Corollary~\ref{cor:limbo-noconfusion}\ref{item:cor:limbo-noconfusion:disjoint}.
\end{proof}

\begin{proof}[Proof of Proposition~\ref{prop:model+hulls}]
  We prove the claims one by one:
  \begin{itemize}
    \item[\ref{item:prop:model+hulls:shullrej}]
      Because $\M$ avoids confusion, we infer from Proposition~\ref{prop:reckoning-for-closed-noconf} that
      $
        \shull{\M_\rej}\union(\shull{\M_\rej}-\M_\acc)=\M_\rej,
      $
      whence $\shull{\M_\rej}=\M_\rej$ and
      $
        \M_\rej - \M_\acc \subseteq \M_\rej
      $.
    \item[\ref{item:prop:model+hulls:favexpr}]
      Claim~\ref{item:prop:model+hulls:shullrej} tells us
      $
        \M_\fav = \M_\acc \intersection -\M_\rej
        \supseteq \group{\M_\acc-\M_\rej} \intersection \M_\acc
      $.
      The converse inequality and thus equality follows from the fact that $f\in\M_\fav$ implies that $f\in\M_\acc$ and $-f\in\M_\rej$, and therefore by positive scaling that $f/2\in\M_\acc$ and $-f/2\in\M_\rej$, so that $f\in\M_\acc-\M_\rej$.
    \item[\ref{item:prop:model+hulls:faviscone}]
      The set $\M_\fav=\group{\M_\acc-\M_\rej} \intersection \M_\acc$ is a cone if it satisfies Positive Scaling~\eqref{eq:scaling} and Combination~\eqref{eq:combination}.
      \begin{itemize}
        \item[\eqref{eq:scaling}]
          $
            \group{\M_\acc-\M_\rej} \intersection \M_\acc
            = \shull{\group{\M_\acc-\M_\rej} \intersection \M_\acc}
          $
          because all sets involved are positively scaled and because Minkowski addition and taking intersections preserves positive scaling.
        \item[\eqref{eq:combination}]
          Let $f_1$ and $f_2$ in $\M_\acc$ and $g_1$ and $g_2$ in $\M_\rej$ be such that $\set{f_1-g_1,f_2-g_2}\subseteq\M_\acc$, then
          $
            (f_1-g_1)+(f_2-g_2) = \group[\big]{f_1+(f_2-g_2)}-g_1
            \in \M_\acc-\M_\rej
          $
          because $g_1\in\M_\rej$ and $f_1+(f_2-g_2)\in\M_\acc$, which follows from $\M_\acc$ being a convex cone.
          \qedhere
      \end{itemize}
  \end{itemize}
\end{proof}

\begin{proof}[Proof of Corollary~\ref{cor:sweetened-deals}]
  We prove $\M_\fav\subseteq\M_\fav+\M_\acc\subseteq\M_\fav$, i.e., use sandwiching:
  \begin{align*}
    \M_\fav
    &= \M_\fav + \M_\fav
      &&\text{(Proposition~\ref{prop:model+hulls}\ref{item:prop:model+hulls:faviscone}: $\M_\fav\in\cones$)} \\
    &\subseteq \M_\fav + \M_\acc &&\text{($\M_\fav\subseteq\M_\acc$)} \\
    &= \group[\big]{\group{\M_\acc-\M_\rej} \intersection \M_\acc}+\M_\acc
      &&\text{(Proposition~\ref{prop:model+hulls}\ref{item:prop:model+hulls:favexpr})} \\
    &\subseteq (\M_\acc-\M_\rej+\M_\acc) \intersection (\M_\acc+\M_\acc) \\
    &= \group{\M_\acc-\M_\rej} \intersection \M_\acc
      &&\text{($\M_\acc\in\cones$ so $\M_\acc+\M_\acc=\M_\acc$)}\\
    &= \M_\fav.
      &&\text{(Proposition~\ref{prop:model+hulls}\ref{item:prop:model+hulls:favexpr})}
  \end{align*}
\end{proof}

\subsection*{\secnumname{sec:order}}\addcontentsline{toc}{subsection}{\secnumname{sec:order}}
\begin{proof}[Proof of Proposition~\ref{prop:intersection}]
  The claim requires us to show that No Confusion~\eqref{eq:no-confusion}, Deductive Closure~\eqref{eq:deduction}, and No Limbo~\eqref{eq:no-limbo} are preserved under non-empty intersections.
  \begin{itemize}
    \item[\eqref{eq:no-confusion}]
      Consider any non-empty family $\somencAs\subseteq\ncAs$; given $\A_\confus=\emptyset$ for all $\A\in\somencAs$, then
      \begin{equation*}\textstyle
        (\Intersection\somencAs)_\confus
        = (\Intersection\somencAs)_\acc
          \intersection
          (\Intersection\somencAs)_\rej
        = (\Intersection_{\A\in\somencAs}\A_\acc)
          \intersection
          (\Intersection_{\A\in\somencAs}\A_\rej)
        = \Intersection_{\A\in\somencAs}(\A_\acc\intersection\A_\rej)
        = \Intersection_{\A\in\somencAs}\A_\confus
        = \emptyset.
      \end{equation*}
    \item[\eqref{eq:deduction}]
      Deductive Closure is preserved because arbitrary intersections of convex cones are still convex cones and a deductively closed assessment is just required to have a convex cone as a set of acceptable gambles.
    \item[\eqref{eq:no-limbo}]
      Consider any non-empty family $\somencMs\subseteq\ncMs$, so with
      $
        \shull{\M_\rej}\union(\shull{\M_\rej}-\M_\acc)
        \subseteq \M_\rej
      $
      for all $\M$ in $\somencMs$; then we have
      \begin{align*}\textstyle
        \shull{(\Intersection\somencMs)_\rej} \union
          \group[\big]{
            \shull{(\Intersection\somencMs)_\rej}
            - (\Intersection\somencMs)_\acc
          }
        &\textstyle=
          \shull{\Intersection_{\M\in\somencMs}\M_\rej} \union
          \group[\big]{
            \shull{\Intersection_{\M\in\somencMs}\M_\rej}
            - \Intersection_{\M\in\somencMs}\M_\acc
          } \\
        &\textstyle\subseteq
          (\Intersection_{\M\in\somencMs}\shull{\M_\rej}) \union
          \group[\big]{
            (\Intersection_{\M\in\somencMs}\shull{\M_\rej})
            - (\Intersection_{\M\in\somencMs}\M_\acc)
          } \\
        &\textstyle\subseteq
          \Intersection_{\M\in\somencMs}
          \group[\big]{
            \shull{\M_\rej} \union (\shull{\M_\rej}-\M_\acc)
          }
        \subseteq \Intersection_{\M\in\somencMs}\M_\rej
        = (\Intersection\somencMs)_\rej.
      \end{align*}
      \qedhere
  \end{itemize}
\end{proof}

\begin{proof}[Proof of Proposition~\ref{prop:maximalncasss}]
  Apply Lemma~\ref{lem:maximalsubmaximal} with $\somemoreAs\defeq\ncAs$,
  $
    \someAs \defeq
      \cset{\Adelim{\someGs}{\linGs\setminus\someGs}}{\someGs\subseteq\linGs}
    \subseteq \ncAs
  $%
  ---for which $\maxsomeAs=\someAs$ holds by construction---, and $\otherA\defeq\otherD\union\Adelim{\otherD_\unres}{\emptyset}$.
  This gives
  $
    \maxncAs = \maxsomemoreAs
             = \somemoreAs\intersection\maxsomeAs
             = \ncAs \intersection \cset{\Adelim{\someGs}{\linGs\setminus\someGs}}{\someGs\subseteq\linGs}
             = \cset{\Adelim{\someGs}{\linGs\setminus\someGs}}{\someGs\subseteq\linGs}.
  $
\end{proof}
\begin{lemma}\label{lem:maximalsubmaximal}
  If for all assessments~$\otherD$ in some class~$\somemoreAs\subseteq\As$ there is another assessment~$\otherA$ in $\somemoreAs$ that dominates it ($\otherD\subseteq\otherA$) and is maximal in some fixed second class $\someAs\subseteq\As$---i.e., $\otherA\in\maxsomeAs$---, then $\maxsomemoreAs=\somemoreAs\intersection\maxsomeAs$: the maximal elements of $\somemoreAs$ are maximal elements of $\someAs$.
\end{lemma}
\begin{proof}[Proof of Lemma~\ref{lem:maximalsubmaximal}]
  We prove 
  $
    \maxsomemoreAs
    \subseteq \somemoreAs\intersection\maxsomeAs \subseteq
    \maxsomemoreAs,
  $
  i.e., use sandwiching:
  \begin{description}
    \item[$\maxsomemoreAs\subseteq\somemoreAs\intersection\maxsomeAs$:]
      Consider $\otherD$ in $\maxsomemoreAs$, then there is a~$\otherA$ in~$\somemoreAs\intersection\maxsomeAs$ such that $\otherD\subseteq\otherA$ and therefore $\otherD=\otherA\in\somemoreAs\intersection\maxsomeAs$.
    \item[$\somemoreAs\intersection\maxsomeAs\subseteq\maxsomemoreAs$:]
      Consider $\A$ in $\somemoreAs\intersection\maxsomeAs$ and $\otherD$ in $\maxsomemoreAs$ such that $\A\subseteq\otherD$.
      Then there is a~$\otherA$ in~$\somemoreAs\intersection\maxsomeAs$ such that $\otherD\subseteq\otherA$ and hence $\A\subseteq\otherA$.
      Since $\A,\otherA\in\maxsomeAs$ we find $\A=\otherD=\otherA$ and therefore $\A\in\maxsomemoreAs$.
      \qedhere
  \end{description}
\end{proof}

\begin{proof}[Proof of Proposition~\ref{prop:maximalncmods}]
  First note that $\ncMs\subseteq\ncDs\subseteq\dcAs$.
  Now, because $\A\subseteq\ext{\Ds}\A\in\ncDs$ for all~$\A$ in~$\maxdcAs$, $\D\subseteq\ext{\Ms}\D\in\ncMs$ for all~$\D$ in~$\maxncDs$, we have that $\maxdcAs\subseteq\maxncDs\subseteq\maxncMs$.
  So we find that $\maxncMs=\maxncDs=\maxdcAs$.
  Next, apply Lemma~\ref{lem:maximalsubmaximal} with $\someAs\defeq\ncAs$, $\somemoreAs\defeq\ncDs$, and $\otherA\defeq\otherD\union\Adelim{\emptyset}{\otherD_\unres}$.
  This gives
  $
    \maxncDs = \maxsomemoreAs
             = \somemoreAs\intersection\maxsomeAs
             = \ncDs \intersection \maxncAs.
  $
  The form of the maximal models follows from Proposition~\ref{prop:maximalncasss} and Deductive Closure~\eqref{eq:deduction}.
\end{proof}

\begin{proof}[Proof of Proposition~\ref{prop:closure}]
  We prove the claims about $\cls{\someAs}$, defined for any assessment $\A$ in $\As$ by $\cls{\someAs}\A=\Intersection\someAs_\A$, in reverse order:
  \begin{itemize}
    \item[\ref{item:prop:closure:id}]
      If $\A\in\someAs$, then $\someAs_\A$ only includes $\A$ and assessments in $\someAs$ dominating it, so
      $
        \cls{\someAs}\A
        =\Intersection\someAs_\A
        = \A.
      $
      Also, $\cls{\someAs}\top=\top$.
      Furthermore, by the definition of an intersection structure, $\cls{\someAs}\A\in\someAs$ if $\someAs_\A\neq\emptyset$, which means $\cls{\someAs}\A\neq\A$ if $\A\notin\someAs\union\set{\top}$.
    \item[\ref{item:prop:closure:cls}]
      We need to show that $\cls{\someAs}$ satisfies the three closure operator properties:
      \begin{itemize}
        \item
          The extensive nature follows from the fact that $\someAs_\A$ only contains assessments dominating~$\A$.
        \item Idempotency is implied by Claim~\ref{item:prop:closure:id}.
        \item
          Consider an assessment $\otherA$ in $\As$ such that $\A\subseteq\otherA$, then the increasing nature follows from the fact that any assessment that dominates~$\otherA$ also dominates~$\A$, so that $\someAs_\otherA\subseteq\someAs_\A$.
          \qedhere
      \end{itemize}
  \end{itemize}
\end{proof}

\begin{proof}[Proof of Proposition~\ref{prop:cls-formulae}]
  First of all, $\ncMs\subseteq\ncDs\subseteq\dcAs\subseteq\ncAs$ implies that, point-wise, $\cls{\ncAs}\subseteq\cls{\dcAs}\subseteq\cls{\ncDs}\subseteq\cls{\ncMs}$.
  \begin{enumerate}
    \item[\ref{item:prop:cls-formulae:ncAs}]
      An assessment $\A$ in $\As\setminus\ncAs$ has no dominating assessments without confusion, so $\ncAs_\A=\emptyset$ and therefore $\cls{\ncAs}\A=\top$.
    \item[\ref{item:prop:cls-formulae:dcAs}]
      Given an assessment $\A$ in $\As\setminus\dcAs$, then for all assessments $\otherA\supset\A$ in $\As$ we have $\ext{\Ds}\otherA\supseteq\ext{\Ds}\A$, whereby there is also confusion in $\ext{\Ds}\otherA$ and hence $\otherA\notin\dcAs$.
      So $\dcAs_\A=\emptyset$ and therefore $\cls{\dcAs}\A=\top$.
      This carries over to $\ncDs$ and $\ncMs$ because, as seen, point-wise $\cls{\dcAs}\subseteq\cls{\ncDs}\subseteq\cls{\ncMs}$.
    \item[\ref{item:prop:cls-formulae:Ds}]
      By construction $\ext{\Ds}$ maps onto $\Ds$, so point-wise $\cls{\Ds}\subseteq\ext{\Ds}$.
      We actually have $\cls{\Ds}=\ext{\Ds}$ because the positive linear hull operator $\phull$ used by $\ext{\Ds}$ generates the \emph{smallest} convex cone encompassing its argument.
    \item[\ref{item:prop:cls-formulae:ncDs}]
      By definition $\ext{\Ds}$ maps $\dcAs$ onto $\ncDs$; so we have $\cls{\ncDs}\subseteq\ext{\Ds}$.
      Then $\cls{\ncDs}=\ext{\Ds}$ follows by the same argument as above.
    \item[\ref{item:prop:cls-formulae:ncMs}]
      Because $\cls{\ncDs}\subseteq\cls{\ncMs}$, we have $\cls{\ncMs}\after\cls{\ncDs}=\cls{\ncMs}$ by idempotency.
      So $\cls{\ncMs}=\cls{\ncMs}\after\ext{\Ds}$ on $\dcAs$ and we need to show that $\cls{\ncMs}=\ext{\Ms}$ on $\ncDs$.
      By Proposition~\ref{prop:reckoning-for-closed-noconf} we have that $\cls{\ncMs}\subseteq\ext{\Ms}$.
      We actually have $\cls{\ncMs}=\ext{\Ms}$ because $\ext{\Ms}$ only rejects the gambles in limbo, and by definition of a model all these have to be rejected.
      \qedhere
  \end{enumerate}
\end{proof}

\subsection*{\secnumname{sec:dominating-models}}\addcontentsline{toc}{subsection}{\secnumname{sec:dominating-models}}
\begin{proof}[Proof of Theorem~\ref{thm:dedclosable-nonemptymax}]
  We have already observed that $\maxncMs_\A=\emptyset$ if $\A\notin\dcAs$.
  We furthermore know from Proposition~\ref{prop:maximalncmods} that $\maxncDs_\A=\maxncMs_\A$.
  Assume that $\A\in\dcAs$, then $\D\defeq\ext{\Ds}\A$ is a deductively closed assessment and therefore $\ncDs_\D=\ncDs_\A\neq\emptyset$.
  So we need to prove that the poset $(\ncDs_\D,\subseteq)$ has maximal elements; $\D\union\Adelim{\emptyset}{\D_\unres}$ is one (cf. Proposition~\ref{prop:maximalncmods}).
\end{proof}

\begin{proof}[Proof of Proposition~\ref{prop:inf-by-max}]
  Let $\M\defeq\cls{\ncMs}\A$; we then know that $\maxncMs_\M=\maxncMs_\A$.
  So it is sufficient to prove for all $\M\in\ncMs$ that $\M=\Intersection\maxncMs_\M$.
  We prove both directions of the equality separately:
  \begin{itemize}
    \item[$\subseteq$]
      The definition of $\maxncMs_\M$ tells us that $\M\subseteq\Intersection\maxncMs_\M$.
    \item[$\supseteq$]
      Assume ex absurdo that $\M\nsupseteq\Intersection\maxncMs_\M$, so $\M_\acc\nsupseteq\group{\Intersection\aff{\maxncMs}_\M}_\acc$ or $\M_\rej\nsupseteq\group{\Intersection\aff{\maxncMs}_\M}_\rej$; then there is some~$f$ in~$\M_\unres$ such that $f\in(\Intersection\ncMs_\M)_\acc$ or $f\in(\Intersection\ncMs_\M)_\rej$.
      In the former case, let $\otherM\defeq\M\reckunion\Adelim{\emptyset}{\set{f}}$, in the latter case, let $\otherM\defeq\M\reckunion\Adelim{\set{f}}{\emptyset}$.
      Using Lemma~\ref{lem:ext-with-unres}, we know that $\otherM\in\ncMs$.
      Then $\M\subseteq\otherM$, but $\otherM$ is not dominated by any element of~$\maxncMs_\M$, a contradiction.
      \qedhere
  \end{itemize}
\end{proof}
\begin{lemma}\label{lem:ext-with-unres}
  Given~$\M$ in~$\ncMs$ and a gamble~$f$ in~$\M_\unres$, then
  $
    \set[\big]{\M\reckunion\Adelim{\emptyset}{\set{f}},
               \M\reckunion\Adelim{\set{f}}{\emptyset}} \subseteq \ncMs.
  $
\end{lemma}
\begin{proof}[Proof of Lemma~\ref{lem:ext-with-unres}]
  We just need to show that
  $
    \set[\big]{\M\dedunion\Adelim{\emptyset}{\set{f}},
               \M\dedunion\Adelim{\set{f}}{\emptyset}} \subseteq \ncDs
  $
  thanks to Proposition~\ref{prop:cls-formulae}\ref{item:prop:cls-formulae:ncMs}, i.e., because unconfused deductively closed assessments can always be extended to unconfused models.
  Now, $\M\union\Adelim{\emptyset}{\set{f}}$ already belongs to $\Ds$ because deductive closure only acts on the acceptable gambles of an assessment; it avoids confusion by choice of $f$.
  For $\M\dedunion\Adelim{\set{f}}{\emptyset}$ we can apply Corollary~\ref{cor:limbo-noconfusion}\ref{item:cor:limbo-noconfusion:condition}, for which the condition is trivially satisfied because $\M$, being a model, has no limbo.
\end{proof}

\subsection*{\secnumname{sec:background}}\addcontentsline{toc}{subsection}{\secnumname{sec:background}}
\begin{proof}[Proof of Theorem~\ref{thm:char-AR}]
  \ref{item:ARbg} is equivalent to $\M\in\As_\bgM$.
  \ref{item:ARcn} is equivalent to $\M\in\Ds$.
  At this point we know $\M\in\Ds_\bgM$ and therefore $0\in\D_\acc\in\cones$, whereby using Theorem~\ref{thm:dcAs-zerocrit} the condition for No Confusion~\eqref{eq:no-confusion} becomes 
  $
    0 \notin \M_\rej-\M_\acc \subseteq \shull{\M_\rej}-\M_\acc.
  $
  \ref{item:ARss} makes this condition implied by $0\notin\M_\rej$, i.e., \ref{item:ARzr}, so $\M\in\ncAs$.
  Now we know $\M\in\ncDs_\bgM\subseteq\ncDs_\zeroM$, whereby using Corollary~\ref{cor:limbo-noconfusion+statusquo} the condition for No Confusion~\eqref{eq:no-confusion} becomes $\shull{\M_\rej}-\M_\acc\subseteq\M_\rej$, i.e., \ref{item:ARss}, completing the proof, as then $\M\in\ncDs_\bgM\intersection\Ms=\ncMs_\bgM$.
\end{proof}

\subsection*{\secnumname{sec:relations}}\addcontentsline{toc}{subsection}{\secnumname{sec:relations}}
\begin{proof}[Proof of Theorem~\ref{thm:char-AD}]
  We have to show equivalence under Definition~\eqref{eq:gamble-rel-def} of \ref{item:AD-accrefl}--\ref{item:AD-rejmixindep} and \ref{item:ARbg}--\ref{item:ARss} of Theorem~\ref{thm:char-AR} with $\bgM\defeq\zeroM$.
  \begin{itemize}
    \item[\ref{item:AD-accrefl}]
      is \eqref{eq:gamble-rel-def}-equivalent to $0\in\M_\acc$, which in turn is equivalent to $\zeroM\subseteq\M_\acc$, i.e., \ref{item:ARbg}.
    \item[\ref{item:AD-rejirrefl}]
      is \eqref{eq:gamble-rel-def}-equivalent to $0\notin\M_\rej$, i.e., \ref{item:ARzr}.
    \item[\ref{item:AD-acctrans}]
      is \eqref{eq:gamble-rel-def}-equivalent to $f-g\in\M_\acc\conj g-h\in\M_\acc\then f-h\in\M_\acc$, or, using gambles $f'\defeq f-g$ and $h'\defeq g-h$, to $f'\in\M_\acc\conj h'\in\M_\acc\then f'+h'\in\M_\acc$; because this must effectively hold for all $f'$ and $g'$ in $\linGs$, we can rewrite it as $\M_\acc+\M_\acc\subseteq\M_\acc$, Combination~\eqref{eq:combination}.
    \item[\ref{item:AD-accmixindep}]
      is \eqref{eq:gamble-rel-def}-equivalent to Positive Scaling~\eqref{eq:scaling}:
      \begin{itemize}
        \item[$\then$]
          \ref{item:AD-accmixindep} \eqref{eq:gamble-rel-def}-implies $f-g\in\M_\acc\iff\mu\cdot(f-g)\in\M_\acc$, or, using gambles $f'\defeq f-g$ and $f'\defeq (f-g)/\mu$, $f'\in\M_\acc\iff\lambda\cdot f'\in\M_\acc$ for $\lambda\in\set{\mu,1/\mu}$, and therefore \eqref{eq:scaling};
        \item[$\neht$]
          \eqref{eq:scaling} \eqref{eq:gamble-rel-def}-implies \ref{item:AD-accmixindep} because $\mu\cdot(f-g)=\group[\big]{\mu\cdot f+(1-\mu)\cdot h}-\group[\big]{\mu\cdot g+(1-\mu)\cdot h}$.
      \end{itemize}
  \end{itemize}
  \noindent
  Together, Combination~\eqref{eq:combination} and Positive Scaling~\eqref{eq:scaling} are equivalent to \ref{item:ARcn}, and therefore so are \ref{item:AD-acctrans} and \ref{item:AD-accmixindep}.
  \begin{itemize}
    \item[\ref{item:AD-rejmixindep}]
      is \eqref{eq:gamble-rel-def}-equivalent to $\shull{\M_\rej}\subseteq\M_\rej$ by the same reasoning used to show equivalence of \ref{item:AD-accmixindep} and Scaling~\eqref{eq:scaling}.
    \item[\ref{item:AD-mixtrans}]
      is \eqref{eq:gamble-rel-def}-equivalent to $f-g\in\M_\rej\conj h-g\in\M_\acc\then f-h\in\M_\rej$, or, using gambles $f'\defeq f-g$ and $h'\defeq h-g$, to $f'\in\M_\rej\conj h'\in\M_\acc\then f'-h'\in\M_\rej$; because this must effectively hold for all $f'$ and $g'$ in $\linGs$, we can rewrite it as $\M_\rej-\M_\acc\subseteq\M_\rej$.
  \end{itemize}
  \noindent
  Together, $\shull{\M_\rej}\subseteq\M_\rej$ and $\M_\rej-\M_\acc\subseteq\M_\rej$ are equivalent to $\shull{\M_\rej}-\M_\acc\subseteq\M_\rej$, i.e., \ref{item:ARss}.
\end{proof}

\begin{proof}[Proof of Proposition~\ref{prop:char-ADbgM}]
  By definition $\M$ is $\bgM$-coherent if it coincides with its natural extension.
  Because $\M\in\ncMs_\zeroM$, this means we should have $\bgM\subseteq\M$, i.e., \ref{item:ARbg}.
  So for all $f$ and $g$ in~$\linGs$ we have $f\sacc g \iff f-g\in\bgM_\sacc\subseteq\M_\acc$ and thus $f\acc g$.
  Similarly, for all $f$ and $g$ in~$\linGs$ we have $f\srej g \iff f-g\in\bgM_\srej\subseteq\M_\rej$ and thus $f\rej g$.
\end{proof}

\subsection*{\secnumname{sec:accept-favour}}\addcontentsline{toc}{subsection}{\secnumname{sec:accept-favour}}
\begin{proof}[Proof of Theorem~\ref{thm:affdcAs-zerocrit}]
  The result follows from Theorem~\ref{thm:dcAs-zerocrit} because Condition~\eqref{eq:accfav-condition} implies $-\A_\rej=\A_\fav$.
\end{proof}

\begin{proof}[Proof of Proposition~\ref{prop:accfav-intersection}]
  We need to show that Condition~\eqref{eq:accfav-condition} is preserved under arbitrary non-empty intersections:
  For any non-empty family $\somemoreAs\subseteq\aff{\someAs}$, we have that
  \begin{equation*}\textstyle
    -(\Intersection\somemoreAs)_\rej
    =\Intersection_{\A\in\somemoreAs}-\A_\rej
    \subseteq\Intersection_{\A\in\somemoreAs}\A_\acc
    = (\Intersection\somemoreAs)_\acc,
  \end{equation*}
  so $\Intersection\somemoreAs$ satisfies Condition~\eqref{eq:accfav-condition}.
\end{proof}

\begin{proof}[Proof of Proposition~\ref{prop:accfav-maximalncasss}]
  For the first claim, apply Lemma~\ref{lem:maximalsubmaximal} with $\someAs\defeq\ncAs$, $\somemoreAs\defeq\aff\ncAs$, and $\otherA\defeq\otherD\union\Adelim{\otherD_\unres}{\emptyset}$ to prove the first equality.
  The second equality follows from Proposition~\ref{prop:maximalncasss} and Condition~\eqref{eq:accfav-condition}.
  For the second claim, by Proposition~\ref{prop:maximalncasss} either $0\in\A_\acc$ or $0\in\A_\rej$.
  To preserve No Confusion under Condition~\eqref{eq:accfav-condition} the latter is disallowed.
\end{proof}

\begin{proof}[Proof of Proposition~\ref{prop:accfav-maximalncmods}]
  First note that $\aff\ncMs\subseteq\aff\ncDs\subseteq\affdcAs$.
  Now, because $\A\subseteq\ext{\Ds}\A\in\aff\ncDs$ for all~$\A$ in~$\aff\maxdcAs$ and $\D\subseteq\ext{\Ms}\D\in\aff\ncMs$ for all~$\D$ in~$\aff\maxncDs$ by Lemma~\ref{lem:exts-pres-aff}, we have that $\aff\maxncMs\subseteq\aff\maxncDs\subseteq\aff\maxdcAs$.
  So we find that $\aff\maxncMs=\aff\maxncDs=\aff\maxdcAs$.

  For the expression of $\aff\maxncMs$, the equality $\aff\maxncAs\intersection\aff\ncDs=\aff\maxncDs$ follows from Lemma~\ref{lem:maximalsubmaximal} with $\someAs\defeq\aff\ncAs$, $\somemoreAs\defeq\aff\ncDs$, and
  $
    \otherA \defeq {\Adelim{\someGs}{\linGs\setminus\someGs}}
            \supseteq \Adelim{\otherD_\acc}{\phull\otherD_\rej}
  $,
  with $\someGs$ in~$\cones$ such that $\otherA\in\aff\maxncAs\intersection\aff\ncDs$ indeed and whose existence is guaranteed by Lemma~\ref{lem:kakutani-for-cones}, taking
  $
    \Adelim{\someGs'}{\someGs''}
    \defeq \Adelim{\otherD_\acc}{\phull\otherD_\rej}.
  $
  (The Axiom of Choice is assumed for infinite-dimensional~$\linGs$.)
  The necessity of the given form of the maximal elements follows from Propositions~\ref{prop:maximalncmods} and~\ref{prop:accfav-maximalncasss}; its sufficiency follows from the existence of appropriate cones~$\someGs$ guaranteed by Lemma~\ref{lem:kakutani-for-cones}.
\end{proof}
\begin{lemma}\label{lem:exts-pres-aff}
  $\ext{\Ds}$ and $\ext{\Ms}$ preserve Condition~\eqref{eq:accfav-condition}.
\end{lemma}
\begin{proof}[Proof of Lemma~\ref{lem:exts-pres-aff}]
  We give a proof for each extension operator separately:
  \begin{itemize}
    \item[$\ext{\Ds}$:]
      Given $\A$ in~$\aff{\As}$, so with $-\A_\rej \subseteq \A_\acc$, then
      $
        -(\ext{\Ds}\A)_\rej
        = -\A_\rej \subseteq \A_\acc \subseteq \phull\A_\acc
        = (\ext{\Ds}\A)_\acc.
      $
    \item[$\ext{\Ms}$:]
      Given $\D$ in~$\aff{\Ds}$, so with 
      $
        -\shull{\D_\rejnacc} \subseteq -\shull{\D_\rej}
        \subseteq \shull{\D_\acc} \subseteq \D_\acc,
      $
      then
      \begin{equation*}
        -(\ext{\Ms}\D)_\rej
        = -\group[\big]{
             \shull{\D_\rejnacc}\union(\shull{\D_\rejnacc}-\D_\acc)
           }
        \subseteq \D_\acc\union(\D_\acc+\D_\acc)
        \subseteq \D_\acc
        = (\ext{\Ms}\D)_\acc,
      \end{equation*}
      where we also used $\D_\acc\in\cones$, i.e., that
      $
        \shull{\D_\acc} \subseteq \D_\acc
        \conj
        \D_\acc+\D_\acc \subseteq \D_\acc.
      $
      \qedhere
  \end{itemize}
\end{proof}
\begin{lemma}\label{lem:kakutani-for-cones}
  For all cones $\someGs'$ and $\someGs''$ in~$\cones$ such that $\someGs'\intersection\someGs''=\emptyset$ and $0\notin\someGs''$ there is a cone~$\someGs$ in~$\cones$ such that $\someGs'\subseteq\someGs$ and $\someGs''\subseteq\linGs\setminus\someGs\subseteq-\someGs$.
  (The Axiom of Choice is assumed for infinite-dimensional~$\linGs$.)
\end{lemma}
\begin{proof}[Proof of Lemma~\ref{lem:kakutani-for-cones}]
  This directly follows from the Kakutani separation property as proven by \citet[Corollary~2]{Hammer-1955}, where $\linGs\setminus\someGs\subseteq-\someGs$ is a consequence of $\linGs\setminus\someGs$ being an intersection of \emph{semi-spaces}---i.e., maximal blunt convex cones.
\end{proof}

\begin{proof}[Proof of Proposition~\ref{prop:accfav-cls}]
  Propositions~\ref{prop:accfav-intersection} and~\ref{prop:closure} imply $\cls{\aff\someAs}$ is a closure operator.
  Now consider Proposition~\ref{prop:cls-formulae}: counterparts for~\ref{item:prop:cls-formulae:ncAs} and~\ref{item:prop:cls-formulae:dcAs} follow because $\top$ satisfies Condition~\eqref{eq:accfav-condition}; counterparts for~\ref{item:prop:cls-formulae:Ds}--\ref{item:prop:cls-formulae:ncMs} follow by Lemma~\ref{lem:exts-pres-aff}.
\end{proof}

\begin{proof}[Proof of Theorem~\ref{thm:accfav-dedclosable-nonemptymax}]
  We prove both sides of the equivalence separately:
  \begin{itemize}
    \item[$\neht$]
      If $\A\notin\aff{\dcAs}$, then $\A\notin\dcAs$ and therefore $\maxncMs_\A=\emptyset$ by Theorem~\ref{thm:dedclosable-nonemptymax}.
      Because $\aff\maxncMs\subseteq\maxncMs$, this implies that $\aff{\maxncMs}_\A=\emptyset$.
    \item[$\then$]
      First, we infer from Proposition~\ref{prop:accfav-maximalncmods} that $\aff{\maxncDs}_\A=\aff{\maxncMs}_\A$.
      Now, assume that $\A\in\aff{\dcAs}$, then we infer from Proposition~\ref{prop:accfav-cls} that $\D\defeq\cls{\ncDs}\A=\cls{\aff{\ncDs}}\A\in\aff{\ncDs}$, and therefore $\aff{\ncDs}_\D=\aff{\ncDs}_\A\neq\emptyset$.
      So we need to show that the poset $(\aff{\ncDs}_\D,\subseteq)$ has maximal elements: Proposition~\ref{prop:accfav-maximalncmods} gives their form and applying Lemma~\ref{lem:kakutani-for-cones} with
      $
        \Adelim{\someGs'}{\someGs''} \defeq \Adelim{\D_\acc}{\phull\D_\rej}
      $
      proves their existence.
      \qedhere
  \end{itemize}
\end{proof}

\begin{proof}[Proof of Proposition~\ref{prop:accfav-inf-by-max}]
  The first equality follows from Proposition~\ref{prop:accfav-cls}.
  Let $\M\defeq\cls{\ncMs}\A$; we then know that $\M\in\aff{\ncMs}_\zeroM$ and $\aff{\maxncMs}_\M=\aff{\maxncMs}_\A$.
  So it is sufficient to prove for all $\M\in\aff{\ncMs}_\zeroM$ that $\M=\Intersection\aff{\maxncMs}_\M$.
  We prove both directions of this equality separately:
  \begin{itemize}
    \item[$\subseteq$]
      The definition of $\aff{\maxncMs}_\M$ tells us that $\M\subseteq\Intersection\aff{\maxncMs}_\M$.
    \item[$\supseteq$]
      Assume ex absurdo that $\M\nsupseteq\Intersection\aff{\maxncMs}_\M$, so $\M_\acc\nsupseteq\group{\Intersection\aff{\maxncMs}_\M}_\acc$ or $\M_\fav\nsupseteq\group{\Intersection\aff{\maxncMs}_\M}_\fav$; then there is some~$f$ in~$\M_\unres$ such that either $f\in(\Intersection\aff{\ncMs}_\M)_\acc$ or $-f\in(\Intersection\aff{\ncMs}_\M)_\fav$.
      In the former case, let $\otherM\defeq\M\reckunion\Adelim{\set{-f}}{\set{f}}$, in the latter case, let $\otherM\defeq\M\reckunion\Adelim{\set{f}}{\emptyset}$.
      Using Lemma~\ref{lem:accfav-ext-with-unres}, we know that $\otherM\in\aff{\ncMs}_\zeroM$.
      Then $\M\subseteq\otherM$, but $\otherM$ is not dominated by any element of~$\aff{\maxncMs}_\M$, a contradiction.
      \qedhere
  \end{itemize}
\end{proof}
\begin{lemma}\label{lem:accfav-ext-with-unres}
  Given~$\M$ in~$\aff{\ncMs}_\zeroM$ and $f$ in~$\M_\unres$, then $\set{\M\reckunion\Adelim{\set{f}}{\emptyset},\M\reckunion\Adelim{\set{-f}}{\set{f}}}\subseteq\aff{\ncMs}_\zeroM$.
\end{lemma}
\begin{proof}[Proof of Lemma~\ref{lem:accfav-ext-with-unres}]
  Thanks to Propositions~\ref{prop:cls-formulae} and~\ref{prop:accfav-cls}, it suffices to check that $\M\reckunion\Adelim{\set{f}}{\emptyset}$ and $\M\reckunion\Adelim{\set{-f}}{\set{f}}$ satisfy Condition~\eqref{eq:accfav-condition}---which follows from Lemma~\ref{lem:exts-pres-aff}---and avoid confusion---which for $\M\reckunion\Adelim{\set{f}}{\emptyset}$ follows from Lemma~\ref{lem:ext-with-unres}.
  So we are finished once we verify that $\M\reckunion\Adelim{\set{-f}}{\set{f}}$ avoids confusion.
  We just need to show that $\M\dedunion\Adelim{\set{-f}}{\set{f}}$ avoids confusion thanks to Propositions~\ref{prop:cls-formulae}\ref{item:prop:cls-formulae:ncMs} and~\ref{prop:accfav-cls}, i.e., because unconfused deductively closed accept-favour assessments can always be extended to unconfused accept-favour models.
  Its set of confusing gambles is $(\M_\rej\union\set{f})\intersection\phull(\M_\acc\union\set{-f})$, of which the second factor is equal to $\M_\acc\union-\ray{f}\union\group{\M_\acc-\ray{f}}$, so by distributivity, we must check that six intersections are empty:
  \begin{enumerate}
    \item $\M_\rej\intersection\M_\acc=\emptyset$ because~$\M$ avoids confusion,
    \item\label{item2:proof:lem:accfav-ext-with-unres}
      $\M_\rej\intersection-\ray{f}=\emptyset$ because $f\notin\M_\acc\supseteq-\M_\rej$,
    \item
      $\M_\rej\intersection\group{\M_\acc-\ray{f}}=\emptyset$ is by Lemma~\ref{lem:minkowski-intersection} equivalent to $(\M_\rej-\M_\acc)\intersection-\ray{f}=\emptyset$, which reduces to~\ref{item2:proof:lem:accfav-ext-with-unres} by Proposition~\ref{prop:model+hulls}\ref{item:prop:model+hulls:shullrej}.
    \item\label{item4:proof:lem:accfav-ext-with-unres}
      $\set{f}\intersection\M_\acc=\emptyset$ because $f\notin\M_\acc$,
    \item
      $\set{f}\intersection-\ray{f}=\emptyset$ because $f\neq0$ for $f$ in~$\M_\unres$,
    \item
      $\set{f}\intersection\group{\M_\acc-\ray{f}}=\emptyset$ is by Lemma~\ref{lem:minkowski-intersection} equivalent to $(\set{f}+\ray{f})\intersection\M_\acc=\emptyset$, which reduces to~\ref{item4:proof:lem:accfav-ext-with-unres}.
      \qedhere
  \end{enumerate}
\end{proof}

\begin{proof}[Proof of Theorem~\ref{thm:char-AF}]
  Essentially, the conditions of Theorem~\ref{thm:char-AR} are repeated.
  But we know that $\M_\fav\in\cones$ by Proposition~\ref{prop:model+hulls}\ref{item:prop:model+hulls:faviscone}, so we can add that to \ref{item:ARcn} to form \ref{item:AFcn}.
  Also, because $\M_\fav=-\M_\rej$ by $\M\in\aff\As$, we can modify \ref{item:ARzr} to $0\notin\M_\fav$, i.e., \ref{item:AFzr}, and \ref{item:ARss} to $\M_\acc+\shull{\M_\fav}\subseteq\M_\fav$, which is equivalent to \ref{item:AFss}, as $\M_\fav\in\cones$ implies $\shull{\M_\fav}=\M_\fav$.
\end{proof}

\subsection*{\secnumname{sec:favindiff}}\addcontentsline{toc}{subsection}{\secnumname{sec:favindiff}}
\begin{proof}[Proof of Theorem~\ref{thm:fifdcAs-zerocrit}]
  Starting from Theorem~\ref{thm:affdcAs-zerocrit}, we can write the condition as ${0\notin\A_\fav+\phull\A_\acc}$, using Condition~\eqref{eq:accfav-condition}.
  We split the proof in two cases:
  \begin{description}
    \item[$\A_\indiff=\emptyset$:]
      The condition's right-hand side is equal to $\A_\fav+\phull\A_\fav=\phull\A_\fav=\phull\A_\fav+\lhull\A_\indiff$.
    \item[$\A_\indiff\neq\emptyset$:]
      We infer from Lemma~\ref{lem:favindiff-accphull} that  the right-hand side is now equal to 
      $
        \group{\A_\fav+\lhull\A_\indiff}
        \union \group{\A_\fav+\phull\A_\fav+\lhull\A_\indiff},
      $
      and the condition is again equivalent to the one stated because $\A_\fav+\phull\A_\fav=\phull\A_\fav$.
      \qedhere
  \end{description}
\end{proof}
\begin{lemma}\label{lem:favindiff-accphull}
  Given~$\A$ in~$\fif\As$ with $\A_\indiff\neq\emptyset$, then $\phull\A_\acc=\lhull\A_\indiff \union\group{\phull\A_\fav+\lhull\A_\indiff}$.
\end{lemma}
\begin{proof}[Proof of Lemma~\ref{lem:favindiff-accphull}]
  We calculate $\phull\A_\acc$ explicitly:
  \begin{align*}
    \phull\A_\acc &= \phull(\A_\fav\union\A_\indiff)
                              &&\text{(Condition~\eqref{eq:favindiff-condition})} \\
                         &= \phull\A_\fav \union \phull\A_\indiff
                                \union \group{\phull\A_\fav+\phull\A_\indiff} \\
                         &= \phull\A_\fav \union \lhull\A_\indiff
                                \union \group{\phull\A_\fav+\lhull\A_\indiff} 
                               &&\text{($\phull\A_\indiff=\lhull\A_\indiff$)} \\
                         &= \lhull\A_\indiff \union \group{\phull\A_\fav+\lhull\A_\indiff}.
                               &&\text{($\A_\indiff\neq\emptyset$, so $0\in\lhull\A_\indiff$)}
  \end{align*}
\end{proof}

\begin{proof}[Proof of Theorem~\ref{thm:favindiff}]
  Note that $0\in\M_\indiff\neq\emptyset$.
  We are going to calculate $\M$ explicitly and obtain an expression that functions as a proof for all claims:
  \begin{align*}
    \M = \cls{\ncMs}\A &= \ext{\Ms}\group{\ext{\Ds}\Adelim{\A_\acc}{-\A_\fav}}
        &&\text{(Proposition~\ref{prop:cls-formulae})}\\
       &= \ext{\Ms}\Adelim{\phull\A_\acc}{-\A_\fav}
        &&\text{(def. $\ext{\Ds}$)}\\
       &= \Adelim[\big]{\phull\A_\acc}
                       {
                          -\group[\big]{
                             \shull{\A_\fav}\union(\shull{\A_\fav}+\phull\A_\acc)
                           }
                       }
        &&\text{(Proposition~\ref{prop:reckoning-for-closed-noconf})}\\
       &= \Adelim[\big]{\lhull\A_\indiff\union\group{\phull\A_\fav+\lhull\A_\indiff}}
            {-\group{\phull\A_\fav+\lhull\A_\indiff}}.
        &&\text{(Lemma~\ref{lem:favindiff-accphull})}
  \end{align*}
\end{proof}

\begin{proof}[Proof of Theorem~\ref{thm:char-FI}]
  To the conditions of Theorem~\ref{thm:char-AF}, Condition~\eqref{eq:favindiff-condition} must be added; this is done by applying Theorem~\ref{thm:favindiff} with $\A\defeq\M$, which leads to the changes from \ref{item:AFcn}~and~\ref{item:AFss} to \ref{item:FIcn}~and~\ref{item:FIss}.
\end{proof}

\subsection*{\secnumname{sec:favourability}}\addcontentsline{toc}{subsection}{\secnumname{sec:favourability}}
\begin{proof}[Proof of Theorem~\ref{thm:ffdcAs-zerocrit}]
  Apply Theorems~\ref{thm:fifdcAs-zerocrit} and~\ref{thm:favindiff} to the assessment
  $
    \A\union\bgM = \Adelim[\big]{
      \A_\fav\union\bgM_\fav\union\bgM_\indiff
    }{
      -(\A_\fav\union\bgM_\fav)
    }.
  $
  The claims then follow by realising that $\lhull\bgM_\indiff=\bgM_\indiff\neq\emptyset$ and, for the final equality, using the fact that Conditions~\eqref{eq:accfav-condition} and~\eqref{eq:favindiff-condition} hold, as $\M\in\fif\ncMs_\zeroM$.
\end{proof}

\begin{proof}[Proof of Theorem~\ref{thm:char-F}]
  We start from Theorem~\ref{thm:char-FI}; \ref{item:Fzr} and \ref{item:Fcn} are identical to \ref{item:FIzr} and \ref{item:FIcn}.
  Because $\M_\indiff=\bgM_\indiff$ by construction, \ref{item:Fss} is equivalent to \ref{item:FIss}.
  Furthermore, if $\bgM_\fav\subseteq\M_\fav$, then $\bgM=\Adelim{\bgM_\indiff\union\bgM_\fav}{-\bgM_\fav}\subseteq\M$ and we can see that the converse also holds, so \ref{item:Fbg} is equivalent to \ref{item:FIbg}.
\end{proof}

\subsection*{\secnumname{sec:acceptability}}\addcontentsline{toc}{subsection}{\secnumname{sec:acceptability}}
\begin{proof}[Proof of Theorem~\ref{thm:afdcAs-zerocrit}]
  The first claim follows by applying Theorem~\ref{thm:dcAs-zerocrit} to the assessment $\A\union\bgM=\Adelim{\bgM_\acc\union\A_\acc}{\bgM_\rej}$.
  For the second claim, we are going to calculate $\M$ explicitly:
  \begin{align*}
    \M &= \ext{\Ms}\group{\ext{\Ds}\Adelim{\bgM_\acc\union\A_\acc}{\bgM_\rej}}
        &&\text{(def. $\cls{\ncMs}$)}\\
       &= \ext{\Ms}\Adelim{\phull(\bgM_\acc\union\A_\acc)}{\bgM_\rej}
        &&\text{(def. $\ext{\Ds}$)}\\
       &= \Adelim[\big]{\phull(\bgM_\acc\union\A_\acc)}
                       {\bgM_\rej-\phull(\bgM_\acc\union\A_\acc)}.
        &&\text{(def. $\ext{\Ms}$, Corollary~\ref{cor:limbo-noconfusion+statusquo})}
  \end{align*}
\end{proof}

\begin{proof}[Proof of Theorem~\ref{thm:char-A}]
  The changes between Theorem~\ref{thm:char-AR} and this theorem are due to the fact that $\M_\rej=\bgM_\rej-\M_\acc$:
  \ref{item:ARbg} reduces to \ref{item:Abg} as $\bgM_\rej\subseteq\M_\rej$ because $0\in\M_\acc$, \ref{item:ARzr} reduces to \ref{item:Azr} by simple substitution, and \ref{item:ARss} can be omitted because
  $
    \shull{\M_\rej}-\M_\acc = \shull{\bgM_\rej-\M_\acc}-\M_\acc
                                         = \bgM_\rej-\M_\acc
                                         = \M_\rej,
  $
  where we used the facts that $\M_\acc\in\cones$ and $\shull{\bgM_\rej}=\bgM_\rej$.
\end{proof}

\subsection*{\secnumname{sec:linprevs}}\addcontentsline{toc}{subsection}{\secnumname{sec:linprevs}}
\begin{proof}[Proof of Proposition~\ref{prop:pr-bgM-inner}]
  We prove both sides of the equality
  $
    \dotbgM = \bgM \defeq \Intersection_{\pr\in\prs}\otherM_\pr
  $
  separately:
  \begin{itemize}
    \item[$\subseteq$]
      By definition of $\linGs_{\gtrrvf{\pr}}$ and $\linGs_{\lessrvf{\pr}}$, we know that $\dotbgM\subseteq\Adelim{\linGs_{\gtrrvf{\pr}}}{\linGs_{\lessrvf{\pr}}}=\otherM_\pr$ for any linear prevision~$\pr$ in~$\prs$.
      So we have that $\dotbgM \subseteq \bgM$.
    \item[$\supseteq$]
      Assume, ex absurdo, that equality does not hold and $\dotbgM \nsupseteq \bgM$.
      Then there is some gamble~$f$ in $\bgM_\acc\setminus\linGs_\gtrdot$ or $\bgM_\rej\setminus\linGs_\lessdot$.
      But for any such $f\notin\linGs_\gtrdot\union\linGs_\lessdot$ there also exists a linear prevision~$\pr$ in~$\prs$ such that $\pr{f}=0$---i.e., $f\in\linGs_{\eqrvf{\pr}}$---\citep[\S3.4.2]{Walley-1991}, whereby $f\notin\linGs_{\gtrrvf{\pr}}\union\linGs_{\lessrvf{\pr}}$, contradicting the assumption.
      \qedhere
  \end{itemize}
\end{proof}

\begin{proof}[Proof of Proposition~\ref{prop:pr-models-disjoint}]
  We need to prove that
  $
    \ncMs_{\otherM_\pr} \intersection \ncMs_{\otherM_\otherpr}
    = \emptyset
  $
  if $\pr\neq\otherpr$, or, equivalently, if $\linGs_{\eqrvf{\pr}}\neq\linGs_{\eqrvf{\otherpr}}$.
  Assume, ex absurdo, that there is some model~$\M$ in the intersection, which means that $\otherM_\pr\union\otherM_\otherpr\subseteq\M$.
  Lemma~\ref{lem:pr-models-intersect} tells us we can choose some $f$ in $\linGs_{\eqrvf{\pr}}\intersection\linGs_{\lessrvf{\otherpr}}$ and $g$ in $\linGs_{\gtrrvf{\pr}}\intersection\linGs_{\eqrvf{\otherpr}}$.
  Then
  \begin{equation*}
    f+g \in (\linGs_{\eqrvf{\pr}}+\linGs_{\gtrrvf{\pr}})
                \intersection (\linGs_{\eqrvf{\otherpr}}+\linGs_{\lessrvf{\otherpr}})
           = \linGs_{\gtrrvf{\pr}}\intersection\linGs_{\lessrvf{\otherpr}}
           = (\otherM_\pr)_\acc\intersection(\otherM_\otherpr)_\rej
           \subseteq \M_\confus,
  \end{equation*}
  which means that there is confusion in~$\M$, a contradiction.
\end{proof}
\begin{lemma}\label{lem:pr-models-intersect}
  Given two non-identical coherent previsions $\pr$ and $\otherpr$ in $\prs$, i.e., such that $\pr\neq\otherpr$, then the marginal gambles corresponding to $\pr$ intersect the non-marginal gambles of $\otherpr$:
  $
    \linGs_{\eqrvf{\pr}}\intersection\linGs_{\lessrvf{\otherpr}}
    = -(\linGs_{\eqrvf{\pr}}\intersection\linGs_{\gtrrvf{\otherpr}})
    \neq \emptyset.
  $
\end{lemma}
\begin{proof}[Proof of Lemma~\ref{lem:pr-models-intersect}]
  The equality follows by linearity of $\pr$ and $\otherpr$, i.e., because $\linGs_{\eqrvf{\pr}}=-\linGs_{\eqrvf{\pr}}$ and $\linGs_{\lessrvf{\otherpr}}=-\linGs_{\gtrrvf{\otherpr}}$.
  For the inequality, take some~$f$ in~$\linGs$ such that $\pr{f}\neq\otherpr{f}$.
  Then again by linearity of $\pr$ and $\otherpr$ we have that $\pr(f-\pr{f})=-\pr(\pr{f}-f)=0$, whence $f-\pr{f}\in\linGs_{\eqrvf{\pr}}$, and $\otherpr(f-\pr{f})=-\otherpr(\pr{f}-f)\neq0$, whence $f-\pr{f}\in\linGs_{\lessrvf{\otherpr}}\union\linGs_{\gtrrvf{\otherpr}}$.
\end{proof}

\begin{proof}[Proof of Proposition~\ref{prop:pr-bgM-outer}]
  We have to prove that
  $
    \otherM_\pr\union\bgM\in\dcAs
    \iff
    \bgM\subset\Adelim{\linGs_\geq}{\linGs_\leq}
  $
  for all $\pr$ in~$\prs$; we do this separately for each direction:
  \begin{itemize}
    \item[$\neht$]
      There are two models in $\ncMs_{\dotbgM}$ beneath $\Adelim{\linGs_\geq}{\linGs_\leq}$ that are maximal:
      $
        \Adelim{\linGs_\geq}{\linGs_<}
        = \Adelim{\linGs_\geq}{\linGs_\leq}\setminus\Adelim{\emptyset}{\set{0}}
      $
      and
      $
        \Adelim{\linGs_>}{\linGs_\leq}
        = \Adelim{\linGs_\geq}{\linGs_\leq}\setminus\Adelim{\set{0}}{\emptyset}.
      $
      So we provide a proof by showing that $\D\defeq\otherM_\pr\union\Adelim{\linGs_\geq}{\linGs_\leq}\in\Ds$ with $\D_\confus=\set{0}$, i.e., $\otherM_\pr\union\bgM$ is included in a deductively closed assessment without confusion (a cone or positively scaled set with the origin removed is still a cone or positively scaled set, respectively).
      The first statement follows from
      \begin{equation*}
        \phull\D_\acc = \phull\group{\linGs_{\gtrrvf{\pr}}\union\linGs_\geq}
        = \linGs_{\gtrrvf{\pr}} \union \linGs_\geq
                                \union (\linGs_{\gtrrvf{\pr}}+\linGs_\geq)
        = \linGs_{\gtrrvf{\pr}} \union \linGs_\geq = \D_\acc,
      \end{equation*}
      where we used the fact that $\linGs_{\gtrrvf{\pr}}$ and $\linGs_\geq$ are cones.
      The second statement follows from
      \begin{align*}
        \D_\confus
        &= (\linGs_{\gtrrvf{\pr}}\union\linGs_\geq)
            \intersection (\linGs_{\lessrvf{\pr}}\union\linGs_\leq) \\
        &= (\linGs_{\gtrrvf{\pr}}\intersection\linGs_{\lessrvf{\pr}})
            \union (\linGs_{\gtrrvf{\pr}}\intersection\linGs_\leq)
            \union (\linGs_\geq\intersection\linGs_{\lessrvf{\pr}})
            \union (\linGs_\geq\intersection\linGs_\leq)
        = \emptyset \union \emptyset  \union \emptyset \union \set{0}
        = \set{0},
      \end{align*}
      because
      $
        \linGs_{\gtrrvf{\pr}}\intersection\linGs_\leq
        = -(\linGs_\geq\intersection\linGs_{\lessrvf{\pr}}) = \emptyset,
      $
      as $\pr{g}\leq\sup{g}\leq0$ for $g$ in $\linGs_\leq$ and $\pr{f}>0$ for $f$ in $\linGs_{\gtrrvf{\pr}}$.
    \item[$\then$]
      Assume, ex absurdo, that $\bgM\nsubset\Adelim{\linGs_\geq}{\linGs_\leq}$, then there is a gamble $f$ in $\linGs$ such that $f\notin\linGs_\geq\union\linGs_\leq$, so such that $\inf{f}<0<\sup{f}$.
      To create a contradiction, i.e., confusion, choose $\pr$ such that $\pr{f}<0$ in case $f\in\bgM_\acc$ and such that $\pr{f}>0$ in case $f\in\bgM_\rej$.
      Such a choice is always possible \citep[see, e.g.,][\S3.4.2]{Walley-1991}.
      \qedhere
  \end{itemize}
\end{proof}

\subsection*{\secnumname{sec:lowprevs}}\addcontentsline{toc}{subsection}{\secnumname{sec:lowprevs}}
\begin{proof}[Proof of Theorem~\ref{thm:lpr-model}]
  $\smash[b]{\M_\lpr}$ avoids confusion if and only if $\A_\lpr$ respects $\bgM$, or, formally, if and only if $\in\dcAs\A_\lpr\union\bgM$.
  For this, it is sufficient that the extension $\D\defeq\ext{\Ds}(\smash[b]{\A_\lpr}\union\bgM)$ avoids confusion.
  Its components are:
  \begin{align*}
    \D_\acc
      &= \phull\group[\big]{\group{\margs{\lpr}+\reals_>}\union\bgM_\acc} \\
      &= \phull\group{\margs{\lpr}+\reals_>} \union \bgM_\acc \union
                \group[\big]{\phull\group{\margs{\lpr}+\reals_>}+\bgM_\acc} 
          &&\text{($\bgM_\acc\in\cones$)} \\
      &= \bgM_\acc \union \group{\linGs_\gtrdot+\phull\margs{\lpr}},
          &&\text{(%
              $
                \linGs_\gtrdot = \linGs_\gtrdot+\reals_>
                \subseteq \bgM_\acc\union\set{0}+\reals_> \subseteq
                \linGs_\geq+\reals_> = \linGs_\gtrdot
              $)}\\
    \D_\rej &= \bgM_\rej.
  \end{align*}
  Then $\D$ avoids confusion if and only if
  \begin{align*}
    \emptyset
    = \D_\confus
    = \group[\big]{\bgM_\acc\union\group{\linGs_\gtrdot+\phull\margs{\lpr}}}
            \intersection \bgM_\rej
    &= \group{\linGs_\gtrdot+\phull\margs{\lpr}} \intersection \bgM_\rej
            &&\text{($\bgM_\acc\intersection\bgM_\rej=\emptyset$)} \\
    &= \phull\margs{\lpr} \intersection (\bgM_\rej-\linGs_\gtrdot)
            &&\text{(Lemma~\ref{lem:minkowski-intersection})} \\
    &= \phull\margs{\lpr}\intersection \linGs_\lessdot,
            &&\text{(%
              $
                \linGs_\lessdot = \linGs_\lessdot-\linGs_\gtrdot
                \subseteq \bgM_\rej-\linGs_\gtrdot \subseteq
                \linGs_\leq-\linGs_\gtrdot = \linGs_\lessdot
              $
            )}
  \end{align*}
  which is equivalent to the condition of the proposition.

  Now, $\M_\lpr \defeq \ext{\Ms}\D$ if $\D_\confus=\emptyset$ and has components
  \begin{align*}
    (\M_\lpr)_\acc &= \bgM_\acc \union \group{\linGs_\gtrdot+\phull\margs{\lpr}}, \\
    (\M_\lpr)_\rej
      &= \shull{\D_\rej} \union (\shull{\D_\rej}-\D_\acc)
      = \bgM_\rej \union (\bgM_\rej-\D_\acc)
      = \bgM_\rej \union \group{\bgM_\rej-\bgM_\acc}
                            \union \group{\bgM_\rej-\linGs_\gtrdot-\phull\margs{\lpr}}
      = \bgM_\rej \union \group{\linGs_\lessdot-\phull\margs{\lpr}},
  \end{align*}
  which proves the expression given.
\end{proof}

\begin{proof}[Proof of Proposition~\ref{prop:bgM-to-vac}]
  For any~$f$ in~$\linGs$, we have that
  \[
      \lpr_\bgM{f} = \sup\cset{\alpha\in\reals}{f-\alpha\in\bgM_\acc}
                   = \sup\cset{\alpha\in\reals}{f-\alpha\gtrdot0}
                   = \sup\cset{\alpha\in\reals}{\alpha\lessdot f}
                   = \inf{f},
  \]
  where the second equality follows from the fact that $\sup$ makes the distinction between $\geq$ and $\gtrdot$ moot.
\end{proof}

\begin{proof}[Proof of Proposition~\ref{prop:M-pr-M}]
  We prove both claims separately:
  \begin{itemize}
    \item[\ref{item:prop:M-pr-M:coh}]
      As $\lpr_\M$ is defined on the linear space $\linGs$, we can use the conditions detailed in Lemma~\ref{lem:clp-on-lin} to prove coherence:
      \begin{itemize}[label=\textbullet]
        \item Accepting sure gains:
          \begin{equation*}
            \lpr_\M{f} = \sup\cset{\alpha\in\reals}{f-\alpha\in\M_\acc}
                              \geq \sup\cset{\alpha\in\reals}{f-\alpha\in\bgM_\acc}
                              = \inf{f},
          \end{equation*}
          by $\bgM\subseteq\M$ and Proposition~\ref{prop:bgM-to-vac}.
        \item Positive homogeneity:
          \begin{align*}
            \lpr_\M(\lambda\cdot f)
            &= \sup\cset{\alpha\in\reals}{\lambda\cdot f-\alpha\in\M_\acc} \\
            &= \lambda\cdot\sup\cset{\beta\in\reals}{\lambda\cdot(f-\beta)\in\M_\acc}
            = \lambda\cdot\sup\cset{\beta\in\reals}{f-\beta\in\M_\acc}
            = \lambda\cdot\lpr_\M{f},
          \end{align*}
          by Positive Scaling~\eqref{eq:scaling}.
        \item Superadditivity:
          \begin{align*}
            \lpr_\M(f+g)
              &= \sup\cset{\alpha\in\reals}{f+g-\alpha\in\M_\acc} \\
              &= \sup\cset{\beta+\gamma}
                                  {f-\beta+g-\gamma\in\M_\acc \conj \beta,\gamma\in\reals} \\
              &\geq \sup\cset{\beta\in\reals}{f-\beta\in\M_\acc}
                      + \sup\cset{\gamma\in\reals}{g-\gamma\in\M_\acc}
              = \lpr_\M{f}+\lpr_\M{g},
          \end{align*}
          by Combination~\eqref{eq:combination}.
      \end{itemize}
    \item[\ref{item:prop:M-pr-M:commitloss}]
      We prove the inclusion separately for the the accept and reject components:
      \begin{itemize}
        \item[$\acc$]
          Theorem~\ref{thm:lpr-model} tells us that
          $
            (\M_{\lpr_\M})_\acc
            = \bgM_\acc\union\group[\big]{\phull\margs{\lpr_\M}+\linGs_\gtrdot}.
          $
          So because $\bgM_\acc\subseteq\M_\acc$ we need to prove that $\phull\margs{\lpr_\M}+\linGs_\gtrdot\subseteq\M_\acc$ or, as $\linGs_\gtrdot$ and $\M_\acc$ are cones, that $\margs{\lpr_\M}+\linGs_\gtrdot\subseteq\M_\acc$.
          Any left-hand side element can be written as $f-\lpr_\M{f}+h=f-\sup\cset{\alpha\in\reals}{f-\alpha\in\M_\acc}+h$, with $f\in\linGs$ and $h\in\linGs_\gtrdot$.
          Furthermore, we can always write $h=\varepsilon+h'$, with $\varepsilon\in\reals_>$ and $h'\in\linGs_\gtrdot$.
          Then the proof is complete because $f-\lpr{f}+\varepsilon\in\M_\acc$ and $\linGs_\gtrdot\subseteq\M_\acc$ so that $\M_\acc+\linGs_\gtrdot\subseteq\M_\acc$.
        \item[$\rej$] We give an explicit derivation:
          \begin{align*}
            (\M_{\lpr_\M})_\rej
            &= \bgM_\rej\union\group[\big]{\linGs_\lessdot-\phull\margs{\lpr_\M}}
              &&\text{(Theorem~\ref{thm:lpr-model})} \\
            &= \bgM_\rej
                  \union \group[\big]{\linGs_\lessdot-\linGs_\gtrdot-\phull\margs{\lpr_\M}} 
              &&\text{($\linGs_\lessdot\in\cones$, $\linGs_\gtrdot=-\linGs_\lessdot$)}\\
            &\subseteq \bgM_\rej\union\group{\bgM_\rej-\M_\acc}
              &&\text{(%
                  $\linGs_\lessdot\subseteq\bgM_\rej$,
                  $
                    \phull\margs{\lpr_\M}+\linGs_\gtrdot\subseteq(\M_{\lpr_\M})_\acc\subseteq\M_\acc
                  $
                 )} \\
            &\subseteq \M_\rej. &&\text{(No Limbo~\eqref{eq:no-limbo})}
          \end{align*}
      \end{itemize}
  \end{itemize}
\end{proof}
\begin{lemma}[{\cite[\S2.3.3]{Walley-1991}}]\label{lem:clp-on-lin}
  A lower prevision $\lpr$ defined on a linear space $\linGs$ is coherent if and only if it satisfies
  \begin{itemize}
    \item Accepting sure gains: $\lpr{f}\geq\inf{f}$,
    \item Positive homogeneity: $\lpr(\lambda\cdot f)=\lambda\cdot\lpr{f}$, and
    \item Superadditivity: $\lpr(f+g)\geq\lpr{f}+\lpr{g}$
  \end{itemize}
  for all scaling factors $\lambda$ in~$\reals_\geq$ and gambles $f$ and $g$ in $\linGs$.
\end{lemma}

\begin{proof}[Proof of Proposition~\ref{prop:pr-M-pr}]
  We prove the different claims sequentially:
  \begin{itemize}
    \item[\ref{item:lprnatex}] We give an explicit derivation:
      \begin{align*}
        \smash[b]{\lpr_{\M_\lpr}}f
          &= \sup\cset{\alpha\in\reals}{f-\alpha\in(\M_\lpr)_\acc}
            &&\text{(Equation~\eqref{eq:Mtolpr})}\\
          &= \sup \cset{\alpha\in\reals}
                        {f-\alpha \in
                          (\phull\margs{\lpr}+\linGs_\gtrdot)\union\bgM_\acc}
            &&\text{(Theorem~\ref{thm:lpr-model})}\\
          &= \max\set[\big]{
                \sup\cset{\alpha\in\reals}
                        {f-\alpha\gtrdot g \conj g\in\phull\margs{\lpr}},
                \lpr_\bgM{f}
              }
            &&\text{(def. $\sup$, Equation~\eqref{eq:Mtolpr})}\\
          &= \max\set{\textstyle\sup_{g\in\phull\margs{\lpr}}\inf(f-g), \inf{f}}
            &&\text{(def. $\sup$, Proposition~\ref{prop:bgM-to-vac})}\\
          &= \textstyle\sup_{g\in\phull\margs{\lpr}}\inf(f-g),
            &&\text{%
                ($\margs{\lpr}\neq\emptyset$, so $0\in\cls{}(\phull\margs{\lpr})$,
                 def. $\sup$)
              }
      \end{align*}
      where `$\cls{}$' denotes closure in the supremum-norm topology.
    \item[\ref{item:natexdomlpr}]
      This claim follows from Claim \ref{item:lprnatex} because $g\defeq f-\lpr{f}\in\margs{\lpr}\subseteq\phull\margs{\lpr}$ for $f$ in~$\someGs$.
    \item[\ref{item:lpreqnatex}]
      Because of Claim~\ref{item:natexdomlpr}, $\lpr_{\M_\lpr}=\lpr$ on~$\someGs$ if and only if $\forall f\in\someGs:\smash[b]{\lpr_{\M_\lpr}}f\leq\lpr{f}$, which, by Claim~\ref{item:lprnatex}, is equivalent to 
      $
        \sup_{f\in\someGs}\sup_{g\in\phull\margs{\lpr}}\inf\group[\big]{(f-\lpr{f})-g}
        \leq 0,
      $
      which in turn is equivalent to the coherence condition by definition of~$\margs{\lpr}$.
      \qedhere
  \end{itemize}
\end{proof}

\subsection*{\secnumname{sec:xch}}\addcontentsline{toc}{subsection}{\secnumname{sec:xch}}
\begin{proof}[Proof of Proposition~\ref{prop:symbgM}]
  We first need to realise that $\Adelim{\linGs_\transfos}{\emptyset}\union\Adelim{\linGs_>}{\linGs_<}\in\fif\As$, $\linGs_>\in\cones$, and $\emptyset\neq\linGs_\transfos\in\lineals$.
  Then Theorem~\ref{thm:fifdcAs-zerocrit} tells us its reckoning extension has No Confusion~\eqref{eq:no-confusion} if and only if $0\notin\linGs_>+\linGs_\transfos$, which is equivalent to the given condition.
  The expression of this extension then follows from Theorem~\ref{thm:favindiff}.
\end{proof}


\end{document}